\documentclass[11pt]{amsart}
\input epsf
\usepackage[arrow,matrix]{xy}
\usepackage{graphicx}
\usepackage{amsmath}
\usepackage{hyperref}
\usepackage{comment}
\usepackage{color}
\usepackage{mathrsfs}
\usepackage[bb=boondox]{mathalfa}
\newcommand{\blue}{\textcolor{black}}
\newcommand{\red}{\textcolor{black}}
\newcommand{\green}{\textcolor{black}}

\usepackage{amsmath, latexsym, amssymb}
\usepackage{epstopdf}
\input xypic

\numberwithin{equation}{section}
\theoremstyle{plain}
\newtheorem{lemma}{Lemma}[section]
\newtheorem{proposition}[lemma]{Proposition}
\newtheorem{theorem}[lemma]{Theorem}
\newtheorem{corollary}[lemma]{Corollary}

\theoremstyle{definition}
\newtheorem{definition}[lemma]{Definition}
\newtheorem{remark}[lemma]{Remark}
\newtheorem{example}[lemma]{Example}
\newtheorem{assumption}[lemma]{Assumption}

  \makeatletter
\def\@tocline#1#2#3#4#5#6#7{\relax
  \ifnum #1>\c@tocdepth 
  \else
    \par \addpenalty\@secpenalty\addvspace{#2}%
    \begingroup \hyphenpenalty\@M
    \@ifempty{#4}{%
      \@tempdima\csname r@tocindent\number#1\endcsname\relax
    }{%
      \@tempdima#4\relax
    }%
    \parindent\z@ \leftskip#3\relax \advance\leftskip\@tempdima\relax
    \rightskip\@pnumwidth plus4em \parfillskip-\@pnumwidth
    #5\leavevmode\hskip-\@tempdima
      \ifcase #1
       \or\or \hskip 1.5 em \or \hskip 2em \else \hskip 3em \fi%
      #6\nobreak\relax
    \hfill\hbox to\@pnumwidth{\@tocpagenum{#7}}\par
    \nobreak
    \endgroup
  \fi}
\makeatother

\begin{document}
\newcommand{\R}{{\mathbb R}}
\newcommand{\C}{{\mathbb C}}
\newcommand{\F}{{\mathbb F}}
\renewcommand{\O}{{\mathbb O}}
\newcommand{\Z}{{\mathbb Z}}
\newcommand{\N}{{\mathbb N}}
\newcommand{\Q}{{\mathbb Q}}
\renewcommand{\L}{{\mathbb L}}
\renewcommand{\H}{{\mathbb H}}
\newcommand{\D}{{\mathbb D}}
\newcommand{\G}{{\mathbb G}}
\newcommand{\W}{{\mathbb W}}

\newcommand{\dass}{da\ss~}

\newcommand{\Aa}{{\mathcal A}}
\newcommand{\Bb}{{\mathcal B}}
\newcommand{\Cc}{{\mathcal C}} 
\newcommand{\Dd}{{\mathcal D}}
\newcommand{\Ee}{{\mathcal E}}
\newcommand{\Ff}{{\mathcal F}}
\newcommand{\Gg}{{\mathcal G}} 
\newcommand{\Hh}{{\mathcal H}}
\newcommand{\Kk}{{\mathcal K}}
\newcommand{\Ii}{{\mathcal I}}
\newcommand{\Jj}{{\mathcal J}}
\newcommand{\Ll}{{\mathcal L}} 
\newcommand{\Mm}{{\mathcal M}} 
\newcommand{\Nn}{{\mathcal N}}
\newcommand{\Oo}{{\mathcal O}}
\newcommand{\Pp}{{\mathcal P}}
\newcommand{\Qq}{{\mathcal Q}}
\newcommand{\Rr}{{\mathcal R}}
\newcommand{\Ss}{{\mathcal S}}
\newcommand{\Tt}{{\mathcal T}}
\newcommand{\Uu}{{\mathcal U}}
\newcommand{\Vv}{{\mathcal V}}
\newcommand{\Ww}{{\mathcal W}}
\newcommand{\Xx}{{\mathcal X}}
\newcommand{\Yy}{{\mathcal Y}}
\newcommand{\Zz}{{\mathcal Z}}

\newcommand{\zt}{{\tilde z}}
\newcommand{\xt}{{\tilde x}}
\newcommand{\Ht}{\widetilde{H}}
\newcommand{\ut}{{\tilde u}}
\newcommand{\Mt}{{\widetilde M}}
\newcommand{\Llt}{{\widetilde{\Ll }}}
\newcommand{\yt}{{\tilde y}}
\newcommand{\vt}{{\tilde v}}
\newcommand{\Ppt}{{\widetilde{\mathcal P}}}
\newcommand{\bp }{{\bar \partial}}

\newcommand{\Remark}{{\it Remark}}
\newcommand{\Proof}{{\it Proof}}
\newcommand{\ad}{{\rm ad}}
\newcommand{\Om}{{\Omega}}
\newcommand{\om}{{\omega}}
\newcommand{\eps}{{\varepsilon}}
\newcommand{\Di}{{\rm Diff}}
\newcommand{\vol}{{\rm vol}}
\newcommand{\im}{{\rm Im }}

\renewcommand{\a}{{\mathfrak a}}
\renewcommand{\b}{{\mathfrak b}}
\newcommand{\e}{{\mathfrak e}}
\renewcommand{\k}{{\mathfrak k}}
\newcommand{\pg}{{\mathfrak p}}
\newcommand{\g}{{\mathfrak g}}
\newcommand{\gl}{{\mathfrak gl}}
\newcommand{\h}{{\mathfrak h}}
\renewcommand{\l}{{\mathfrak l}}
\newcommand{\sm}{{\mathfrak m}}
\newcommand{\n}{{\mathfrak n}}
\newcommand{\s}{{\mathfrak s}}
\renewcommand{\o}{{\mathfrak o}}
\newcommand{\so}{{\mathfrak so}}
\renewcommand{\u}{{\mathfrak u}}
\newcommand{\su}{{\mathfrak su}}
\newcommand{\ssl}{{\mathfrak sl}}
\newcommand{\ssp}{{\mathfrak sp}}
\renewcommand{\t}{{\mathfrak t }}
\newcommand{\Cinf}{C^{\infty}}
\newcommand{\la}{\langle}
\newcommand{\ra}{\rangle}
\newcommand{\half}{\scriptstyle\frac{1}{2}}
\newcommand{\p}{{\partial}}
\newcommand{\notsub}{\not\subset}
\newcommand{\iI}{{I}} 
\newcommand{\bI}{{\partial I}} 
\newcommand{\LRA}{\Longrightarrow}
\newcommand{\LLA}{\Longleftarrow}
\newcommand{\lra}{\longrightarrow}
\newcommand{\LLR}{\Longleftrightarrow}
\newcommand{\lla}{\longleftarrow}
\newcommand{\INTO}{\hookrightarrow}

\newcommand{\QED}{\hfill$\Box$\medskip}
\newcommand{\UuU}{\Upsilon _{\square}(H_0) \times \Uu _{\square} (J_0)}
\newcommand{\bm}{\boldmath}

\newcommand{\commentgreen}[1]{\green{(*)}\marginpar{\fbox{\parbox[l]{3cm}{\green{#1}}}}}
\newcommand{\commentred}[1]{\red{(*)}\marginpar{\fbox{\parbox[l]{3cm}{\red{#1}}}}}
\newcommand{\commentblue}[1]{\blue{(*)}\marginpar{\fbox{\parbox[l]{3cm}{\blue{#1}}}}}

\newcommand{\Der}{\operatorname{Der}}
\newcommand{\id}{\operatorname{id}}
\newcommand{\bla}{\langle \! \langle}
\newcommand{\bra}{\rangle \! \rangle}
\newcommand{\blq}{[ \! [}
\newcommand{\brq}{] \! ]}
\newcommand{\pr}{\mathrm{pr}}
\newcommand{\rk}{\operatorname{rank}}

\definecolor{orange}{rgb}{0,0,0}

\date{\today}

\title[Dirac-Jacobi bundles]
{Dirac-Jacobi bundles}

\author{Luca Vitagliano}
\address{DipMat, Universit\`a degli Studi di Salerno \& Istituto Nazionale di Fisica Nucleare, GC Salerno, via Giovanni Paolo II n${}^\circ$ 123, 84084 Fisciano (SA) Italy.}
\email{lvitagliano@unisa.it}

\begin{abstract}
We show that a suitable notion of \emph{Dirac-Jacobi structure on a generic line bundle $L$}, is provided by \emph{Dirac structures in the omni-Lie algebroid of $L$}. Dirac-Jacobi structures on line bundles generalize Wade's $\mathcal E^1 (M)$-Dirac structures and unify generic (i.e.~non-necessarily coorientable) precontact distributions, Dirac structures and local Lie algebras with one dimensional fibers in the sense of Kirillov (in particular, Jacobi structures in the sense of Lichnerowicz). We study the main properties of Dirac-Jacobi structures and prove that integrable Dirac-Jacobi structures on line-bundles integrate to (non-necessarily coorientable) precontact groupoids. This puts in a conceptual framework several results already available in literature for $\mathcal E^1 (M)$-Dirac structures.
\end{abstract}



\maketitle

\tableofcontents

\section{Introduction}

A Dirac structure on a manifold $M$ is a maximal isotropic subbundle of the \emph{generalized tangent bundle} $TM \oplus T^\ast M$, whose sections are preserved by the Courant (equivalently, the Dorfman) bracket \cite{C1990, B2013}. Dirac structures unify presymplectic and Poisson structures. While every submanifold of a presymplectic manifold is equipped with an induced presymplectic structure, not every submanifold of a Poisson manifold is equipped with an induced Poisson structure. However, every submanifold of a Poisson manifold is (almost everywhere) equipped with a Dirac structure. From the point of view of Hamiltonian mechanics, submanifolds in a Poisson manifold are constraints on the phase space. Hence Dirac geometry is the right conceptual framework for constrained Hamiltonian mechanics on Poisson manifolds.

In \cite{W2000} Wade defines $\mathcal E^1 (M)$-Dirac structures. They play a similar role in precontact/Jacobi geometry as Dirac structures do in presymplectic/Poisson geometry. However, Wade's definition does not capture non-coorientable precontact distributions, nor local Lie algebras with one dimensional fibers in the sense of Kirillov \cite{Kiri1976}. Recall that a \emph{precontact manifold} is a manifold $M$ equipped with a precontact distribution, i.e.~a hyperplane distribution $C$. The quotient bundle $L := TM / C$ is a non-necessarily trivial line bundle. When $L$ is trivial, $C$ is called \emph{coorientable}. A coorientable precontact distribution can be presented as the kernel of a globally defined $1$-form. Hence coorientable precontact distributions are simpler than generic ones and most authors prefer to work with the former. However, there are important (pre)contact distributions which are not coorientable, e.g.~the canonical contact distribution on the manifold of contact elements is not. Actually, a large part of the theory of coorientable precontact manifolds can be extended to the generic case, after developing a suitable language to deal with generic line bundles.

A similar situation is encountered in Jacobi geometry. A \emph{Jacobi manifold} in the sense of Lichnerowicz \cite{Lich1978} can be regarded as a manifold $M$ equipped with a Lie bracket on smooth functions which is a first order differential operator in each entry. Now, functions on $M$ are sections of the trivial line bundle $\R_M := M \times \R \to M$. If we take sections of a generic line bundle, instead of functions, we get the notion of a \emph{local Lie algebra with one dimensional fibers} in the sense of Kirillov \cite{Kiri1976}. Lichnerowicz's Jacobi manifolds are simpler than Kirillov's local Lie algebras and most authors prefer to work with the former. However, there are important local Lie algebras which do not come from Jacobi manifolds. For instance, a non coorientable contact manifold $(M,C)$ defines a local Lie algebra (containing a full information on $C$) which doesn't come from a Jacobi structure. Actually, a large part of the theory of Jacobi manifolds can be extended to local Lie algebras with one dimensional fibers. 

In this paper we show that a suitable notion of \emph{Dirac-Jacobi structure on a generic line bundle} $L$, is provided by \emph{Dirac structures in the omni-Lie algebroid of $L$}. The omni-Lie algebroid of a vector bundle $E$, here denoted by $\D E$, and Dirac structures therein, have been introduced by Chen and Liu in \cite{CL2010}, and further studied by Chen, Liu and Sheng in \cite{CLS2011}. In the case when $E$ is a line bundle, Dirac structures in $\D E$ encompass local Lie algebras with one dimensional fibers (see \cite{CL2010}), as well as non necessarily coorientable precontact distributions, and Wade's $\mathcal E^1 (M)$-Dirac structures. We study systematically \emph{Dirac-Jacobi bundles}, i.e.~line bundles $L$ equipped with a Dirac structure in $\D L$, extending to this general context a number of results already available for $\mathcal E^1 (M)$-Dirac manifolds. We also find completely new results, including a useful theorem on the local structure of Dirac-Jacobi bundles (Theorem~\ref{theor:local}), a general discussion on backward/forward images of Dirac-Jacobi structures (Section \ref{sec:morphisms}), and a coisotropic embedding theorem (Section \ref{sec:coisotrop}), paralleling similar results holding for Dirac manifolds (see \cite{DW2008}, \cite{B2013}, and \cite{CZ2009} respectively). In our opinion, the omni-Lie algebroid approach to Dirac-Jacobi structures clarifies Wade's theory, putting it in a simple and efficient conceptual framework.

The paper is organized as follows. In Section \ref{sec:line} we recall those aspects of differential geometry of line bundles that we will need in the sequel. Specifically, we discuss, in some details, the \emph{Atiyah algebroid} (also called \emph{gauge algebroid}) of a vector bundle and its functorial properties. In Section \ref{sec:cont_lcs_geom} we present an alternative point of view on (pre)contact geometry: it turns out that (pre)contact geometry can be obtained from (pre)symplectic geometry ``replacing smooth functions with sections of an arbitrary line bundle $L$, and vector fields with sections of the Atiyah algebroid of $L$''. This principle is a guide-line throughout the paper. Actually, Dirac-Jacobi geometry can be obtained from Dirac geometry ``replacing the tangent bundle with the Atiyah algebroid of a line bundle''. This vague statement will be much clearer after Section \ref{sec:DJ_line} where we define Dirac-Jacobi structures on generically non-trivial line bundles. The examples show that Dirac-Jacobi bundles encompass (non-necessarily coorientable) precontact manifolds, local Lie algebras with one dimensional fibers (in particular, Jacobi manifolds and flat line bundles), and Dirac manifolds. In Section~\ref{sec:char_fol} we begin a systematic analysis of Dirac-Jacobi bundles. In particular, we show that a Dirac-Jacobi bundle is essentially the same as a (generically singular) foliation equipped with precontact or locally conformal presymplectic (lcps) structures on its leaves (see \cite{IM2002} for the trivial line bundle version of this result). In Section \ref{sec:loc_struct} we describe, to some extent, the local structure of a Dirac-Jacobi bundle around a point in either a precontact or a lcps leaf of the characteristic foliation (Theorem~\ref{theor:local}). This description is not a full local \emph{normal form theorem}. Nonetheless it is useful for several purposes (see, e.g., Corollary \ref{cor:parity}, and Proposition \ref{prop:KD}). In addition, \textcolor{black}{it allows to prove the existence of certain structures \emph{transverse to characteristic leaves} (Propositions \ref{prop:transv} and \ref{prop:transv_uniq}}). In Section \ref{sec:null_distr} we define the \emph{null distribution} of a Dirac-Jacobi bundle. It plays a similar role as the null distribution of a Dirac manifold, in particular, that of a presymplectic manifold. Namely, under suitable regularity conditions, the null distribution can be quotiented out. The quotient manifold is naturally equipped with a local Lie algebra with one dimensional fibers, which, together with the null foliation, completely determines the original Dirac-Jacobi structure. This allows to describe the structure of Dirac-Jacobi bundles, with sufficiently regular null distribution, in a similar way as for Dirac manifolds. In Section~\ref{sec:morphisms} we discuss how to pull-back and push-forward a Dirac-Jacobi structure along a smooth map. This allows us to define two different kind of morphisms (backward, and forward morphisms) between Dirac-Jacobi bundles. The analysis in this section closely parallels a similar analysis for Dirac manifolds in \cite[Section 5]{B2013}. In Section~\ref{sec:coisotrop} we discuss coisotropic submanifolds in manifolds equipped with a local Lie algebra structure. Under suitable regularity conditions, a coisotropic submanifold inherits a Dirac-Jacobi structure. Conversely, under suitable regularity conditions, a Dirac-Jacobi bundle can be coisotropically embedded in a manifold equipped with a local Lie algebra with one dimensional fibers (Theorem \ref{theor:coisotrop}) (see 
\cite[Theorem 8.1]{CZ2009} for a Poisson geometric version of this result). In Section~\ref{sec:integr} we prove that Dirac-Jacobi bundles are infinitesimal counterparts of (non-necessarily coorientable) precontact groupoids. Every Dirac structure comes equipped with a Lie algebroid structure integrating (if at all integrable) to a presymplectic groupoid \cite{BCZ2004}. Similarly every $\mathcal E^1 (M)$-Dirac structure comes equipped with a Lie algebroid structure integrating (if at all integrable) to a cooriented precontact groupoid \cite{IW2006}. Theorem \ref{theor:int} generalizes these results to Dirac-Jacobi bundles. Finally, recall that a local Lie algebra with one dimensional fibers can be regarded as a homogenous Poisson manifold via the so called \emph{Poissonization trick}. Similarly, a $\mathcal E^1 (M)$-Dirac structure can be regarded as a homogeneous Dirac structure via a \emph{Dirac-ization trick} \cite{IM2002}, and this is sometimes useful in proving theorems in Dirac-Jacobi geometry from available theorems in Dirac geometry (see, e.g., \cite{IW2006}, see also \cite{WZ2004}). In Appendix \ref{app:Dirac_trick} we show how to adapt the Dirac-ization trick to Dirac-Jacobi structures on non-necessarily trivial line bundles.

\subsection*{Notation, conventions and terminology}
Let $M$ be a smooth manifold, and let $E \to M$ be a vector bundle. A \emph{distribution} $V$ in $E$ is the assignment $x \mapsto V_x$ of a subspace $V_x$ of the fiber $E_x$ of $E$ over $x$ for all points $x \in M$. Let $V$ be a distribution in $E$. The \emph{rank} of $V$ is the integer-valued map $\rk V$ on $M$ defined by $\rk V : x \mapsto \rk_x V := \dim V_x$. A (smooth) section of $V$ is a section $e$ of $E$ such that $e(x) \in V_x$ for all $x \in M$. Distribution $V$ is \emph{smooth} if $V_x$ is spanned by values at $x$ of sections of $V$, for all $x \in M$. The rank of a smooth distribution is a lower semi-continuous function. A smooth distribution is \emph{regular} if its rank is locally constant. Hence a regular distribution in a vector bundle over a connected manifold is a vector subbundle. If $E$ is a Lie algebroid, a distribution $V$ in $E$ is called \emph{involutive} if sections of $V$ are preserved by the Lie bracket on $\Gamma (E)$.

Typical examples of distributions in vector bundles are kernels and images of vector bundle morphisms over the identity. The kernel of a vector bundle morphism (in particular, the intersection of two vector subbundles) is a distribution with upper semi-continuous rank. Hence it is smooth if and only if its rank is locally constant. The image of a vector bundle morphism is a smooth distribution. In particular, its rank is a lower semi-continuous function.

A distribution $K$ in $TM$ is \emph{integrable} if every point of $M$ is contained in a \emph{plaque}, i.e.~a connected, immersed submanifold $\Oo$ such that $T \Oo = K|_\Oo$. An integrable distribution $K$ in $TM$ determines a partition $\Ff$ of $M$ into \emph{leaves}, i.e.~maximal plaques. Partition $\Ff$ will be referred to as a \emph{foliation} (of $M$), specifically the \emph{integral foliation of $K$}, and $K$ will be also called the \emph{tangent distribution to $\Ff$} and denoted by $T \Ff$. (A version of) Stefan-Sussman Theorem asserts that a smooth distribution $K$ in $TM$ is integrable if and only if it is involutive and, additionally, $\rk K$ is constant along the flow lines of sections of $K$. In particular, a regular distribution in $TM$ is integrable if and only if it is involutive (Frobenius theorem). The integral foliation of a regular distribution in $TM$ will be called \emph{regular}. A regular foliation of $M$ is \emph{simple} if the space $M_{\mathrm{red}}$ of leaves carries a smooth manifold structure such that the canonical projection $M \to M_{\mathrm{red}}$ is a (surjective) submersion.

We assume that the reader is familiar with (fundamentals of) the theory of Lie groupoids, Lie algebroids and their representations (see, e.g., \cite{CF2010} and references therein). The unfamiliar reader will find a quick introduction to those aspects of the theory relevant for this paper in \cite{S2013} (see Section 1.2 therein, see also Chapter 4). We only recall here that, given a Lie algebroid $A \to M$, with Lie bracket $[-,-]_A$ and anchor $\rho_A$, and a vector bundle $E \to M$ equipped with a representation of $A$, i.e.~a flat $A$-connection, there is an $E$-valued \emph{Cartan calculus} on $A$ consisting of the following operators:
\begin{itemize}
\item the $E$-valued \emph{Lie algebroid differential}
\[
d_{A, E} : \Gamma (\wedge^\bullet A^\ast \otimes E) \to \Gamma (\wedge^\bullet A^\ast \otimes E),
\]
\end{itemize}
and, for every section $\alpha$ of $A$,
\begin{itemize}
\item the \emph{contraction} $i_\alpha : \Gamma (\wedge^\bullet A^\ast \otimes E) \to \Gamma (\wedge^\bullet A^\ast \otimes E)$ taking an $E$-valued, skew-symmetric multilinear form $\omega$ on $A$ to $i_\alpha \omega = \omega (\alpha, -, \ldots, -)$,
\item the \emph{Lie derivative} $\Ll_\alpha : \Gamma (\wedge^\bullet A^\ast \otimes E) \to \Gamma (\wedge^\bullet A^\ast \otimes E)$, defined as $\Ll_\alpha := [i_\alpha, d_{A,E}] = i_\alpha d_{A,E} + d_{A,E} i_\alpha$.
\end{itemize}
The above operators fulfill the following (additional) \emph{Cartan identities}
\begin{gather*}
[\Ll_\alpha, i_\beta]  = i_{[\alpha, \beta]_A}, \quad [\Ll_\alpha, \Ll_\beta] = \Ll_{[\alpha, \beta]_A} , \\
[d_{A,E}, d_{A,E}]  =  [d_{A, E}, \Ll_\alpha] = [i_\alpha , i_\beta] = 0.
\end{gather*}
for all $\alpha, \beta \in \Gamma (A)$, where $[-,-]$ denotes the \emph{graded commutator}.

\section{Reminder on the Atiyah algebroid of a vector bundle}\label{sec:line}
In order to fix the notation, we collect, in this section, some already known facts about the Atiyah algebroid of a vector bundle, including its definition and its interaction with vector bundle morphisms. For more information about the Atiyah algebroid the reader may refer, e.g., to Mackenzie's book \cite{M2005} and references therein (especially reference \cite{KM2002} of the present paper). Beware that our terminology and notation are slightly different from Mackenzie's ones. 

Let $M$ be a smooth manifold, and let $E \to M$ be a vector bundle over it. A \emph{derivation of $E$} is an $\R$-linear map $\Delta: \Gamma (E) \to \Gamma (E)$ such that there exists a, necessarily unique, vector field $X \in \mathfrak X (M)$, called the \emph{symbol of $\Delta$} and also denoted by $\sigma (\Delta)$, satisfying the following Leibniz rule
\[
\Delta (f e) = X(f) e + f \Delta (e),
\]
for all $f \in C^\infty (M)$ and $e \in \Gamma (E)$.

\begin{remark}
Derivations are first order differential operators. If $E$ is a line bundle, then every first order differential operator $\Gamma (E) \to \Gamma (E)$ is a derivation.
\end{remark}

Derivations of $E$ can be regarded as sections of a Lie algebroid $D E \to M$, the \emph{Atiyah algebroid of $E$}, defined as follows. The fiber $D_x E$ of $D E$ through $x \in M$ consists of $\R$-linear maps $\Delta : \Gamma (E) \to E_x$ such that there exists a, necessarily unique, tangent vector $X \in T_x M$, called the \emph{symbol of $\Delta$} and also denoted by $\sigma (\Delta)$, satisfying the following Leibniz rule
\[
\Delta (f e) = X (f) e_x + f (x) \Delta (e),
\]
$f \in C^\infty (M)$ and $e \in \Gamma (E)$. It is easy to see that $D_x E$ is a vector space. Choose coordinates $(x^i)$ on $M$, and a local basis $(\varepsilon^a)$ of $\Gamma (E)$. Then a basis of  $D_x E$ is $(\delta_i , \varepsilon_a^b )$ defined as follows. For $e = f_a \varepsilon^a \in \Gamma (E)$ put
\[
\delta_i (e) := \frac{\partial f_a}{\partial x^i} (x) \varepsilon^a_x \quad \text{and} \quad \varepsilon_a^b (e) := f_a (x) \varepsilon^b_x.
\]
This shows that the $D_x E$'s are fibers of a vector bundle $D E \to M$ whose rank is $\dim M + (\rk E)^2$. Sections of $D E \to M$ are denoted by $\Der E$. They identify with derivations of $E$. Hence $D E \to M$ is a Lie algebroid, whose Lie bracket is the commutator of derivations and whose anchor  $\sigma : D E \to TM$ maps a derivation to its symbol. A representation of a Lie algebroid $A \to M$ in a vector bundle $E \to M$ can then be regarded as a morphism of Lie algebroids $A \to D E$. In particular, there is a \emph{tautological representation} of the Lie algebroid $D E$ in $E$ itself, given by the identity $\id : D E \to D E$. The de Rham complex $\Gamma (\wedge^\bullet (D E)^\ast \otimes E)$ of $D E$ with values in $E$ is sometimes called the \emph{der-complex} \cite{R1980}. We denote be $d_{D}$ the differential in the der-complex. The der-complex is actually acyclic. Even more, it possesses a canonical contracting homotopy given by contraction with the identity derivation $\mathbb 1 : E \to E$, $\mathbb 1 e = e$. The Atiyah algebroid $D E$ is often called the \emph{gauge algebroid of $E$}.

Correspondence $E \to D E$ is functorial in the following sense. Let $\mathbf{VB}^{\mathrm{reg}}$ be the category whose objects are vector bundles (over possibly different base manifolds) and whose morphisms are \emph{regular morphisms} of vector bundles, i.e.~vector bundle maps that are isomorphisms on fibers. Then correspondence $E \mapsto DE$ can be promoted to a functor from $\mathbf{VB}^{\mathrm{reg}}$ to the category of Lie algebroids, with morphisms of Lie algebroids over possibly different base manifolds. Namely, let $E \to M$ and $E' \to M'$ be vector bundles, and let $F : E \to E'$ be a regular vector bundle morphism over a smooth map $\underline{F} : M \to M'$. In particular, a section $e'$ of $E'$ can be pulled-back to a section $F^\ast e'$ of $E$, defined by $(F^\ast e')(x) := (F|_{E_x}^{-1} \circ e \circ \underline{F}) (x)$, for all $x \in M$. Then $F$ induces a morphism of Lie algebroids $d_{D} F : D E \to D E'$ defined by
\[
((d_{D} F) \Delta) e' := \Delta (F^\ast e'), \quad \Delta \in D E, \quad e' \in \Gamma (E').
\]
It is easy to see that $\rk_x d_{D} F = \rk_x d\underline F + (\rk E)^2$ for all $x \in M$. If there is no risk of confusion, we also denote by $F_\ast$ the vector bundle morphism $d_{D} F$. Of a special interest is the case when $\underline F : M \to M'$ is the inclusion of a submanifold and $F : E = E'|_M \to E'$ is the inclusion of the restricted vector bundle. In this case $d_D F : D (E'|_M) \to D E'$ is an embedding whose image consists of derivations of $E'$ whose symbol is tangent to $M$. We will often regard $D (E'|_M)$ as a subbundle of $D E'$ (over the submanifold $M$).

\begin{remark}\label{rem:inf_aut}
There is a more geometric interpretation of derivations of a vector bundle which is often useful. Namely, derivations of a vector bundle $E \to M$ can be understood as infinitesimal vector bundle automorphisms of $E$, as follows. let $\{ \Delta_t \}$ be a smooth one-parameter family of derivations of $E$, and let $X_t = \sigma (\Delta_t)$. Denote by $\{ \underline \Phi {}_t \}$ the one-parameter family of diffeomorphisms generated by $\{ X_t \}$. Then there exists a unique one-parameter family $\{ \Phi_t \}$ of vector bundle automorphisms $\Phi_t : E \to E$, over $\{ \underline \Phi {}_t \}$ such that
\[
\frac{d}{dt} \Phi_t^\ast e = \Phi_t^\ast \Delta_t e,
\]
for all $e \in \Gamma (E)$. If $\Delta_t = \Delta$ is constant, then family $\{ \Phi_t \}$ is a flow, over the flow of $X = \sigma(\Delta)$, and 
\[
\Delta e = \left. \frac{d}{dt} \right|_{t=0} \Phi_t^\ast e.
\]
\end{remark}

In the case when $E = L$ is a line bundle, then $\rk D L = \dim M +1$ and a basis of $D_x L$ is $(\delta_i, \mathbb 1_x)$, where $\mathbb 1_x$ maps section $e$ to $e_x$. Moreover, the dual bundle of $D L$ is $J^1 L \otimes L^\ast$, where $J^1 L$ is the first jet bundle of $L$. In this case, der-complex looks like
\[
0 \longrightarrow \Gamma (L) \longrightarrow \Gamma (J^1 L) \longrightarrow \Gamma (\wedge^2 (DL)^\ast \otimes L) \longrightarrow \cdots,
\]
and the first differential $\Gamma (L) \to \Gamma (J^1 L)$ agrees with the first jet prolongation $j^1 : \Gamma (L) \to \Gamma (J^1 L)$. 
Let $L \to M$ and $L' \to M'$ be line bundles and let $F : L \to L'$ be a regular morphism of vector bundles over a smooth map $\underline{F} : M \to M'$. A sections $\psi '$ of $J^1 L'$ can be pulled-back to a section $F^\ast \psi '$ of $J^1 L$ as follows. First define a vector bundle morphism $j^1 F : \underline{F}^\ast J^1 L' \to J^1 L$ by putting $(j^1 F) j_{\underline F (x)}^1 \lambda ' := j^1_x (F^\ast \lambda ')$, for all $x \in M$. Next put $F^\ast \psi = j^1 F \circ \psi \circ \underline F$. By definition $j^1 F^\ast \lambda' = F^\ast j^1 \lambda' $ for all $\lambda' \in \Gamma (L')$. If there is no risk of confusion we also denote by $F^\ast$ the vector bundle morphism $j^1 F$.

The pull-backs $F^\ast : \Gamma (L') \to \Gamma (L)$ and $F^\ast : \Gamma (J^1 L') \to \Gamma (J^1 L)$ can be extended to a degree zero map, also denoted by $F^\ast : \Gamma (\wedge^\bullet (DL')^\ast \otimes L') \to \Gamma (\wedge^\bullet (DL)^\ast \otimes L)$, in the obvious way. Moreover,
\[
F^\ast d_{D} \omega' = d_{D} F^\ast \omega'
\]
for all $\omega' \in \Gamma (\wedge^\bullet (DL')^\ast \otimes L')$.

Vector bundle morphism $j^1 F$ is adjoint to $d_{D} F$ in the sense that\linebreak $\langle F_\ast \Delta , \psi ' \rangle = \langle \Delta , F^\ast \psi\rangle$ for all $\Delta \in D_x L$ and $\psi \in J^1_{\underline{F}(x)} L'$, where $\langle -,-\rangle : D L \otimes J^1 L \to L$, is the duality pairing twisted by $L$. Hence, if $\Delta \in \Der L$ and $\Delta ' \in \Der L'$ are \emph{$F$-related}, i.e.~$\Delta (F^\ast \lambda') = F^\ast (\Delta ' \lambda ')$ for all $\lambda ' \in \Gamma (L')$ (in other words, $F_\ast (\Delta_x) = \Delta'_{\underline F (x)}$ for all $x \in M$), then
\begin{equation}\label{eq:F-related}
\begin{aligned}
i_\Delta F^\ast \omega '  & = F^\ast i_{\Delta '} \omega '  \\
\Ll_{\Delta} F^\ast \omega ' & =  F^\ast \Ll_{\Delta '} \omega ' 
\end{aligned}
\end{equation}
for all $\omega ' \in \Gamma (\wedge^\bullet (DL')^\ast \otimes L')$.

\section{A new look at contact and locally conformal symplectic geometries}\label{sec:cont_lcs_geom}

A \emph{precontact manifold} is a manifold $M$ equipped with a \emph{precontact distribution} $C$, i.e.~a hyperplane distribution on $M$. Let $(M,C)$ be a precontact manifold. Denote by $L$ the quotient line bundle $TM/C$, and by $\theta : TM \to L$ the projection. We also interpret $\theta$ as an $L$-valued $1$-form on $M$. It contains a full information on $C$. Actually, a precontact distribution on $M$ can be equivalently defined as a line bundle $L$ equipped with a \emph{precontact form} $\theta : TM \to L$, i.e.~a nowhere zero $L$-valued $1$-form $\theta : TM \to L$. The curvature form $\omega_C : \wedge^2 C \to L$, $(X,Y) \mapsto \theta ([X,Y])$, defines a morphism $(\omega_C)_\flat : C \to C^\ast \otimes L$ of vector bundles whose kernel $K_C$ is a, generically singular, involutive, subdistribution of $C$ called the \emph{null distribution} of $C$. Since the rank of $K_C$ is an upper semi-continuous function on $M$, then $K_C$ is smooth if and only if it is regular. In this case, the rank of $K_C$ is locally constant and $K_C$ is integrable by involutivity.

A precontact distribution $C$ is \emph{contact} if $\omega_C$ is non-degenerate, i.e.~$(\omega_C)_\flat$ is an isomorphism. A contact distribution $C$ defines a \emph{local Lie algebra $(L = TM/C, \{-,-\})$} in the sense of Kirillov (see Remark \ref{rem:cont_Jac} below). {Following Marle's terminology \cite{Marle1991} we will use the name \emph{Jacobi bundle} for a local Lie algebra with one dimensional fiber. A Jacobi bundle} is a line bundle $L$ equipped with a Lie bracket $\{-,-\}: \Gamma (L) \times \Gamma (L) \to \Gamma (L)$ on its sections, which is, additionally,  a first order differential operator (hence a derivation) in each entry. The Lie bracket of a Jacobi bundle will be called a \emph{Jacobi bracket}. Jacobi manifolds in the sense of Lichnerowicz \cite{Lich1978} are natural sources of (trivial) Jacobi bundles and Jacobi brackets, but there are interesting examples of non-trivial Jacobi bundles.

\begin{remark}\label{rem:cont_Jac}
The Jacobi bracket $\{-,-\}$ defined by $C$ can be described as follows (see, e.g., \cite{CS2013}). Let $X \in \mathfrak X (M)$. The map $\phi_X : C \to C^\ast \otimes L$, $Y \mapsto \theta ([X, Y])$ is a well-defined morphism of vector bundles. Put
\[
X' = (\omega_C)_\flat^{-1} (\phi_X) \in \Gamma (C).
\]
It can be showed that $X-X'$ does only depend on $\lambda := \theta (X)$. We denote it by $X_\lambda$. Now, let $\lambda = \theta (X), \mu = \theta (Y) \in \Gamma (L)$, with $X,Y \in \mathfrak X (M)$. Then $\{\lambda, \mu \} = \theta ([X_\lambda , X_\mu])$. Notice, additionally, that the vector field $X_\lambda$ is the symbol of the derivation $\{\lambda,-\}$ and, moreover, its flow preserves $C$.
\end{remark}

In the following, it will be convenient to take a slightly more general point of view on precontact geometry. Namely, we relax the requirement that a precontact form is nowhere zero, and call a \emph{precontact form} any $1$-form $\theta$ on $M$ with values in a line bundle $L \to M$. The kernel of $\theta$ is a non-necessarily regular distribution on $M$ whose rank is $\dim M$ at points where $\theta$ vanishes, and $\dim M - 1$ at points where $\theta$ is not zero. When $\ker \theta$ is not regular, the null distribution cannot be defined as above. However, there is still a way out as discussed below (Remark \ref{rem:K_omega}).

Now, we propose an alternative approach to precontact geometry, inspired by the \emph{(pre)symplectization trick}. The latter consists in regarding a precontact manifold as a homogeneous symplectic manifold, which is always possible. In our approach, precontact geometry is, in a sense, \emph{symplectic geometry on the Atiyah algebroid of a line bundle}. Let $L \to M$ be any line bundle. Consider the \emph{Atiyah algebroid} $D L$ of $L$ and let $\Der  L$ be sections of $D L$ (see Section \ref{sec:line}). As already remarked, there is a tautological representation of $D  L$ in $L$, and an associated (acyclic) \emph{de Rham complex} $(\Gamma (\wedge^\bullet (D  L)^\ast \otimes L), d_{D})$ (remember that $(D  L)^\ast = J^1 L \otimes L^\ast$). In the following, we denote by $\Omega_L^\bullet$ the graded space $\Gamma (\wedge^\bullet (D  L)^\ast \otimes L )$. Since the contraction $i_ \mathbb{1}$ with the identity operator $ \mathbb{1} := \id_{\Gamma (L)} \in \Der  L$ is a contracting homotopy for $(\Omega^\bullet_L, d_{D})$, i.e.~$[i_ \mathbb{1}, d_{D}] = \id$, it immediately follows that $\Omega^\bullet_L = \ker d_{D} \oplus \ker i_ \mathbb{1}$ with projections $\Omega_L^\bullet \to \ker d_{D}$, $\omega \mapsto \omega - i_ \mathbb{1} d_{D} \omega$, and $\Omega_L^\bullet \to \ker i_ \mathbb{1}$, $\omega \mapsto i_{\mathbb 1} d_{D} \omega$. In particular, $d_{D} : \ker i_ \mathbb{1} \to \ker d_{D}$ is a (degree one) isomorphism of graded vector spaces, and $i_ \mathbb{1} : \ker d_{D} \to \ker i_ \mathbb{1}$ is its inverse isomorphism.

\begin{remark}\label{rem:i_Delta_point}
The $C^\infty (M)$-linear operator $i_ \mathbb{1}$ does actually come from a constant rank morphism of vector bundles $\wedge^\bullet (D  L)^\ast \otimes L \to \wedge^\bullet (D  L)^\ast \otimes L$ also denoted by $i_ \mathbb{1}$. In particular, it induces acyclic differentials on fibers (also denoted $i_\mathbb{1}$). 
\end{remark}

\begin{proposition}\label{prop:prec}
Precontact forms on $M$ with values in $L$ are in one-to-one correspondence with $2$-cocycles in $(\Omega_L^\bullet, d_{D})$. Nowhere zero precontact forms correspond to $2$-cocycles $\omega$ such that $\omega_x \notin \ker i_\mathbb{1}$ for any $x \in M$.
\end{proposition}
\begin{proof}
Let $\theta : TM \to L$ be a precontact form on $M$. Denote by $\Theta : D  L \to L$ the composition
\[
D  L \overset{\sigma}{\longrightarrow} TM \overset{\theta}{\longrightarrow} L,
\]
where $\sigma$ is the anchor of $D  L$, i.e.~the \emph{symbol map}. Regard $\Theta$ as a $1$-cochain in $(\Omega^\bullet_L, d_{D})$ and put $\omega := - d_{D} \Theta$. Then $\omega$ is a $2$-cocycle. From $\sigma \mathbb 1 = 0$, we get $i_{\mathbb 1} \Theta = 0$, hence $\Theta = - i_ \mathbb{1} \omega$. Conversely, let $\omega$ be a $2$-cocycle in $(\Omega^\bullet_L, d_{D})$. Put $\Theta := - i_ \mathbb{1} \omega$, so that $\Theta \in \ker i_ \mathbb{1}$. Since the kernel of $\sigma$ is actually generated by $ \mathbb{1}$, it follows that $\Theta$ descends to an $L$-valued $1$-form $\theta : TM \to L$, i.e.~$\Theta = \theta \circ \sigma$. 

For the second part of the statement, let $\Delta \in \Der  L$ and $x \in M$. We have $\omega ( \mathbb{1}_x, \Delta_x) = \omega ( \mathbb{1}, \Delta)_x = - \Theta (\Delta)_x = - \theta_x (\sigma (\Delta_x))$, and the claim follows from surjectivity of $\sigma$.
\end{proof}

\begin{remark}\label{rem:theta=0}
Let $x \in M$. The proof of Proposition \ref{prop:prec} shows that $\theta_x = 0$ if and only if $\omega (\mathbb 1_x , -) = 0$.
\end{remark}

Clearly, every $2$-cochain $\omega \in \Omega^2_L$ defines a morphism $\omega_\flat : D  L \to (D  L)^\ast \otimes L = J^1 L$. Put $K_\omega := \ker \omega_\flat$. It is a generically non-smooth distribution in $D  L$.

\begin{remark}\label{rem:brack_omega}
By definition, the $2$-cochain $\omega$ is non-degenerate if $\omega_\flat$ is an isomorphism, i.e.~$K_\omega = 0$. In this case the inverse isomorphism $\omega^\sharp : J^1 L \to D  L$ defines a skew-symmetric bracket $\{-,-\}_\omega : \Gamma (L) \times \Gamma (L) \to \Gamma (L)$ via,
\[
\{ \lambda, \mu \}_\omega := \langle \omega^\sharp (j^1 \lambda), j^1 \mu \rangle ,
\]
where $\langle-, - \rangle : D  L \otimes J^1 L \to L$ is the duality pairing twisted by $L$. By definition, $\{-,-\}_\omega$ is a first order differential operator in each entry. Moreover, it is easy to see that $\{-,-\}_\omega$ is a Lie bracket, i.e.~it satisfies the Jacobi identity, if and only if $d_{D} \omega = 0$. In this case $(L, \{-,-\}_\omega)$ is a Jacobi bundle.
\end{remark}

\begin{proposition}\label{prop:null_prec}
Let $\theta : TM \to L$ be a nowhere zero precontact form with values in the line bundle $L \to M$, let $C := \ker \theta$ be the corresponding precontact distribution, and let $\omega \in \Omega^2_L$ be the associated $2$-cocycle. The symbol map $\sigma : D  L \to TM$ establishes a linear bijection between $K_\omega$ and $K_C$. In particular, $\theta$ is a contact form, i.e.~$C := \ker \theta$ is a contact distribution, if and only if $\omega$ is non-degenerate. In this case the bracket $\{-,-\}_\omega$ of Remark \ref{rem:brack_omega} is the Jacobi bracket defined by $C$.
\end{proposition}

\begin{proof}
We use the same notations as in the proof of Proposition \ref{prop:prec}. Work point-wise taking a generic $x \in M$. First of all we remark that, from Proposition \ref{prop:prec}, $ \mathbb{1}_x \notin K_\omega$. Since $ \mathbb{1}$ generates $\ker \sigma$, it follows that the restriction $\sigma : K_\omega \to TM$ is injective. Now, let $\Delta \in \Der  L$ be such that $\Delta_x \in K_\omega$. Compute $\theta (\sigma (\Delta_x)) = \Theta (\Delta_x) = - \omega ( \mathbb{1}_x, \Delta_x) = 0$. This shows that $\sigma (K_\omega) \subset C$. Since $\ker \Theta$ is a regular distribution in $D L$, we can choose $\Delta \in \ker \Theta$. We have to show that, for all such $\Delta$, $\sigma (\Delta_x) \in K_C$, i.e.~that, for all $X \in \Gamma (C)$, $\omega_C (\sigma (\Delta_x), X_x) = 0$, where $\omega_C : \wedge^2 C \to L$ is the curvature form of $C$. Thus, let $\square \in \Der L$ be such that $\sigma (\square) = X$. In particular, $\square \in \Gamma (\ker \Theta)$. Compute
\[
\begin{aligned}
\omega_C (\sigma (\Delta_x) , \sigma (\square_x)) & = \theta ([\sigma (\Delta), \sigma (\square)])_x \\
                                                                                 &= \Theta ([\Delta, \square])_x \\
                                                                                 & = \omega (\Delta_x, \square_x) - \Delta (\Theta (\square))_x + \square (\Theta (\Delta))_x \\  
                                                                                & = 0.
\end{aligned}
\]
Conversely, let $Y \in \Gamma (C)$ be such that $Y_x \in K_C$, and let $\Delta \in \Der  L$ be such that $Y = \sigma (\Delta)$. Then the same computation as above shows that $\ker \Theta_x \subset \ker \omega (\Delta_x, -)$. Since $\Theta_x \neq 0$, it follows that $\omega (\Delta_x, -) = r \Theta_x = \omega (r  \mathbb{1}_x, -)$ for some real number $r$. Hence $\Delta_x - r  \mathbb{1}_x \in K_\omega $. But $\sigma (\Delta_x - r \mathbb{1}_x ) = \sigma (\Delta_x) = X_x$. This concludes the proof of the first part of the proposition.

It remains to show that, when $\omega$ is non degenerate, $\{ -, - \}_\omega$ agrees with the Jacobi bracket $\{-,-\}$ determined by $C$. To see this, first notice that derivations of the form $\{\lambda, -\}$ preserve $\Theta$. Namely, let $\lambda \in \Gamma (L)$ and let $\Delta_\lambda = \{ \lambda , -\}$. The Lie derivative $\Ll_{\Delta_\lambda} \Theta$ of $\Theta$ along $\Delta_\lambda$ vanishes identically. Indeed, for all $\Delta \in \Der  L$, put $\mu = \Theta ( \Delta) = \theta (\sigma (\Delta)) \in \Gamma (L)$, and compute
\[
\begin{aligned}
(\Ll_{\Delta_\lambda} \Theta) (\Delta) & = \Delta_\lambda (\Theta (\Delta)) - \Theta ([\Delta_\lambda, \Delta])\\
                                                            & = \{ \lambda , \mu \} - \theta ([\sigma (\Delta_\lambda), \sigma (\Delta)]) \\
                                                            & = \theta ([\sigma (\Delta_\lambda), \sigma (\Delta - \Delta_{\mu})] \\
                                                           & = 0,
\end{aligned}
\] 
where the last equality follows from the fact that $\sigma (\Delta - \Delta_\mu) \in \Gamma (C)$, and $\sigma (\Delta_\lambda) = X_\lambda$ preserves $C$ (see Remark \ref{rem:cont_Jac}). Now
\[
\omega_\flat (\Delta_\lambda) = -i_{\Delta_\lambda} d_{D}\Theta = d_{D} i_{\Delta_\lambda} \Theta = d_{D}\lambda = j^1 \lambda,
\] 
i.e.~$\omega^\sharp (j^1 \lambda) = \Delta_\lambda$, which concludes the proof.
\end{proof}

\begin{remark}\label{rem:K_omega}
Propositions \ref{prop:prec} and \ref{prop:null_prec} show that the information contained in a precontact distribution $C$ can be actually encoded in a $2$-cocycle $\omega$ in $(\Omega^\bullet_L, d_{D})$. Moreover, from $\omega$, we immediately recover the notion of null subdistribution of a precontact distribution and Jacobi bracket of a contact distribution. Finally, Proposition \ref{prop:null_prec} suggests how to define the null distribution of a (non-necessarily nowhere zero) precontact $1$-form $\theta : TM \to L$: it is $\sigma (K_\omega)$, where $\omega \in \Omega^2_L$ is the $2$-cocycle corresponding to $\theta$. Notice that $\sigma : K_\omega \to \sigma (K_\omega)$ is not injective precisely at points where $\theta$ vanishes. Specifically, from Remark \ref{rem:theta=0}, we see that $\mathbb 1_x \in \ker \sigma \cap K_\omega$ whenever $\theta_x = 0$. 
\end{remark}

Now we turn to locally conformal presymplectic (lcps) geometry. Recall that a lcps manifold is a manifold $M$ equipped with a lcps structure, i.e.~a pair $(\omega, b)$ where $b$ is a closed $1$-form and $\omega$ is a $2$-form such that $d \omega + b \wedge \omega = 0$. A lcps structure $(\omega, b)$ is locally conformal symplectic (lcs) if $\omega$ is non degenerate. There is a more conceptual approach to lcps geometry (see, e.g., \cite[Appendix A]{V2014})). Namely, a closed $1$-form $b$ can be understood as a flat connection in the trivial line bundle $\R_M := \R \times M$, and $d  + b \wedge (-)$ is the associated flat connection differential. This suggests to revise the definition of a lcps structure as follows. A lcps structure on $M$ is a triple $(L, \nabla, \underline\omega)$ where $L \to M$ is a line bundle, $\nabla$ is a flat connection in $L$ and $\underline\omega : \wedge^2 TM \to L$ is an $L$-valued $2$-form such that $d_\nabla \underline\omega = 0$, where $d_\nabla$ is the flat connection differential associated to $\nabla$ (notation $\underline \omega$ instead of $\omega$ will be clear in few lines). A lcps structure $(L, \nabla, \underline\omega)$ is lcs if $\underline\omega$ is non-degenerate, i.e.~the induced vector bundle morphism $\underline\omega{}_\flat : TM \to T^\ast M \otimes L$ is an isomorphism. For instance, even-dimensional characteristic leaves of a Jacobi bundle are equipped with lcs structures in this more general (line bundle theoretic) sense.

\begin{remark}\label{rem:lcs_Jac}
A lcs structure $(L, \nabla, \underline \omega)$ defines a Jacobi bracket $\{-,-\}$ on $\Gamma (L)$ as follows. Let $\lambda \in \Gamma (L)$, and let $d_\nabla \lambda : TM \to L$ be its flat connection differential. Put $X_\lambda := (\underline \omega{}_\flat)^{-1}(d_\nabla \lambda)$. Then $\{\lambda, \mu\} = \underline \omega (X_\lambda, X_\mu)$. The Jacobi identity for $\{-,-\}$ is equivalent to $d_\nabla \omega = 0$.
\end{remark}

Now we show that lcps geometry can be put in the same framework as that proposed for precontact geometry. Specifically, a flat connection in $L$ is an injective Lie algebroid morphism $\nabla : TM \to D L$. Hence the image $\Ii_\nabla := \operatorname{im} \nabla$ is a subalgebroid in $D L$ and the symbol $\sigma : \Ii_\nabla \to TM$ is an isomorphism of Lie algebroids. Conversely, it follows by dimension counting that a transitive Lie subalgebroid of $D L$ is either $D L$ itself or is of the form $\Ii_\nabla$ for a unique flat connection $\nabla$. Hence there is a one-to-one correspondence between flat connections in $L$ and transitive, proper Lie subalgebroids in $D L$. Differential $d_\nabla$ is the $L$-valued Lie algebroid differential of $\Ii_\nabla$ up to the identification $\Ii_\nabla \simeq TM$. This shows that \emph{precontact/lcps geometry is the geometry of an $L$-valued, closed $2$-form on a transitive Lie subalgebroid of the Atiyah algebroid of a line bundle $L$}.

\begin{remark}\label{rem:lcps}
Let $\underline \omega$ be an $L$-valued $2$-form on the manifold $M$. It can be also regarded as a $2$-form on $D L$ itself as follows. Let $ \omega = \sigma^\ast \underline \omega \in \Omega^2_L$ be given by $(\sigma^\ast \underline \omega) (\Delta, \square) := \underline \omega (\sigma (\Delta), \sigma (\square))$, for all $\Delta, \square \in \Der L$. In particular, $\omega \in \ker i_{\mathbb 1}$. Conversely, let $\omega \in \ker i_{\mathbb 1} \cap \Omega^2_L$. Then it descends to an $L$-valued $2$-form on $M$, i.e.~$\omega = \sigma^\ast \underline \omega$ for some $\underline \omega : \wedge^2 TM \to L$. In other words, an $L$-valued $2$-form $\underline \omega$ on $M$ is equivalent to a $2$-cochain $\omega$ in $(\Omega^\bullet_L, d_{D})$ such that $i_{\mathbb 1} \omega = 0$. Now, let $(L, \nabla, \underline \omega)$ be a lcps structure and $\omega = \sigma^\ast \underline \omega$. If the rank of $\underline \omega$ is at least four, then $\omega$ determines $\nabla$. However, in general, it doesn't (consider the case $\underline \omega = 0$), and $\nabla$ is an extra piece of information. Finally, a direct computation shows that
\begin{equation}\label{eq:d_nabla}
(d_{D} \omega ) (\nabla_X, \nabla_Y, \nabla_Z) = (d_\nabla \underline \omega) (X, Y, Z)
\end{equation}
for all $X,Y,Z \in \mathfrak X (M)$. Hence  condition $d_\nabla \underline \omega = 0$ is equivalent to condition $(d_{D} \omega )|_{\Ii_\nabla} = 0$.
\end{remark}

\section{Dirac-Jacobi bundles}\label{sec:DJ_line}
Wade defined a Jacobi version of a Dirac structure, and originally called it a \emph{Dirac structure on $\Ee^1 (M)$} \cite{W2000}. Later the terms \emph{$\Ee^1 (M)$-Dirac structure} \cite{IM2002} and \emph{Dirac-Jacobi structure} \cite{IW2006} were also used to indicate the same object. Wade's definition is based on the following remark. Let $M$ be a smooth manifold and let $\R_M := M \times \R$ be the trivial line bundle over $M$. Sections of the bundle $\Ee^1 (M) := (TM \oplus \R_M) \oplus (T^\ast M \oplus \R_M)$ are equipped with a canonical bracket extending the \emph{Courant bracket} on sections of $TM \oplus T^\ast M$ \cite[Section 3]{W2000}. An $\Ee^1 (M)$-Dirac structure is then a maximal isotropic subbundle of $\Ee^1 (M)$ whose sections are preserved by the canonical bracket. $\Ee^1 (M)$-Dirac structures encompass cooriented precontact distributions and Jacobi structures in the sense of Lichnerowicz \cite{Lich1978} as special cases. Wade also defined conformal Dirac structures in the same spirit as conformal Jacobi structures \cite[Section 1.4]{DLM}, which are essentially equivalent to Jacobi bundles, but she didn't develop the theory beyond the definition (except for the special case of locally conformal Dirac structures \cite{W2004}). In this section we propose an alternative, and easy to use, definition of a conformal Dirac structure, which we call a \emph{Dirac-Jacobi bundle}. Dirac-Jacobi bundles are slightly more general than $\Ee^1 (M)$-Diract structures, and encompass non-coorientable precontact distributions and Jacobi bundles.

With the new understanding of precontact (and lcps) structures presented in the previous section, it is actually easy to find a common framework for precontact distributions and Jacobi bundles. Namely, let $L \to M$ be a line bundle. Similarly as in the presymplectic/Poisson case (or in Wade's $\mathcal{E}^1 (M)$-Dirac case) consider the vector bundle 
\[
\D  L := D L \oplus J^1 L.
\]
Vector bundle $\D  L$ was first considered by Chen and Liu in \cite{CL2010}. They named it the \emph{omni-Lie algebroid} because it plays a similar role for Lie algebroids as Weinstein's \emph{omni-Lie algebras} do for Lie algebras \cite{W2000}. The omni-Lie algebroid $\D L$ is an instance of \emph{$E$-Courant algebroid} \cite{CLS2010} and of \emph{$AV$-Courant algebroid} \cite{LB2011}. Additionally, it is a \emph{contact Courant algebroid} in the sense of Grabowski \cite{G2013}. Grabowski's contact Courant algebroids are actually line bundle theoretic versions of Courant-Jacobi algebroids \cite{GM2002} (called generalized Courant algebroids in \cite{NC2004}). For the purposes of this paper we will not need the whole technology developed in References \cite{NC2004,GM2002,CLS2010,LB2011,G2013}. It will suffice to notice that $\D L$ possesses the following structures:
\begin{enumerate}
\item two natural projections
\begin{equation}\label{eq:pr}
\pr_{D} : \D L \longrightarrow D L, \quad \text{and} \quad \pr_{J^1} : \D L \longrightarrow J^1 L ,
\end{equation}
\item a non-degenerate, symmetric, $L$-valued $2$-form
\[
\bla -, - \bra : \Gamma (\D L) \times \Gamma (\D L) \to \Gamma (L),
\]
given by:
\[
\bla (\Delta , \varphi), (\square, \psi) \bra := \langle \Delta ,\psi \rangle + \langle \square , \varphi \rangle,
\]
and
\item a (non-skew symmetric, Dorfman-like) bracket
\[
\blq -,- \brq : \Gamma (\D L) \times \Gamma (\D L) \to \Gamma (\D L),
\]
given by:
\begin{equation}\label{eq:Dorf_brack}
\blq (\Delta, \varphi), (\square, \psi) \brq := \left([\Delta, \square], \Ll_\Delta \psi - i_\square d_{D} \varphi \right)
\end{equation}
\end{enumerate}
for all $\Delta, \square \in \Der L$, and all $\varphi, \psi \in \Gamma (J^1 L)$.

An easy computation exploiting Cartan calculus on the Lie algebroid $D L$ shows that the bracket $\blq -, - \brq$ satisfies
\begin{align}
\blq \alpha , \beta \brq + \blq \beta , \alpha \brq & = d_{D} \bla \alpha , \beta \bra , \label{eq:cour1}\\ 
\blq \alpha , \blq \beta, \gamma \brq \brq & = \blq \blq \alpha, \beta \brq , \gamma \brq + \blq \beta , \blq \alpha , \gamma \brq \brq ,  \label{eq:cour2}\\
\blq \alpha , f \beta \brq  & = f \blq \alpha , \beta \brq + (\sigma \circ \pr_{D})(\alpha)(f) \beta , \label{eq:cour3}
\end{align}

and, moreover,
\begin{equation}
\bla \blq \alpha, \beta \brq , \gamma \bra + \bla \blq \alpha, \gamma \brq , \beta \bra = \pr_{D} (\alpha) \left( \bla \beta, \gamma \bra \right). \label{eq:cour4}
\end{equation}

for all $\alpha, \beta, \gamma \in \Gamma (\D L)$ and $f \in C^\infty (M)$. 

\begin{remark}
When $L = \R_M$ is the trivial line bundle, then $\D L = \Ee^1 (M)$, and (the skew-symmetrization of) bracket (\ref{eq:Dorf_brack}) is to be compared with the bracket on $\Gamma (\Ee^1 (M))$ used by Wade \cite[Section 3]{W2000}.
\end{remark}

\begin{definition}\label{def:dirac}
A \emph{Dirac-Jacobi bundle} over a manifold $M$ is a line bundle $L \to M$ equipped with a \emph{Dirac-Jacobi structure}, i.e.~a vector subbundle $\mathfrak L \subset \D L$ such that
\begin{enumerate}
\item $ \mathfrak L $ is \emph{maximal isotropic} with respect to $\bla -, - \bra$, and
\item $\Gamma (\mathfrak L )$ is \emph{involutive} with respect to $\blq -, - \brq$, i.e.~$\blq \Gamma (\mathfrak L ), \Gamma (\mathfrak L ) \brq \subset \Gamma (\mathfrak L )$.
\end{enumerate}
\end{definition}

\begin{remark}\label{rem:omni-lie}
Maximal isotropic, involutive subbundles $\mathfrak L$ of the omni-Lie algebroid $\D L$ were first considered in \cite{CL2010}. Among other things, the authors show that, if $\mathfrak L \cap J^1 L = 0$, then $\mathfrak L$ is the \emph{graph of a Jacobi structure} and vice-versa. See Example \ref{ex:Jacobi} below for more details. Dirac structures in omni-Lie algebroids were also considered in \cite{CLS2011} from a somewhat different perspective.
\end{remark}

\begin{remark}\label{rem:char}
The bilinear map $\bla -, - \bra$ has split signature in any local basis of $\Gamma (L)$. Hence a subbundle $ \mathfrak L  \subset \D L$ is maximal isotropic if and only if it is both isotropic and coisotropic, if and only if it is either isotropic or coisotropic and, additionally, $\rk \mathfrak L = (\rk \D L) / 2 = \dim M +1$. Now, for a vector bundle $V \to M$ and a distribution $W$ in $V$, denote by $W^0$ the annihilator of $W$ in $V^\ast \otimes L = \operatorname{Hom} (V, L)$, i.e.~$W^0$ is the distribution consisting of $\phi \in V^\ast \otimes L$ such that $\langle \phi, w \rangle = 0$ for all $w \in W$. The following point-wise equalities hold
\begin{equation}\label{eq:char}
\pr_{D} (\mathfrak L )^0 = \mathfrak L  \cap J^1 L \quad \text{and} \quad \pr_{J^1} (\mathfrak L ) = (\mathfrak L  \cap D L)^0,
\end{equation}
for every maximal isotropic subbundle $ \mathfrak L $ of $\D L$.

Consider the short exact sequence
\begin{equation}\label{eq:Sp_dual}
0 \longrightarrow \langle \mathbb 1 \rangle \longrightarrow D L  \overset{\sigma}{\longrightarrow} TM \longrightarrow 0, 
\end{equation}
where $\mathbb 1 : \Gamma (L) \to \Gamma (L)$ is the identity operator and generates the subbundle $\langle \mathbb 1 \rangle = \operatorname{End} L$ of linear endomorphisms of $L$. Taking duals and tensoring by $L$ in (\ref{eq:Sp_dual}), we get the so called  
\emph{Spencer sequence}
\begin{equation}\label{eq:Spenc}
0 \longleftarrow L \overset{\pr_L}{\longleftarrow} J^1 L \longleftarrow T^\ast M \otimes L \longleftarrow 0,
\end{equation}
where $\pr_L$ is the canonical projection given by $j^1 \lambda \mapsto \lambda$, and the embedding $T^\ast M \otimes L \INTO J^1 L$ adjoint to the symbol $\sigma $ is given by $df \otimes \lambda \mapsto j^1 (f \lambda) -f j^1 \lambda$, where $\lambda \in \Gamma (L)$, and $f \in C^\infty (M)$. In what follows we will always regard $T^\ast M \otimes L$ as a subbundle of $J^1 L$ understanding the embedding $T^\ast M \otimes L \INTO J^1 L$ (beware that a different convention for the Spencer sequence is often used where the embedding $T^\ast M \otimes L \INTO J^1 L $ maps $df \otimes \lambda$ to $f j^1 \lambda - j^1 f \lambda$. However, this choice is not adjoint to (\ref{eq:Sp_dual})).

In addition to (\ref{eq:char}), one can easily check two more pairs of point-wise equalities:
\begin{equation}\label{eq:char1}
\begin{aligned}
((\sigma \circ \pr_{D}) \mathfrak L)^0 & = \mathfrak L  \cap (T^\ast M \otimes L),  \\ 
 (\pr_L \circ \pr_{J^1}) \mathfrak L  & = (\mathfrak L  \cap \langle \mathbb 1 \rangle)^0,
\end{aligned}
\end{equation}
and
\begin{equation}\label{eq:char2}
\begin{aligned}
(\pr_{D} \mathfrak L \cap \langle \mathbb 1 \rangle)^0 & = \pr_L( \mathfrak L \cap J^1 L), \\
  \pr_{J^1} (\mathfrak L ) \cap (T^\ast M \otimes L) & = (\sigma (\mathfrak L  \cap D L))^0.
\end{aligned}
\end{equation}
for every maximal isotropic subbundle $ \mathfrak L $ of $\D L$.
\end{remark}

\begin{remark}
Let $ \mathfrak L $ be a maximal isotropic subbundle of $\D L$. It is involutive (with respect to bracket $\blq -, - \brq$) iff
\[
\bla \blq \alpha, \beta \brq , \gamma \bra = 0,
\]
for all $\alpha , \beta , \gamma \in \Gamma ( D)$. Now, from (\ref{eq:cour1}), (\ref{eq:cour3}), and (\ref{eq:cour4}), the expression 
\[
\Upsilon_{\mathfrak L} (\alpha , \beta , \gamma) := \bla \blq \alpha, \beta \brq , \gamma \bra
\]
is skew-symmetric and $C^\infty (M)$-linear in its arguments $\alpha, \beta, \gamma$. Hence it defines a section $\Upsilon_{\mathfrak L}$ of $\wedge^3 D^\ast \otimes L$. We call $\Upsilon_{\mathfrak L}$ the \emph{Courant-Jacobi tensor} of $ \mathfrak L $. The involutivity condition on a maximal isotropic subbundle $ \mathfrak L $ is thus $\Upsilon_{\mathfrak L} = 0$.
\end{remark}

\begin{remark}
When $L = \R_M$, Definiton \ref{def:dirac} gives back Wade's definition of a Dirac structure in $\Ee^1 (M)$ \cite[Definition 3.2]{W2000}. However, there are interesting examples of Dirac-Jacobi structures on non-trivial line bundles. For instance, every Jacobi bundle is a Dirac-Jacobi bundle of a specific kind (see Example~\ref{ex:Jacobi} below). Finally, it should be mentioned that Dirac structures in $\Ee^1 (M)$ are special instances of Vaisman's Dirac structures in the \emph{stable big tangent bundle of index $h$} (see \cite{V2007} for more details).
\end{remark}

\begin{example}\label{ex:2form}
Let $\omega \in \Omega^2_L = \Gamma (\wedge^2 (D L)^\ast \otimes L)$ be a $2$-cochain in the der-complex (see Section \ref{sec:line}). Since $\omega$ is skew-symmetric, the graph
\[
D_\omega := \operatorname{graph} \omega_\flat = \{ (\Delta, \omega_\flat (\Delta)): \Delta \in D L \}
\]
of $\omega_\flat : D L \to J^1 L$ is a maximal isotropic subbundle of $\D L$. Moreover $ \mathfrak L _\omega$ is involutive if and only if $d_{D} \omega = 0$. Indeed, an easy computation shows that
\[
\Upsilon_{\mathfrak L_\omega} (\alpha_1, \alpha_2, \alpha_3) = (d_{D}\omega) (\Delta_1 , \Delta_2, \Delta_3),
\]
for all $\alpha_i = (\Delta_i, \omega_\flat (\Delta_i)) \in \Gamma (\mathfrak L _\omega)$, $i = 1,2,3$. So $ \mathfrak L _\omega$ is a Dirac-Jacobi structure on $L$ if and only if $\omega$ is a $2$-cocycle in $(\Omega_L^\bullet, d_{D})$. Conversely, a Dirac-Jacobi structure $ \mathfrak L  \subset \D L$ is of the form $ \mathfrak L _\omega$ for some $2$-cocycle $\omega$  in $(\Omega_L^\bullet, d_{D})$ if and only if $ \mathfrak L  \cap J^1 L = \mathbf{0}$. If, additionally, $ (\mathbb 1, 0) \notin \mathfrak L _x$, then $\omega_x \notin \ker i_{\mathbb 1}$, and vice-versa, $x \in M$. Hence, from Proposition \ref{prop:prec}, $L$-valued precontact forms $\theta : TM \to L$ identify with Dirac-Jacobi structures $ \mathfrak L $ on $L$ such that 
\begin{equation}\label{eq:Dirac_prec}
\mathfrak L \cap J^1 L = \mathbf{0},
\end{equation}
and nowhere zero precontact forms correspond to Dirac-Jacobi structures such that (\ref{eq:Dirac_prec}) and, additionally, $ (\mathbb 1, 0) \notin \mathfrak L _x$ for any $x \in M$. In particular, Dirac-Jacobi bundles encompass (non-necessarily coorientable) precontact manifolds.
\end{example}

\begin{example}\label{ex:Jacobi}
Let $\{-,-\} : \Gamma (L) \times \Gamma (L) \to \Gamma (L)$ be a skew-symmetric, first order bidifferential operator. Interpret $\{-,-\}$ as a section $J$ of $\wedge^2 (J^1 L)^\ast \otimes L$, by putting
\[
J(j^1 \lambda, j^1 \mu) = \{ \lambda, \mu\}
\]
for all $\lambda, \mu \in \Gamma (L)$. Clearly, $J$ defines a morphism $J^\sharp : J^1 L \to (J^1 L)^\ast \otimes L = D L$. Since $J$ is skew-symmetric, the graph
\[
\mathfrak L_J := \operatorname{graph} J^\sharp = \{ (J^\sharp(\psi),\psi): \psi \in J^1 L \}
\]
of $J^\sharp$ is a maximal isotropic subbundle of $\D L$. Moreover $ \mathfrak L _J$ is involutive if and only if $J$ is a Lie bracket. Indeed, an easy computation shows that
\[
\Upsilon_{\mathfrak L_J} (\alpha_1, \alpha_2, \alpha_3) = 2 \{\lambda_1 , \{ \lambda_2,  \lambda_3\}\} + \text{cyclic permutations},
\]
for all $\alpha_i  \in \Gamma (\mathfrak L _J)$ of the form $\alpha_i = (J^\sharp (j^1 \lambda_i), j^1 \lambda_i)$, $\lambda_i \in \Gamma (L)$, $i = 1,2,3$. So $ \mathfrak L _J$ is a Dirac-Jacobi structure on $L$ if and only if $(L, \{-,-\})$ is a Jacobi bundle. Conversely, a Dirac-Jacobi structure $ \mathfrak L  \subset \D L$ is of the form $ \mathfrak L _J$ for some Jacobi bundle $(L, \{-,-\})$ if and only if $ \mathfrak L  \cap D L = \mathbf{0}$. Hence Jacobi bundles $(L, \{-,-\})$ identify with Dirac-Jacobi bundles $(L, \mathfrak L)$ such that 
\begin{equation}\label{eq:DJ_jac}
\mathfrak L \cap D L = \mathbf{0},
\end{equation}
as first noticed in \cite{CL2010} (see Theorem 3.16 therein).
\end{example}

\begin{example}\label{ex:inv}
Let $V \subset D L$ be a vector subbundle and let $V^0 \subset J^1 L$ be its annihilator, i.e.~$V^0 := \{\psi \in J^1 L : \langle \Delta, \psi \rangle = 0 \text{ for all } \Delta \in V\}$. Clearly, $V \oplus V^0 \subset \D L$ is an isotropic subbundle. Moreover, it is maximal isotropic by dimension counting. Finally it is involutive if and only if $V$ is involutive, i.e.~$\Gamma (V)$ is preserved by the commutator of derivations. Indeed, an easy computation shows that
\[
\Upsilon_{V \oplus V^0} (\alpha_1, \alpha_2, \alpha_3) = \langle [\Delta_1 , \Delta_2] , \psi_3 \rangle + \langle [\Delta_2 , \Delta_3] , \psi_1  \rangle + \langle [\Delta_3 , \Delta_1] , \psi_2  \rangle ,
\]
for all $\alpha_i = (\Delta_i, \psi_i) \in \Gamma (V \oplus V^0)$, $i = 1,2,3$. So, involutive vector subbundles of $D L$ can be regarded as Dirac-Jacobi structures on $L$. Notice that the image of a flat linear connection $\nabla : TM \to D L$ is an involutive vector subbundle of $D L$. Accordingly, a flat connection in $L$ can be regarded as a Dirac-Jacobi structure on $L$ which we denote by $ \mathfrak L _\nabla$. 
On the other hand, the identity operator $\mathbb{1} : \Gamma (L) \to \Gamma (L)$ also spans an involutive vector subbundle of $D L$. We denote by $ \mathfrak L _\mathbb{1}$ the corresponding Dirac-Jacobi structure. Notice that $\D L = \mathfrak L _\nabla \oplus \mathfrak L_\mathbb{1}$ for every flat linear connection $\nabla$ in $L$.
\end{example}

\begin{example}\label{ex:lcps}
A lcps structure $(L, \nabla, \underline \omega)$ (see discussion preceding Remark \ref{rem:lcps}) defines a Dirac-Jacobi bundle as follows. Put $\omega := \sigma^\ast \underline \omega$ (see Remark \ref{rem:lcps}) and consider the Dirac-Jacobi structure $ \mathfrak L _\nabla$ defined by $\nabla$ (Example \ref{ex:inv}). \emph{Deform} $ \mathfrak L _\nabla$ to a new vector subbundle
\[
\mathfrak L_{\nabla, \underline \omega} := \{ (\Delta, \psi + \omega_\flat (\Delta)) : (\Delta, \psi) \in \mathfrak L _\nabla \}.
\]
of $\D L$. Since $\omega$ is skew-symmetric, and $\langle \Delta , \psi \rangle = 0$ for all $(\Delta , \psi) \in \mathfrak L _\nabla $, then $ \mathfrak L _{\nabla, \underline \omega}$ is a maximal isotropic subbundle. Moreover, it is involutive if and only if $d_\nabla \underline \omega = 0$. Indeed, an easy computation shows that
\[
\Upsilon_{\mathfrak L_{\nabla, \underline\omega}} (\alpha_1, \alpha_2, \alpha_3) = (d_{D} \omega)(\nabla_{X_1}, \nabla_{X_2}, \nabla_{X_3}) = (d_\nabla \underline \omega )(X_1, X_2, X_3), 
\] 
for all $\alpha_i \in \Gamma (\mathfrak L _{\nabla, \underline \omega})$, where $X_i := \sigma (\mathrm{pr}_{D} \alpha_i)$, $i = 1,2,3$, and, in the last equality, we used (\ref{eq:d_nabla}). Dirac-Jacobi bundle $(L, \mathfrak L_{\nabla, \underline \omega})$ contains a full information on the lcps structure $(L, \nabla, \underline \omega)$. Hence  lcps manifolds identify with certain Dirac-Jacobi bundles.
\end{example}

\begin{example}\label{ex:hom_Poisson}
Recall that a \emph{homogeneous Poisson structure} is a pair $(\pi, Z)$ where $\pi$ is a Poisson bivector, and $Z$ is a vector field such that $\Ll_Z \pi = -\pi$. A homogeneous Poisson structure on $M$ can be regarded as a Dirac-Jacobi structure on the trivial line bundle $\R_M$ as follows. Let $\pi$ be a bivector field and let $Z$ be a vector field on $M$. Put
\begin{align}\label{eq:hom_Poisson}
 \mathfrak L_{(\pi,Z)} :=  \{ (h - hZ + \pi^\sharp (\eta), \eta \otimes 1 + \eta (Z) j^1 1) \in \D \R_M : h \in \R, \eta \in T^\ast M \},
\end{align}
where $\pi^\sharp : T^\ast M \to TM$ is the vector bundle morphism induced by $\pi$. It is easy to see that $\mathfrak L_{(\pi,Z)} \subset \D \R_M$ is a maximally isotropic subbundle, and it is a Dirac-Jacobi structure if and only if $(\pi, Z)$ is a homogeneous Poisson structure \cite[Section 4]{W2000}. In this case, $(\pi, Z)$ is completely determined by $\mathfrak L_{(\pi,Z)}$. Even more, Dirac-Jacobi structures of the form $\mathfrak L_{(\pi, Z)}$ can be characterized according to the following
\begin{proposition}\label{prop:hom_Poisson}
A Dirac Jacobi structure $\mathfrak L \subset \D \R_M$ is of the form $\mathfrak L_{(\pi, Z)}$ for some homogeneous Poisson structure $(\pi, Z)$ if and only if $\tau : \mathfrak L \cap D \R_M \to \R_M$ is an isomorphism (here $\tau : D \R_M \to \R_M$ is the projection $\Delta \mapsto \Delta (1)$).
\end{proposition}
\begin{proof}
Let $(\pi, Z)$ be a homogeneous Poisson structure. It follows from (\ref{eq:hom_Poisson}) that $\mathfrak L_{(\pi, Z)} \cap D \R_M = \{ h - hZ : h \in \R \}$, hence $\operatorname{pr} : \mathfrak L_{(\pi,Z)} \cap D \R_M \to \R_M$, $h - hZ \mapsto h$ is an isomorphism.

Conversely, let $\mathfrak L \subset \D \R_M$ be a Dirac-Jacobi structure such that $\tau : \mathfrak L \cap D \R_M \to \R_M$ is an isomorphism. Put 
\[
\Delta := \tau^{-1} (1) \in \Gamma (D \R_M) \quad \text{and}  \quad Z := -\sigma (\Delta) = \Delta (1) - \Delta \in \mathfrak X (M).
\]
Next we define a bivector field $\pi$. To do this, notice, first of all, that $\operatorname{pr}_{J^1} (\mathfrak L) = (\mathfrak L \cap D \R_M)^0$ is a vector subbundle in $J^1 \R_M$. Even more, the canonical projection $\zeta : J^1 \R_M \to T^\ast M$, $j^1 h \mapsto dh$, restricts to an isomorphism \[
\zeta : \operatorname{pr}_{J^1} (\mathfrak L) \to T^\ast M.
\]
In the following we denote $\mathcal T := \operatorname{pr}_{J^1} (\mathfrak L)$. The vector bundle $\mathcal T$ is equipped with a skew-symmetric bilinear map $\Pi : \wedge^2 \mathcal T \to \R_M$ given by
\[
\Pi (\phi, \psi) := \langle \Delta, \psi \rangle,
\]
for all $(\Delta, \phi), (\square, \psi) \in \mathfrak L$. From the fact that $\mathfrak L$ is isotropic and the first identity in (\ref{eq:char}), $\Pi$ is well defined. It is easy to see that $\Pi$ (and $\mathcal T$) completely determine $\mathfrak L$. Indeed
\begin{equation}\label{eq:L_Pi}
\mathfrak L = \{ (\Delta, \psi) : \psi \in \mathcal T \text{ and } \Pi (\psi, \phi) = \langle \Delta , \phi \rangle \text{ for all } \phi \in \mathcal T \}.
\end{equation}
Now, bilinear map $\Pi$, identifies, via isomorphism $\zeta :  \mathcal T \to T^\ast M$, with a bivector field $\pi \in \Gamma (\wedge^2 TM)$. Define $\mathfrak L_{(\pi, Z)} \subset \D \R_M$ via (\ref{eq:hom_Poisson}). Then, as already noticed, $\mathfrak L_{(\pi, Z)}$ is a maximally isotropic subbundle of $\D \R_M$. Using (\ref{eq:L_Pi}) it is easy to see that $\mathfrak L_{(\pi, Z)} \subset \mathfrak L$, hence $\mathfrak L_{(\pi, Z)} = \mathfrak L$ (in particular, $(\pi, Z)$ is a homogeneous Poisson structure).
\end{proof}
We conclude that homogeneous Poisson structures identify with Dirac-Jacobi structures $\mathfrak L \subset \D \R_M$ such that $\tau : D \R_M \cap \mathfrak L \to \R_M$ is an isomorphism, or, in other words, the vector subbundle $\mathfrak L \cap DL$ in $D \R_M$ has rank $1$ and is transversal to the image of the standard inclusion $TM \INTO D \R_M$.
\end{example}

\begin{example}\label{ex:Dirac}
A Dirac structure $ \mathfrak L  \subset TM \oplus T^\ast M$ can be regarded as a Dirac-Jacobi structure $\widehat{\mathfrak L}$ on the line bundle $\R_M$ by putting \cite[Remark 3.1]{W2000}:
\[
\widehat{\mathfrak L} := \{(X, \eta \otimes 1 + h j^1 1) \in \D \R_M : h \in \R \text{ and } (X, \eta) \in \mathfrak L  \}.
\]
Being the image of $\widehat{\mathfrak L}$ under projection $\D \R_M \to TM \oplus T^\ast M$, Dirac structure $ \mathfrak L $ is completely determined by $\widehat{\mathfrak L}$. In other words, correspondence $ \mathfrak L  \mapsto \widehat{\mathfrak L}$ is an inclusion of Dirac structures on $M$ into Dirac-Jacobi structures on $\R_M$.
\end{example}


\begin{example}\label{ex:gauge}
New Dirac-Jacobi structures can be obtained via \emph{gauge transformations} (see also \cite[Example 3.6]{B2013}). Let $\omega \in \Omega^2_L$ and let $ \mathfrak L $ be a Dirac-Jacobi structure on $L$. Define a vector subbundle $\tau_\omega \mathfrak L \subset \D L$ by putting
\[
\tau_\omega \mathfrak L := \{ (\Delta, \psi + \omega_\flat (\Delta)) : (\Delta, \psi) \in \mathfrak L  \}.
\] 
Since $\omega$ is skew symmetric, $\tau_\omega \mathfrak L$ is isotropic. Moreover, it shares the same rank as $ \mathfrak L $, hence it is maximal isotropic. An easy computation shows that
\[
\begin{aligned}
\Upsilon_{\tau_\omega \mathfrak L} (\alpha_1, \alpha_2, \alpha_3 ) & = \Upsilon_{\mathfrak L} (\alpha_1', \alpha_2', \alpha_3') + (d_{D} \omega)(\Delta_1, \Delta_2, \Delta_3) \\
& =  (d_{D} \omega)(\Delta_1, \Delta_2, \Delta_3),
\end{aligned}
\]
for all $\alpha_i = \alpha'_i + (0, \omega_\flat (\Delta_i)) \in \Gamma (\tau_\omega \mathfrak L)$, where $\alpha'_i \in \Gamma (\mathfrak L )$, and $\Delta_i = \pr_{D} (\alpha_i) = \pr_{D}(\alpha'_i)$, $i = 1,2,3$. Hence $\tau_\omega \mathfrak L$ is a Dirac-Jacobi structure whenever 
\[
d_{D} \omega |_{\pr_{D} (\mathfrak L )} = 0
\]
(as in the case of $ \mathfrak L _{\nabla, \underline \omega}$ in Example \ref{ex:lcps}). In particular, if $\omega$ is $d_{D}$-closed, then $\tau_\omega \mathfrak L$ is a Dirac-Jacobi structure for all $ \mathfrak L $.
\end{example}

\section{Characteristic foliation of a Dirac-Jacobi bundle}\label{sec:char_fol}
Dirac structures are equivalent to presymplectic foliations, i.e.~foliations\linebreak equipped with a presymplectic structure on each leaf. Similarly, Dirac-Jacobi bundles are equivalent to foliations equipped with either a precontact form or an lcps structure on each leaf. This was proved in \cite{IM2002} for Dirac-Jacobi structures on the trivial line bundle $\R_M$. In this section we discuss the general case. 

Let $(L, \mathfrak L)$ be a Dirac-Jacobi bundle over a manifold $M$. Vector bundle $ \mathfrak L  \to M$ is naturally a Lie algebroid acting on the line bundle $L$. Indeed, denote by $[-,-]_{\mathfrak L} : \Gamma (\mathfrak L ) \times \Gamma (\mathfrak L ) \to \Gamma (\mathfrak L )$ the restriction of the bracket (\ref{eq:Dorf_brack}) to sections of $ \mathfrak L $. Moreover, denote by $\rho_{\mathfrak L} : \mathfrak L \to TM$ the composition of $\pr_{D}$ (see (\ref{eq:pr})) restricted to $ \mathfrak L $, followed by the symbol map $\sigma : D L \to TM$. From (\ref{eq:cour1}) and (\ref{eq:cour3}), $[-,-]_{\mathfrak L}$ and $\rho_{\mathfrak L}$ are the Lie bracket and the anchor of a Lie algebroid. Additionally, denote by $\nabla^{\mathfrak L} : \mathfrak L \to D L$ the restriction of $\mathrm{pr}_{D}$ to $ \mathfrak L $ so that $\rho_{\mathfrak L} = \sigma \circ \nabla^{\mathfrak L}$. Since $\pr_{D}$ intertwines the bracket (\ref{eq:Dorf_brack}) and the commutator of derivations, $\nabla^{\mathfrak L}$ is a Lie algebroid representation.

As usual, the image of the anchor $\rho_{\mathfrak L}$ is an integrable, non-necessarily regular, smooth distribution in $TM$, integrating to a foliation which we denote by $\Ff_{\mathfrak L} $ and call the \emph{characteristic foliation of the Dirac-Jacobi bundle $(L, \mathfrak L)$}. Leaves of $\Ff_{\mathfrak L} $ are called \emph{characteristic leaves of $(L, \mathfrak L)$}. Similarly, the image of the flat connection $\nabla^{\mathfrak L}$ is an involutive, non-necessarily regular, smooth distribution in the Lie algebroid $D L$. We denote it by $\Ii_{\mathfrak L}$. The symbol $\sigma : D L \to TM$ maps $\Ii_{\mathfrak L}$ surjectively onto $T \Ff_{\mathfrak L} $. 

Distribution $\Ii_{\mathfrak L}$ is equipped with an $L$-valued, skew-symmetric bilinear map $\omega : \wedge^2 \Ii_{\mathfrak L} \to L$ given by
\[
\omega (\Delta, \square) := \langle \square, \varphi \rangle,
\]
for all $(\Delta, \varphi), (\square, \psi) \in \mathfrak L $. From the fact that $\mathfrak L$ is isotropic and the second identity in (\ref{eq:char}), $\omega$ is well defined. 

Our next aim is to show that $\Ii_{\mathfrak L}$ and $\omega$ induce, on every characteristic leaf, either a precontact form, or a lcps structure. First we prove a 

\begin{lemma}\label{lem:rank}
The rank of $\Ii_{\mathfrak L}$ is constant along characteristic leaves.
\end{lemma}

\begin{proof}
Let $\Oo$ be a characteristic leaf. The bundle map $\rho_{\mathfrak L} : \mathfrak L|_\Oo \to T \Oo$ is surjective. Hence  it splits and, for every vector field $X \in \mathfrak{X} (\Oo)$, there exists $\alpha \in \Gamma (\mathfrak L |_\Oo)$, such that $\rho_{\mathfrak L} (\alpha) = X$. It follows that $\nabla^{\mathfrak L}_\alpha$ is a derivation of $D (L|_\Oo)$. So $\nabla^{\mathfrak L}_\alpha$ generates a flow $\Phi = \{ \Phi_t \}$ of infinitesimal automorphisms of the line bundle $L|_\Oo \to \Oo$ (see Remark \ref{rem:inf_aut}) which lifts to a flow $d_D \Phi := \{d_D \Phi_t \}$ of infinitesimal automorphisms of the vector bundle $D (L|_\Oo) \to \Oo$. Since $\Ii_{\mathfrak L}$ is an involutive distribution, it is preserved under $d_D \Phi$. In particular, $\rk \Ii_{\mathfrak L}$ is constant along the integral curves of $X$. The assertion follows from arbitrariness of $X$ and connectedness of $\Oo$.
\end{proof}

\begin{corollary}\label{cor:char_fol}
For every characteristic leaf $\Oo$, the rank of $\Ii_{\mathfrak L}|_\Oo$ is either $\dim \Oo + 1$, or $\dim \Oo$. In the first case, $\Ii_{\mathfrak L} |_\Oo = D (L|_\Oo)$, while in the second case, $\Ii_{\mathfrak L} |_\Oo$ is the image of a flat connection in $L|_\Oo$.
\end{corollary}

\begin{proof}
It immediately follows from Lemma \ref{lem:rank} and surjectivity of $\sigma \circ \nabla^{\mathfrak L} = \rho_{\mathfrak L} : \mathfrak L|_\Oo \to T \Oo$.
\end{proof}

According to the above corollary, and similarly as for Jacobi bundles, characteristic leaves can be of two different kinds.

\begin{definition}\label{def:char_fol}
A characteristic leaf $\Oo$ is said
\begin{itemize}
\item  \emph{precontact} if $\rk \Ii_{\mathfrak L}|_\Oo = \dim \Oo + 1$,
\item  \emph{lcps} if $\rk \Ii_{\mathfrak L}|_\Oo = \dim \Oo$.
\end{itemize}
Every point in a precontact (resp.~lcps) leaf is said a \emph{precontact} (resp.~\emph{lcps}) \emph{point}.
\end{definition}

Definition (\ref{def:char_fol}) is motivated by the following

\begin{proposition}\label{prop:char_fol}
The $2$-form $\omega$ induces, on every characteristic leaf $\Oo$, a $2$-cochain $\omega_\Oo$ in $ (\Omega^\bullet_{L|_\Oo} , d_{D})$.
\begin{enumerate}
\item If $\Oo$ is precontact, then $\omega_\Oo = -d_D (\theta_\Oo \circ \sigma)$ is the $2$-cocycle corresponding to a (necessarily unique) $L|_\Oo$-valued precontact form $\theta_\Oo : T \Oo \to L|_\Oo$.
\item If $\Oo$ is lcps, then $\Ii_{\mathfrak L}|_\Oo$ is the image of a flat connection $\nabla^\Oo$ in $L|_\Oo$, and $\omega_\Oo = \sigma^\ast \underline \omega{}_\Oo$, for a (necessarily unique) lcps structure $(L|_\Oo, \nabla^\Oo, \underline \omega{}_\Oo)$.
\end{enumerate}
\end{proposition}

\begin{proof}
Let $\Oo$ be a characteristic leaf. Then, $\Ii_{\mathfrak L}|_\Oo$ is a transitive subalgebroid in $D (L|_\Oo)$.

\begin{enumerate}
\item Let $\Oo$ be precontact. From Corollary \ref{cor:char_fol}, $\Ii_{\mathfrak L}|_\Oo = D (L|_\Oo)$, and $\omega_\Oo := \omega|_\Oo$ is a $2$-cochain in $(\Omega^\bullet_{L|_\Oo}, d_{D})$. A direct computation shows that 
\[
(d_{D} \omega) (\Delta_1, \Delta_2, \Delta_3) = \Upsilon_{\mathfrak L} (\alpha_1, \alpha_2, \alpha_3) = 0,
\]
for all $\alpha_i = (\Delta_i, \psi_i) \in \mathfrak L |_\Oo$, $i = 1,2,3$. So $\omega_\Oo$ is a $2$-cocycle. Hence, from the proof of Proposition \ref{prop:prec}, there is $\theta_\Oo$ as in the statement.

\item Now, let $\Oo$ be lcps. From Corollary \ref{cor:char_fol}, $\Ii_{\mathfrak L}|_\Oo$ is the image of a flat connection $\nabla^\Oo$ in $L|_\Oo$. Use connection $\nabla^\Oo$ to split $D (L|_\Oo)$ as $\operatorname{im} \nabla^\Oo \oplus \langle \mathbb 1 \rangle$. The restriction $\omega |_\Oo$ is only defined on $\Ii_{\mathfrak L}|_\Oo = \operatorname{im} \nabla^\Oo $ but can be uniquely prolonged to a $2$-cochain $\omega_\Oo$ in $(\Omega^\bullet_{L|_\Oo}, d_{D})$ such that $i_{\mathbb 1} \omega_\Oo = 0$. It follows that $\omega_\Oo = \sigma^\ast \underline \omega{}_\Oo$ for a unique $L|_\Oo$-valued $2$-form $\underline \omega{}_\Oo$ on $\Oo$ (cf.~Remark \ref{rem:lcps}). A direct computation shows that
\[
(d_{\nabla^\Oo} \underline \omega{}_\Oo ) (X_1, X_2, X_3) = (d_{D} \omega)(\nabla^\Oo_{X_1}, \nabla^\Oo_{X_2}, \nabla^\Oo_{X_3}) = \Upsilon_{\mathfrak L} (\alpha_1, \alpha_2, \alpha_3) = 0,
\]
for all $\alpha_i = (\nabla^\Oo_{X_i}, \psi_i) \in \mathfrak L |_\Oo$, $X_i \in T \Oo$, $i = 1, 2, 3$. So $(L|_\Oo, \nabla^\Oo, \underline \omega{}_\Oo)$ is a lcps structure on $\Oo$.
\end{enumerate}
\end{proof}

\begin{remark}\label{rem:char_fol}
Let $(L, \mathfrak L)$ be a Dirac-Jacobi bundle on $M$. The Dirac-Jacobi structure $ \mathfrak L $ is completely determined by the distribution $\Ii_{\mathfrak L}$ and the $2$-form $\omega: \wedge^2 \Ii_{\mathfrak L} \to L$. Specifically, 
\begin{equation}\label{eq:fol}
\mathfrak L = \{ (\Delta , \psi) : \Delta \in \Ii_{\mathfrak L} \text{ and } i_\Delta \omega = \psi |_{\Ii_{\mathfrak L}} \}.
\end{equation}
Indeed, it is immediate to check that the right hand side of (\ref{eq:fol}) contains $ \mathfrak L $ while it is contained in the orthogonal complement of $ \mathfrak L $ with respect to $\bla -, - \bra$.

Equivalently, $ \mathfrak L $ is completely determined by its characteristic foliation equipped with the induced precontact/lcps structures on leaves. In more details, let $\Oo$ be a characteristic leaf. By the very definition of characteristic foliation, the restricted bundle $ \mathfrak L |_\Oo$ can be regarded as a subbundle in $\D (L|_\Oo)$. Actually $(L|_\Oo, \mathfrak L|_\Oo)$ is a Dirac-Jacobi bundle. Specifically, $ \mathfrak L |_\Oo$ is the Dirac-Jacobi structure corresponding to $\omega_\Oo$ if $\Oo$ is precontact (see Example~\ref{ex:2form}) and to $(\nabla^\Oo, \underline \omega{}_\Oo)$ if $\Oo$ is lcps (see Example \ref{ex:lcps}).

Conversely, let $L \to M$ be a line bundle, and let $\Ff$ be a foliation\linebreak equipped with either an $L|_\Oo$-valued precontact form $\theta_\Oo$ or a lcps structure $(L|_\Oo, \nabla^\Oo, \underline \omega{}_\Oo)$ on each leaf $\Oo$. The foliation $\Ff$ determines a distribution $\Ii$ in $D L$ and a $2$-form $\omega : \wedge^2 \Ii \to L$ by 
\[
\Ii |_\Oo =  
\begin{cases}
D (L|_\Oo) & \text{if $\Oo$ is a precontact leaf} \\
\operatorname{im} \nabla^\Oo & \text{if $\Oo$ is a lcps leaf}
 \end{cases}
\]
and 
\[
\omega |_\Oo = 
\begin{cases}
- d_{D} (\theta_\Oo \circ \sigma) & \text{if $\Oo$ is a precontact leaf} \\
\sigma^\ast \underline \omega{}_\Oo & \text{if $\Oo$ is a lcps leaf}.
 \end{cases}
\]
In its turn, the pair $(\Ii, \omega)$ determines a distribution $ \mathfrak L $ in $\D L$ by
\[
\mathfrak L := \{ (\Delta , \psi) : \Delta \in \Ii \text{ and } i_\Delta \omega = \psi |_{\Ii} \}.
\]
If $ \mathfrak L $ is a smooth subbundle, we call $\Ff$ a (\emph{smooth}) \emph{precontact}/\emph{lcps foliation}. In this case, $ \mathfrak L $ is a Dirac-Jacobi structure and $\Ff$ is its characteristic foliation. So Dirac-Jacobi bundles are equivalent to precontact/lcps foliations. 
\end{remark}

\section{On the local structure of Dirac-Jacobi bundles}\label{sec:loc_struct}
In this section we study the local structure of Dirac-Jacobi bundles around both precontact, and lcps points. As a corollary we show that the parity of the dimension of precontact (resp.~lcps) leaves is locally constant (Corollary \ref{cor:parity}). 
We also discuss the \emph{transverse structures} to characteristic leaves (Propositions \ref{prop:transv} and \ref{prop:transv_uniq}). 
Our analysis parallels that of Dufour and Wade for Dirac manifolds \cite{DW2008}.

Let $L \to M$ be a line bundle. Moreover, let $(z^i)$ be coordinates on $M$ and let $\mu$ be a local basis of $\Gamma (L)$. Recall that $\Gamma (J^1 L)$ is locally generated by $(dz^i \otimes \mu, j^1 \mu)$. Here we regard $T^\ast M \otimes L$ as a vector subbundle in $J^1 L$ via the embedding $T^\ast M \otimes L \INTO J^1 L$ adjoint to the symbol map $\sigma : D L \to TM$ (see Remark \ref{rem:char}). The adjoint basis of $(dz^i \otimes \mu, j^1 \mu)$ in $\Der L$ is $(\delta_i , \mathbb{1})$, where $\delta_i (f \mu) := \frac{\partial f}{\partial z^i}  \mu$, for all $f \in C^\infty (M)$.

\begin{theorem}\label{theor:local}
Let $(L, \mathfrak L)$ be a Dirac-Jacobi bundle over a smooth manifold $M$, let $x_0 \in M$, and let $\Oo$ be the leaf of the characteristic foliation $\Ff_{\mathfrak L} $ through $x_0$.

If $\Oo$ is lcps, then, locally around $x_0$, there are coordinates $(x^i, y^a)$ on $M$, a generator $\mu$ of $\Gamma (L)$, and there is a basis $( \alpha_i , \beta^a , \beta)$ of $\Gamma (\mathfrak L )$, where $i = 1, \ldots, \dim \Oo$, and $a = 1, \ldots, \operatorname{codim} \Oo$, such that 
\begin{equation}\label{eq:norm_form_lcps}
\begin{aligned}
\alpha_i & = \left( \delta_i + E_i^a \delta_a + E_i \mathbb{1} ,\ F_{ij} dx^j \otimes \mu \right), \\
\beta^a & = \left( G^{ab} \delta_b + G^a \mathbb{1}, \ (dy^a - E_{i}^a dx^i) \otimes \mu \right), \\
\beta & =  \left( - G^{b} \delta_b, \ j^1 \mu - E_{i} dx^i \otimes \mu \right),
\end{aligned}
\end{equation}
with $F_{ij} + F_{ji} = G^{ab} + G^{ba} = 0$.

On the other hand, if $\Oo$ is precontact, then, locally around $x_0$, there are coordinates $(x^i, y^a)$ on $M$, a generator $\mu$ of $\Gamma (L)$, and there is a basis $( \alpha_i , \alpha, \beta^a)$ of $\Gamma (\mathfrak L )$, where $i = 1, \ldots, \dim \Oo$, and $a = 1, \ldots, \operatorname{codim} \Oo$, such that 
\begin{equation}\label{eq:norm_form_precont}
\begin{aligned}
\alpha_i & = \left( \delta_i + E_i^a \delta_a ,\ F_{ij} dx^j \otimes \mu - F_i j^1 \mu \right), \\
\alpha & =  \left( \mathbb{1} + E^a \delta_a , \ F_{i} dx^i \otimes \mu \right),
\\
\beta^a & = \left( G^{ab} \delta_b, \ (dy^a - E_{i}^a dx^i) \otimes \mu -E^a j^1 \mu \right),
\end{aligned}
\end{equation}
with $F_{ij} + F_{ji} = G^{ab} + G^{ba} = 0$.
\end{theorem}

\begin{proof}
The proof is an adaptation of the proof of \cite[Theorem 3.2]{DW2008}. Let $\Oo$ be a lcps leaf. Choose coordinates $(x^i, y^a)$ on $M$ around $x_0$ such that $\Oo = \left\{ y^a = 0 \right\}$, and let $\mu$ be a local generator of $\Gamma (L)$ around $x_0$ such that $\mu |_\Oo$ is constant with respect to the flat connection $\nabla^\Oo$ induced by $ \mathfrak L $ in $L|_\Oo$ (see Proposition~\ref{prop:char_fol}). In particular, $\Ii_{\mathfrak L}|_\Oo$ is locally generated by $\delta_i |_\Oo$. So, there is a local basis of $\Gamma (\mathfrak L )$, around $x_0$, of the form
\[
\begin{aligned}
\overline{\overline{\alpha}}_i & = ( \delta_i + \overline{\overline{\Delta}}_i ,\ \overline{\overline{k}}_i ), \\
\overline{\overline{\beta}}{}^a & = ( \overline{\overline{\square}}{}^a, \ \overline{\overline{h}}{}^a ), \\
\overline{\overline{\beta}} & =  ( \overline{\overline{\square}}, \ \overline{\overline{h}} ),
\end{aligned}
\]
with $\overline{\overline\Delta}_i |_\Oo = \overline{\overline\square}{}^a |_\Oo = \overline{\overline\square} |_\Oo = 0$. In particular,
\[
\delta_i + \overline{\overline{\Delta}}_i = (\delta_i^j + \Delta_i^j)\delta_i + \text{a linear combination of $(\delta_a, \mathbb 1)$},
\]
for some smooth functions $\Delta_i^j$ (here $\delta_i^j$ is the Kronecker symbol). Since $\Delta_i^j (x_0) = 0$, matrix $\| \delta_i^j + \Delta_i^j \|$ is invertible around $x_0$. Let $\|A_j^i \|$ be its inverse, and put $\overline{\alpha}_i := A_i^j \overline{\overline{\alpha}}_j$. Then $\overline{\alpha}_i$ is of the form 
\[
\overline{\alpha}_i = (\delta_i + \text{a linear combination of $(\delta_a, \mathbb{1})$ vanishing on $\Oo$},\ \overline{k}_i ).
\]
On the other hand, $\overline{\overline\square}{}^a, \overline{\overline\square}$ are of the form 
\[
\begin{aligned}
\overline{\overline{\square}}{}^a & = B^{ai} \delta_i + \text{a linear combination of $(\delta_a, \mathbb 1)$} ,\\
\overline{\overline{\square}} & = B^i \delta_i + \text{a linear combination of $(\delta_a, \mathbb 1)$} .
\end{aligned}
\]
Put $\overline{\beta}{}^a = \overline{\overline{\beta}}{}^a - B^{ai} \overline{\alpha}_i $, and $\overline{\beta} = \overline{\overline{\beta}} - B^{i} \overline{\alpha}_i $ . Then $\overline{\beta}{}^a, \overline{\beta}$ are of the form
\[
\begin{aligned}
\overline{\beta}{}^a & = (\text{a linear combination of $(\delta_a, \mathbb{1})$ vanishing on $\Oo$},\ \overline{h}{}^a ), \\
\overline{\beta} & = (\text{a linear combination of $(\delta_a, \mathbb{1})$ vanishing on $\Oo$},\ \overline{h} ).
\end{aligned}
\]
Moreover, since $ \mathfrak L $ is isotropic,
$
\langle \delta_i,  \overline{h}{}^a \rangle|_\Oo = \langle \delta_i,  \overline{h} \rangle|_\Oo = 0
$. System $(\overline{\alpha}_i, \overline{\beta}{}^a, \overline{\beta})$ is a local basis of $\Gamma (\mathfrak L )$ around $x_0$. Its representative matrix in the basis $(\delta_i, \delta_a, \mathbb 1, dx^i \otimes \mu, dy^a \otimes \mu, j^1 \mu)$ of $\Gamma (\D L)$ is of the form:
\[
\begin{tabular}{ccccc}
& $\delta_x$ & $\delta_y, \mathbb 1$ & $dx \otimes \mu$ & $dy \otimes \mu,  j^1 \mu$ \\
${\overline{\alpha}}$ & \multicolumn{1}{||c}{$\mathbb{Id}$} & $\mathbb{A}$ &
{$\ast$} &\multicolumn{1}{c||}{$\ast$}\\
$\overline{\beta}$ & \multicolumn{1}{||c}{$\mathbb{0}$} & $\mathbb{B} $  &
{$\mathbb{C} $} &  \multicolumn{1}{c||}{$\mathbb{D}$}
\end{tabular}\ ,
\] 
where $\mathbb{Id}$ is the identity matrix, and $\mathbb{A}, \mathbb{B}, \mathbb{C}$ vanish at $x_0$. It follows that matrix $\mathbb{D}$ is invertible at $x_0$, hence it is invertible around $x_0$. Put 
\[
\left( \begin{array}{c} 
\beta^a \\
\beta
\end{array}
\right)
= \mathbb D ^{-1} \left( \begin{array}{c} 
\overline{\beta}{}^a \\
\overline{\beta}
\end{array}
\right).
\]
Finally, let $\overline{k}_i$ be locally given by 
\[
\overline{k}_i = (\overline{k}_{ij} dx^j + \overline{k}_{ia} dy^a ) \otimes \mu + \overline{K}_i j^1 \mu,
\]
and put 
\[
\alpha_i = \overline{\alpha}_i - \overline{k}_{ia} \beta^a -\overline{K}_{i} \beta.
\]
System $(\alpha_i, {\beta}{}^a, {\beta})$ is a local basis of $\Gamma (\mathfrak L )$ around $x_0$. Its representative matrix is of the form:
\[
\begin{tabular}{ccccc}
& $\delta_x$ & $\delta_y , \mathbb 1$ & $dx \otimes \mu$ & $dy\otimes \mu ,  j^1 \mu$ \\
${\alpha}$ & \multicolumn{1}{||c}{$\mathbb{Id}$} & $\mathbb{E} $ &
{$\mathbb{F}$} &\multicolumn{1}{c||}{$\mathbb{0}$}\\
${\beta}$ & \multicolumn{1}{||c}{$\mathbb{0}$} & $\mathbb{G}$  &
{$\mathbb{H} $} &  \multicolumn{1}{c||}{$\mathbb{Id}$}
\end{tabular}\ .
\] 
Since $ \mathfrak L $ is isotropic, 
\[
\mathbb F + \mathbb F^t = \mathbb G + \mathbb G^t = \mathbb H + \mathbb E^t = 0,
\]
where $(-)^t$ denotes transposition. This concludes the proof of the first part of the statement.

The proof of the second part is very similar. We report it here for completeness. Let $\Oo$ be a precontact leaf. Choose again coordinates $(x^i, y^a)$ on $M$, around $x_0$, such that $\Oo = \left\{ y^a = 0 \right\}$, and let $\mu$ be any local generator of $\Gamma (L)$ around $x_0$. Thus, there is a local basis of $\Gamma (\mathfrak L )$, around $x_0$, of the form
\[
\begin{aligned}
\overline{\overline{\alpha}}_i & = ( \delta_i + \overline{\overline{\Delta}}_i ,\ \overline{\overline{k}}_i ),  \\
\overline{\overline{\alpha}} & =  ( \mathbb{1} + \overline{\overline{\Delta}}, \ \overline{\overline{k}} ), \\
\overline{\overline{\beta}}{}^a & = ( \overline{\overline{\square}}{}^a, \ \overline{\overline{h}}{}^a ),
\end{aligned}
\]
with $\overline{\overline\Delta}_i |_\Oo = \overline{\overline\square}{}^a |_\Oo = \overline{\overline\Delta} |_\Oo = 0$. In particular,
\[
\left( \begin{array}{c}
\delta_i + \overline{\overline{\Delta}}_i \\
\mathbb 1 + \overline{\overline{\Delta}}
\end{array} \right)  = (\mathbb{Id} + \mathbb{M} )\left( \begin{array}{c}
\delta_i  \\
\mathbb 1 
\end{array} \right)  + \text{a linear combination of $(\delta_a)$},
\]
for some function matrix $\mathbb M$. Since $\mathbb M(x_0) = 0$, matrix $\mathbb{Id} + \mathbb M$ is invertible around $x_0$. Put 
\[
\left( \begin{array}{c}
\overline{\alpha}_i \\
\overline{\alpha}
\end{array} \right)  =  (\mathbb{Id} + \mathbb M)^{-1 }\left( \begin{array}{c}
\overline{\overline{\alpha}}_i \\
\overline{\overline{\alpha}}
\end{array} \right) .
\]
Then $\overline{\alpha}_i, \overline{\alpha}$ are of the form
\[
\begin{aligned}
\overline{\alpha}{}_i & = ( \delta_i + \text{a linear combination of $\delta_a$} ,\  \overline{k}_i),\\
\overline{\alpha} & = (\mathbb 1 + \text{a linear combination of $\delta_a$},\  \overline{k}) .
\end{aligned}
\]
On the other hand, the $\overline{\overline{\square}}{}^a$ are of the form 
\[
\overline{\overline{\square}}{}^a  = B^{ai} \delta_i + B^a \mathbb 1 + \text{a linear combination of $\delta_a$}.
\]
Put $\overline{\beta}{}^a = \overline{\overline{\beta}}{}^a - B^{ai} \overline{\alpha}_i - B^a \overline{\alpha}$. Then the $\overline{\beta}{}^a$ are of the form
\[
\overline{\beta}{}^a = (\text{a linear combination of $\delta_a$ vanishing on $\Oo$},\ \overline{h}{}^a ).
\]
Moreover, since $ \mathfrak L $ is isotropic,
$
\langle \delta_i,  \overline{h}{}^a \rangle|_\Oo = 0
$. System $(\overline{\alpha}_i, \overline{\alpha}, \overline{\beta}{}^a)$ is a local basis of $\Gamma (\mathfrak L )$ around $x_0$. Its representative matrix is of the form:
\[
\begin{tabular}{ccccc}
& $\delta_x , \mathbb 1$ & $\delta_y$ & $dx \otimes \mu , j^1 \mu$ & $dy \otimes \mu$ \\
$\overline{\alpha}$ & \multicolumn{1}{||c}{$\mathbb{Id}$} & $\mathbb{A}$ &
{$\ast$} &\multicolumn{1}{c||}{$\ast$}\\
$\overline{\beta}$ & \multicolumn{1}{||c}{$\mathbb{0}$} & $\mathbb{B} $  &
{$\mathbb{C} $} &  \multicolumn{1}{c||}{$\mathbb{D}$}
\end{tabular}\ ,
\] 
where $\mathbb{A}, \mathbb{B}, \mathbb{C}$ vanish at $x_0$. It follows that matrix $\mathbb{D}$ is invertible at $x_0$, hence it is invertible around $x_0$. Let $\| \Delta^b_a \|$ be its inverse and put $\beta^a = \Delta^a_b \beta^b$. Finally, let $\overline{k}_i, \overline k$ be locally given by 
\[
\begin{aligned}
\overline{k}_i & = (\overline{k}_{ij} dx^j + \overline{k}_{ia} dy^a ) \otimes \mu + \overline{K}_i j^1 \mu, \\
\overline{k} & = (\overline{K}{}'_{j} dx^j + \overline{k}_{a} dy^a ) \otimes \mu + \overline{K} j^1 \mu, 
\end{aligned}
\]
and put 
\[
\begin{aligned}
\alpha_i & = \overline{\alpha}_i - \overline{k}_{ia} \beta^a, \\
\alpha & = \overline{\alpha} - \overline{k}_{a} \beta^a.
\end{aligned}
\]
System $(\alpha_i, \alpha, {\beta}{}^a)$ is a local basis of $\Gamma (\mathfrak L )$ around $x_0$. Its representative matrix is of the form:
\[
\begin{tabular}{ccccc}
& $\delta_x , \mathbb 1$ & $\delta_y $ & $dx \otimes \mu  ,  j^1 \mu$ & $dy\otimes \mu$ \\
${\alpha}$ & \multicolumn{1}{||c}{$\mathbb{Id}$} & $\mathbb{E} $ &
{$\mathbb{F}$} &\multicolumn{1}{c||}{$\mathbb{0}$}\\
${\beta}$ & \multicolumn{1}{||c}{$\mathbb{0}$} & $\mathbb{G}$  &
{$\mathbb{H} $} &  \multicolumn{1}{c||}{$\mathbb{Id}$}
\end{tabular}\ .
\] 
Since $ \mathfrak L $ is isotropic, 
\[
\mathbb F + \mathbb F^t = \mathbb G + \mathbb G^t = \mathbb H + \mathbb E^t = 0.
\]
This concludes the proof.
\end{proof}

\begin{corollary}\label{cor:point}
Let $(L, \mathfrak L)$ be a Dirac-Jacobi bundle over a manifold $M$ and let $\Oo = \{x_0\} \subset M$ be a zero-dimensional lcps leaf of the characteristic foliation. Then, around $x_0$, $\mathfrak L = \mathfrak L_J$ for some (local) Jacobi bracket $J : \Gamma (L) \times \Gamma (L) \to \Gamma (L)$, vanishing at $x_0$ (see Example \ref{ex:Jacobi}).
\end{corollary}

\begin{proof}
Since $\dim \Oo = 0$, it follows from Theorem \ref{theor:local} (\ref{eq:norm_form_lcps}) that, locally\linebreak  around $x_0$, there are coordinates $(y^a)$, a generator $\mu$ of $\Gamma (L)$, and a basis $(\beta^a, \beta)$ of $\Gamma (\mathfrak L)$, $a = 1, \ldots, \dim M$, such that
\[
\begin{aligned}
\beta^a & = \left( G^{ab} \delta_b + G^a \mathbb{1},\  dy^a \otimes \mu \right), \\
\beta & =  \left( - G^{b} \delta_b,\  j^1 \mu \right).
\end{aligned}
\]
Hence $\mathfrak L \cap DL = 0$. This concludes the proof.
\end{proof}

\begin{corollary}\label{cor:point_2}
Let $(L, \mathfrak L)$ be a Dirac-Jacobi bundle over a manifold $M$ and let $\Oo = \{x_0\} \subset M$ be a zero-dimensional precontact leaf of the characteristic foliation. Then, around $x_0$, after trivializing $L$, we have $\mathfrak L = \mathfrak L_{(\pi, Z)}$ for some (local) homogeneous Poisson structure $(\pi, Z)$ such that both $\pi$ and $Z$ vanish at $x_0$ (see Example \ref{ex:hom_Poisson}).
\end{corollary}

\begin{proof}
Since $\dim \Oo = 0$, it follows from Theorem \ref{theor:local} (\ref{eq:norm_form_precont}) that, locally\linebreak  around $x_0$, there are coordinates $(y^a)$, a generator $\mu$ of $\Gamma (L)$, and a basis $(\alpha, \beta^a)$ of $\Gamma (\mathfrak L)$, $a = 1, \ldots, \dim M$, such that
\[
\begin{aligned}
\alpha & = (\mathbb{1} + E^a \delta_a , \ 0), \\
\beta^a & = \left( G^{ab} \delta_b,\  dy^a \otimes \mu - E^a j^1 \mu \right).
\end{aligned}
\]
Now, $\mathfrak L \cap DL$ is generated by $\mathbb{1} + E^a \delta_a$ which is transversal to the subbundle generated by the $\delta_a$'s. Upon using $\mu$ to identify $L$ with $\R_M$ around $x_0$, the assertion follows from Proposition \ref{prop:hom_Poisson}.
\end{proof}


\begin{corollary}\label{cor:parity}
Let $(L, \mathfrak L)$ be a Dirac-Jacobi bundle over a manifold $M$. Then the parity of the dimension of precontact (resp.~lcps) leaves of the characteristic foliation $\Ff_{\mathfrak L} $ is locally constant on $M$.
\end{corollary}
\begin{proof}
Pick a point $x_0 \in M$ and let $x$ be a point in a connected neighborhood of $x_0$ where $ \mathfrak L $ has one of the forms given by Theorem \ref{theor:local}. Let $\Oo$ (resp.~$\Oo'$) be the characteristic leaf through $x_0$ (resp.~$x$). If $\Oo$ and $\Oo^\prime$ are both lcps then $ \mathfrak L _x$ is given by (\ref{eq:norm_form_lcps}) and $(\Ii_{\mathfrak L})_x = \nabla^{\Oo'} (T_x \Oo')$ is spanned by
\[
\delta_i + E_i^a \delta_a + E_i \mathbb 1, \ G^{ab} \delta_b + G^a \mathbb 1, \ -G^b \delta_b.
\]
Hence
\[
\dim \Oo' = \rk_x\! \Ii_{\mathfrak L} = \rk_{x_0}\Ii_{\mathfrak L} + \rk_x \mathbb G = \dim \Oo + \rk_x \mathbb G.
\]
Since $\mathbb G$ is a skew-symmetric matrix, it follows that the parity of $\dim \Oo'$ and  $\dim \Oo$ is the same. The precontact case can be discussed in a similar way.
\end{proof}

\begin{remark}
Locally, the dimension of a lcps leaf and the dimension of a precontact leaf have different parities. This can be proved either in a similar way as in the proof of Corollary \ref{cor:parity}, or exploiting the \emph{Dirac-ization trick} (see Remark \ref{rem:opposite} in the Appendix).
\end{remark}

Theorem \ref{theor:local} shows also that the regularity of the characteristic foliation on one side and that of distribution $\Ii_{\mathfrak L}$ on the other side are (partially) intertwined. Namely, notice, first of all, that, when the smooth distribution $\Ii_{\mathfrak L}$ is regular, $\Ff_{\mathfrak L} $ is not necessarily a regular foliation as the following example shows.

\begin{example}\label{ex:1}
Let $M = \R$, $L = \R_M$, and $ \mathfrak L  = V \oplus V^0$ (Example \ref{ex:inv}), with $V \subset D L = T\R \times \R$ the smooth, involutive, vector subbundle generated by $ x \partial / \partial x + 1 $. Then $\Ii_{\mathfrak L} = V$ which is a vector subbundle, hence a regular distribution, but $T \Ff_{\mathfrak L}  = \sigma (\Ii_{\mathfrak L}) = \langle x \partial / \partial x \rangle$ which has rank zero for $x = 0$, and rank one otherwise.
\end{example}

However, when $\Ff_{\mathfrak L} $ is a regular foliation, then $\Ii_{\mathfrak L}$ is necessarily a regular distribution.

\begin{corollary}
Let $(L, \mathfrak L)$ be a Dirac-Jacobi bundle over a smooth manifold $M$. If $\Ff_{\mathfrak L} $ is a regular foliation, then $\Ii_{\mathfrak L}$ is a regular distribution. In particular, it is a vector bundle over each connected component of $M$.
\end{corollary}

\begin{proof}
We need to prove that the rank of $\Ii_{\mathfrak L}$ is locally constant. First prove that $\rk \Ii_{\mathfrak L} $ is constant around every precontact point. Thus, let $\Oo$ be a precontact leaf of the characteristic distribution, $x_0 \in \Oo$, and let $x$ be a point in a connected neighborhood $U$ of $x_0$ where $ \mathfrak L $ has the form (\ref{eq:norm_form_precont}). Then $(\Ii_{\mathfrak L})_x$ is spanned by
\[
\delta_i + E_i^a, \ \mathbb 1 + E^a \delta_a, \ G^{ab} \delta_b.
\]
In particular, $\rk_x \Ii_{\mathfrak L} = \dim \Oo + 1 + \rk_x \mathbb G$, and it can only change in $U$ if the rank of $\mathbb G$ changes. On the other hand, $T_x \Ff_{\mathfrak L}  = \sigma (\Ii_{\mathfrak L})_x$ is spanned by
\[
\frac{\partial}{\partial x^i} + E^a_i \frac{\partial}{\partial y^a}, \ E^a  \frac{\partial}{\partial y^a}, \ G^{ab}  \frac{\partial}{\partial y^b},
\]
and there are two cases. Either $\mathbf E := \| E^a \|^t$ is in the image of $\mathbb G$, and in this case $\dim T_x \Ff_{\mathfrak L}  = \dim \Oo + \rk_x \mathbb G$, or $\mathbf E$ is not in the image of $\mathbb G$, and in this case $\dim T_x \Ff_{\mathfrak L}  = \dim \Oo + 1 + \rk_x \mathbb G$. Since $\dim T_x \Ff_{\mathfrak L} $ is locally constant, we conclude that $\rk_x \mathbb G$ can only change by one in $U$. However, $\mathbb G$ is skewsymmetric so that $\rk_x \mathbb G$ is even for all $x$. Hence $\rk_x \mathbb G$ is constant in $U$ and $\rk_x \Ii_{\mathfrak L}$ is constant in $U$ as well. In a similar way one can prove that $\rk \Ii_{\mathfrak L}$ is constant around every lcps point. Details are left to the reader.
\end{proof}

We conclude this section with one further application of Theorem \ref{theor:local}. Namely, we show that, similarly as for Poisson structures \cite{W1983}, \textcolor{black}{for Jacobi structures \cite{DLM}}, and for Dirac structures \cite{DW2008}, for every Dirac-Jacobi bundle over $M$, and every point $x_0 \in M$,
there is
\begin{enumerate}
\item \emph{a Jacobi structure transverse to $\Oo$}, if $x_0$ is a lcps point,
\item \emph{a homogeneous Poisson structure transverse to $\Oo$}, if $x_0$ is a precontact point,
\end{enumerate}
where $\Oo$ is the characteristic leaf through $x_0$.

\begin{proposition}\label{prop:transv}
Let $(L, \mathfrak L)$ be a Dirac-Jacobi bundle over a smooth manifold $M$, let $\Oo$ be a leaf of the characteristic foliation $\Ff_{\mathfrak L} $, and let $x_0 \in \Oo$. Additionally, let $Q \subset M$ be a submanifold transverse to $\Oo$ at $x_0$, with $\dim Q = \operatorname{codim} \Oo$, i.e.~$x_0 \in Q$ and $T_{x_0} M = T_{x_0} \Oo \oplus T_{x_0} Q$. Locally around $x_0$, 
\begin{enumerate}
\item if $\Oo$ is a lcps leaf, then $\mathfrak L$ induces a Jacobi bracket $J_Q : \Gamma (L|_Q) \times \Gamma (L|_Q) \to \Gamma (L|_Q)$ on the restricted line bundle $L|_Q$, with $(J_Q)_{x_0} = 0$;  
\item if $\Oo$ is a precontact leaf, then $\mathfrak L$, together with a trivialization of the restricted line bundle $L|_Q$, induces a homogeneous Poisson structure $(\pi_Q, Z_Q)$ on $Q$, with $(\pi_Q)_{x_0} = (Z_Q)_{x_0} = 0$.
\end{enumerate}
\end{proposition}

\begin{proof}
Let $x_0$ be a lcps \textcolor{black}{(resp.~precontact)} point. First we show that $\mathfrak L$ induces a Dirac-Jacobi structure on $\mathfrak L_Q$, at least around $x_0$. This is a consequence of Proposition \ref{prop:bw} below, and of Theorem \ref{theor:local}. Namely, from Example \ref{ex:bw}, it is enough to check that $\mathfrak L \cap (N^\ast Q \otimes L|_Q)$ has constant rank around $x_0$ (here $N^\ast Q$ is the \emph{conormal bundle} to $Q$). Thus, choose coordinates around $x_0$ as in Theorem \ref{theor:local} (\ref{eq:norm_form_lcps}) \textcolor{black}{(resp.~(\ref{eq:norm_form_precont}))} with the additional property that $Q = \{x^i = 0 \}$. Using (\ref{eq:norm_form_lcps}) \textcolor{black}{(resp.~(\ref{eq:norm_form_precont}))}, it is straightforward to check that $\mathfrak L \cap (N^\ast Q \otimes L|_Q) = 0$. Hence $\mathfrak L$ induces a Dirac-Jacobi structure on $\mathfrak L|_Q$ given by
\[
\mathfrak L_Q := \{ (\Delta, i_Q^\ast \psi) \in \D L|_Q : (\Delta, \psi) \in \mathfrak L \},
\]
where $i_Q : Q \INTO M$ is the inclusion (see Example \ref{ex:bw} (\ref{eq:bw_sub}) for more details). Finally, it is easy to see that $\{x_0\}$ is a lcps \textcolor{black}{(resp.~precontact)} leaf of the characteristic distribution of $(L|_Q, \mathfrak L_Q)$. The assertion now follows from Corollary~\ref{cor:point} \textcolor{black}{(resp.~\ref{cor:point_2})}.
\end{proof}


Similarly as for Poisson \cite{W1983}, Jacobi \cite{DLM} and Dirac \cite{DW2008} manifolds, the transverse structures described in Proposition \ref{prop:transv} are actually independent of the choice of $Q$, up to isomorphisms, and do only depend on the characteristic leaf $\Oo$, as explained by the following

\begin{proposition}\label{prop:transv_uniq}
Let $(L, \mathfrak L)$ be a Dirac-Jacobi bundle over a smooth manifold $M$, let $\Oo$ be a leaf of the characteristic foliation $\Ff_{\mathfrak L} $, and let $x_0, x'_0 \in \Oo$. Additionally, let $Q,Q' \subset M$ be submanifolds transverse to $\Oo$ at $x_0, x_0'$, respectively, with $\dim Q = \dim Q' = \operatorname{codim} \Oo$, i.e.~$x_0 \in Q$, $x_0' \in Q'$, and 
\[
T_{x_0} M = T_{x_0} \Oo \oplus T_{x_0} Q, \quad  T_{x'_0} M = T_{x'_0} \Oo \oplus T_{x'_0} Q'.
\]
\begin{enumerate}
\item if $\Oo$ is a lcps leaf, then there are neighborhoods $U,U'$ of $x_0, x_0'$ in $Q,Q'$, respectively, and a Jacobi isomorphism $(U, L|_U, J_Q) \simeq (U', L|_{U'}, J_{Q'})$, where $J_Q, J_{Q'}$ are the Jacobi structures induced by $\mathfrak L$ as in Proposition~\ref{prop:transv}.(1);
\item if $\Oo$ is a precontact leaf, and we fix a trivialization of $L$ in a neighborhood of both $x_0$ and $x'_0$, then there are neighborhoods $U,U'$ of $x_0, x_0'$ in $Q,Q'$, respectively, and an isomorphism of homogeneous Poisson manifolds $(U, \pi_Q, Z_Q) \simeq (U', \pi_{Q'}, Z_{Q'})$, where $(\pi_Q, Z_Q), (\pi_{Q'},Z_{Q'})$ are the homogeneous Poisson structures induced by $\mathfrak L$, and the fixed trivializations, as in Proposition \ref{prop:transv}.(2).
\end{enumerate}
\end{proposition}

\begin{proof} If $\dim \Oo = 0$ there is nothing to prove. So, assume $\dim \Oo > 0$. Then we can choose $x_0 \neq x_0'$.
\begin{enumerate}
\item Let $\Oo$ be a lcps leaf. Since $\Oo$ is connected there is no loss of generality if we assume that $x_0, x_0'$ are in the same neighborhood $V$ where $\mathfrak L$ has the form (\ref{eq:norm_form_lcps}). Additionally, it follows from the proof of Theorem \ref{theor:local}, that we can choose the coordinates $(x^i,y^a)$ so that, around $x_0$ and $x_0'$, $Q = \{ x^i = 0 \}$ and $Q' = \{ x^i = c^i \}$, respectively, where the $c^i$'s are constants. Consider the (local) derivations $\Delta_i := \pr_{D} (\alpha_i) = \delta_i + E_i^a \delta_a + E_i \mathbb 1\in \Gamma (DL)$, $i = 1, \ldots, \dim \Oo$. They correspond to (local) infinitesimal automorphisms of $L \to M$. Integrating, and composing their flows if necessary, we can construct a (local) automorphism $F$ of $L \to M$, mapping $L|_Q \to Q$ to $L|_{Q'} \to Q'$. It remains to show that $F : L|_Q \to L|_{Q'}$ is a Jacobi map, i.e.~it identifies $J_Q$ and $J_{Q'}$. To see this, consider the skewsymmetric bidifferential operator $J_V : \wedge^2 J^1 L|_V \to L|_V$ given by:
\[
J_V \left(j^1 \lambda, j^1 \nu \right)= \left( 2G^{ab} \frac{\partial f}{\partial y^a}   \frac{\partial g}{\partial y^b}  +G^a \left(f \frac{\partial g}{\partial y^a} - g \frac{\partial f}{\partial y^a} \right) \right) \mu,
\]
where $\lambda = f \mu$ and $\nu = g \mu$ for some local functions $f,g$. A direct computation shows that the conditions $\Upsilon_{\mathfrak L} (\beta^a, \beta^b, \beta^c) = 0$ are equivalent to $J_V$ being a Jacobi structure, and the conditions $\Upsilon_{\mathfrak L} (\beta^a, \beta^b, \alpha_i) = 0$ are equivalent to $\Delta_i$ being an infinitesimal Jacobi automorphism. It follows that $F : L|_V \to L|_V$ is a Jacobi map. Additionally $J_V$ restricts to both $Q$ and $Q'$ and the restrictions agree with $J_Q$ and $J_{Q'}$ respectively. Hence $F : L|_Q \to L|_{Q'}$ is a Jacobi map as well. 
\item First of all recall that, according to Dazord, Lichnerowicz and Marle \cite{DLM}, an isomorphism of homogeneous Poisson manifolds $(M,\pi, Z)$, $(M, \pi', Z')$ is a Poisson isomorphism $F : (M,\pi) \to (M', \pi')$ such that $F_\ast Z - Z'$ is a Hamiltonian vector field, i.e.~$F_\ast Z - Z' = (\pi')^\sharp (dh)$ for some function $h \in C^\infty (M')$. 

Now, go back to the statement and let $\Oo$ be a precontact leaf. Fix once for all a nowhere zero section $\mu$ of $L$ in a connected neighborhood containing both $x_0$ and $x_0'$. This is always possible (e.g., choose an embedded curve $\gamma$ connecting $x_0$ and $x_0'$ and a nowhere zero section $\mu'$ of $L$ along $\gamma$. Now extend $\mu'$ to a global section $\mu$ of $L$. It follows that $\mu$ is nowhere zero in a tubular neighborhood of $\gamma$). Similarly as above, we can restrict to the case when $x_0,x_0'$ are in the same neighborhood $V$ where $\mathfrak L$ has the form (\ref{eq:norm_form_precont}), and, additionally, $Q = \{ x^i = 0 \}, Q' = \{ x^i = c^i \}$, for some constants $c^i$. Consider the vector fields $X_i = (\sigma \circ \pr_D)(\alpha_i) = \frac{\partial}{\partial x^i} + E^a_i \frac{\partial}{\partial y^a}$, $i = 1, \ldots \dim \Oo$. Integrating them, and composing their flows if necessary, we can construct a (local) diffeomorphism $F : V \to V$ mapping $Q$ to $Q'$. It remains to show that $F: Q \to Q'$ is an isomorphism of homogeneous Poisson manifolds, i.e.~it identifies $\pi_Q$ and $\pi_{Q'}$, and, additionally, $F_\ast Z_Q -Z_{Q'} = \pi_{Q'}^\sharp (dh)$ for some function $h \in C^\infty (Q')$. To see this, consider the bivector $\pi_V$ and the vector field $Z_V$ on $V$ given by
\[
\pi_V = G^{ab} \frac{\partial}{\partial y^a} \wedge \frac{\partial}{\partial y^b}, \quad Z_V = - E^a \frac{\partial}{\partial y^a}.
\]
A direct computation shows that the conditions $\Upsilon_{\mathfrak L} (\beta^a, \beta^b, \beta^c) = 0$ are equivalent to $\pi_V$ being a Poisson bivector, and the conditions $\Upsilon_{\mathfrak L} (\beta^a, \beta^b, \alpha) = 0$ are quivalent to $\Ll_{Z_V} \pi_V = - \pi_V$. Hence $(\pi_V, Z_V)$ is a homogeneous Poisson structure on $V$. Additionally, $\Upsilon_{\mathfrak L} (\beta^a, \beta^b, \alpha_i) = 0$ are equivalent to $X_i$ being an infinitesimal Poisson isomorphism, and conditions $\Upsilon_{\mathfrak L} (\beta^a, \alpha, \alpha_i) = 0$ are equivalent to $\Ll_{X_i} Z_V = \pi_V^\sharp (dF_i)$. It follows that the flow of $X_i$ consists of automorphisms of the homogeneous Poisson manifold $(V, \pi_V, Z_V)$ so that $F : V \to V$ itself is an automorphism of $(V, \pi_V, Z_V)$. Finally, both $\pi_V$ and $Z_V$ restrict to $Q$ (resp.~$Q'$) and their restrictions agree with $\pi_Q$ and $Z_Q$ (resp.~$\pi_{Q'}$ and $Z_{Q'}$). Hence $F : Q \to Q'$ is an isomorphism of homogeneous Poisson manifolds.\vspace{-2em}
\end{enumerate}
\end{proof}

\section{Null distributions, admissible sections and admissible functions}\label{sec:null_distr}
\subsection{Null distributions of a Dirac-Jacobi bundle}\label{subsec:null_distr}
Let $L \to M$ be a line bundle and let $\theta : TM \to L$ be an $L$-valued precontact form on $M$. Denote by $K$ the null distribution of $\theta$ (see Remark \ref{rem:K_omega}). We make the following
\begin{assumption}\label{ass:simple}
Distribution $K$ is simple and $L$ is the pull-back, along the projection $\pi : M \to M_{\mathrm{red}}$, of a line bundle $L_{\mathrm{red}} \to M_{\mathrm{red}}$ over the leaf space $M_{\mathrm{red}}$ of $K$.
\end{assumption}
If, additionally, a certain cohomology class vanishes (see below), then $\theta$ descends to a unique contact form $ \theta_{\mathrm{red}} : T M_{\mathrm{red}} \to L_{\mathrm{red}}$ such that $\theta = \pi^\ast \theta_{\mathrm{red}}$. Contact manifold $(M_{\mathrm{red}}, \theta_{\mathrm{red}})$ is the \emph{contact reduction} $(M, \theta)$. All the above assumptions are always valid locally when $K$ is regular. So, morally, precontact forms are pull-backs of contact forms along submersions.

Similarly, let $(L, \nabla, \underline \omega)$ be a lcps structure on a smooth manifold $M$. Denote by $K := \ker \underline \omega{}_\flat$ the null distribution of $\underline \omega$. Make Assumption \ref{ass:simple} as above. If a certain cohomology class vanishes, then $\nabla = \pi^\ast \nabla_{\mathrm{red}} $, and $\underline\omega = \pi^\ast \underline \omega{}_{\mathrm{red}} $ for a unique lcs structure $(L_{\mathrm{red}}, \nabla_{\mathrm{red}}, {\underline \omega}{}_{\mathrm{red}})$ on $M_{\mathrm{red}}$. The lcs manifold $(M_{\mathrm{red}}, L_{\mathrm{red}}, \nabla_{\mathrm{red}}, \underline \omega{}_{\mathrm{red}})$ is the \emph{lcs reduction} of $(M, L, \nabla, \underline \omega)$. Again all the above assumptions are always valid locally when $K$ is regular. So, morally, lcps structures are pull-backs of lcs structures along submersions.

We can repeat the above discussion for each leaf of a precontact/lcps foliation and conclude that, morally, precontact/lcps foliations are pull-backs of contact/lcs foliations along submersions. In their turn precontact/lcps foliations are equivalent to Dirac-Jacobi bundles (see Remark \ref{rem:char_fol}), and contact/lcs foliations are equivalent to Jacobi bundles \cite{Kiri1976}. So, morally, Dirac-Jacobi bundles are pull-backs of Jacobi bundles along submersions. In this section, we clarify this claim (see, e.g., \cite[Section~4.3]{B2013} for the analogous situation in Dirac geometry).

Let $(L, \mathfrak L)$ be a Dirac-Jacobi bundle over a smooth manifold $M$. There are distinguished subdistributions $E_{\mathfrak L}  \subset \Ii_{\mathfrak L} \subset D L$, and $K_{\mathfrak L}  \subset T \Ff_{\mathfrak L}  \subset TM$, defined by
\[
E_{\mathfrak L}  := \mathfrak L  \cap D L, \quad \text{and} \quad K_{\mathfrak L}  := \sigma (E_{\mathfrak L} ).
\]
We call $E_{\mathfrak L} $ and $K_{\mathfrak L} $ the \emph{null der-distribution} and the \emph{null distribution} of $(L, \mathfrak L)$, respectively. This terminology is motivated by the following remark, which immediately follows from (\ref{eq:fol}). Let $\Oo$ be a characteristic leaf of $(L, \mathfrak L)$. If $\Oo$ is precontact, then $E_{\mathfrak L} |_\Oo = K_{\omega_\Oo}$ is the null distribution of the $2$-cocycle $\omega_\Oo $ in $( \Omega^\bullet_{L|_\Oo}, d_{D})$ corresponding to the precontact form $\theta_\Oo$ on $\Oo$. Hence $K_{\mathfrak L} |_\Oo$ is the null distribution of $\theta_\Oo$ (see Remark \ref{rem:K_omega}). On the other hand, if $\Oo$ is lcps, then $E_{\mathfrak L} |_\Oo = K_{\omega_\Oo} \cap \operatorname{im} \nabla^\Oo$, and $K_{\mathfrak L} |_\Oo = \ker (\underline \omega{}_\Oo)_\flat$ is the null distribution of the lcps structure $(L|_\Oo, \nabla^\Oo, \underline \omega_\Oo)$ on $\Oo$.

The following remark provides a way how to characterize Dirac-Jacobi structures in Examples \ref{ex:2form}-\ref{ex:Dirac}. 

\begin{remark}\label{rem:list}
It is straightforward to check the following.
\begin{itemize}
\item When $ \mathfrak L  = \mathfrak L _\omega$ for some $2$-cocycle $\omega$ in $(\Omega_L^\bullet, d_{D})$ (Example \ref{ex:2form}), then $\Ff_{\mathfrak L} $ consists of just one precontact leaf $\Oo = M$, and $\omega_\Oo = \omega$. Moreover, $E_{\mathfrak L}  = K_\omega$, and $K_{\mathfrak L} $ is the null distribution of the $L$-valued precontact form corresponding to $\omega$.
\item When $ \mathfrak L  = \mathfrak L _J$ for some Jacobi bracket $J : \Gamma (L) \times \Gamma (L) \to \Gamma (L)$ (Example \ref{ex:Jacobi}), then $\Ff_{\mathfrak L} $ is the contact/lcs characteristic foliation of the Jacobi bundle $(L,J)$. Hence $E_{\mathfrak L}  = 0$ and $K_{\mathfrak L}  = 0$.
\item When $ \mathfrak L  = V \oplus V^0$ for some involutive vector subbundle $V \subset D L$ (Example \ref{ex:inv}), then $V$ is a Lie algebroid and $\Ff_{\mathfrak L} $ is its characteristic foliation. Moreover, $\Ii_{\mathfrak L}|_\Oo = V|_\Oo$, and $\omega_\Oo = 0$ for every characteristic leaf $\Oo$. Hence $E_{\mathfrak L}  = V$ and $K_{\mathfrak L}  = T \Ff_{\mathfrak L} $. In particular, if $V$ is the image of a flat connection $\nabla : TM \to D L$ in $L$, then $\Ff_{\mathfrak L} $ consists of just one lcps leaf $\Oo = M$ and $\Ii_{\mathfrak L} = \operatorname{im} \nabla$. On the other hand, if $ \mathfrak L  = \mathfrak L _{\mathbb{1}}$, then $\Ff_{\mathfrak L} $ is a precontact foliation with zero dimensional leaves.
\item When $ \mathfrak L  = \mathfrak L _{\nabla, \underline \omega}$ for a lcps structure $(L, \nabla, \underline\omega)$ (Example \ref{ex:lcps}), then $\Ff_{\mathfrak L}  $ consists of just one lcps leaf $\Oo = M$, $\Ii_{\mathfrak L} = \operatorname{im} \nabla$, and $\omega_\Oo = \sigma^\ast \underline \omega$. Moreover $K_{\mathfrak L} $ is the null distribution of $\underline \omega$ and $E_{\mathfrak L} $ is the image via $\nabla$ of $K_{\mathfrak L} $. 
\item When $L = \R_M$ and $ \mathfrak L  = \mathfrak L_{(\pi, Z)}$ for some homogeneous Poisson structure $(\pi, Z)$ (Example \ref{ex:hom_Poisson}), then $T\Ff_{\mathfrak L} $ is spanned by $Z$ and the Hamiltonian vector fields with respect to $\pi$. In particular, the characteristic leaves are the flowouts of the symplectic leaves of $\pi$ along $Z$. For every symplectic leaf $\Pp$ of $\pi$, the vector field $Z$ is either everywhere tangent or everywhere transversal to $\Pp$. Accordingly, there are two kinds of characteristic leaves $\Oo$ of $\mathfrak L_{(\pi,Z)}$. Either $\Oo = \Pp$ for some symplectic leaf $\Pp$ of $\pi$ such that $Z$ is tangent to $\Pp$, or $\Oo$ is a disjoint union of a one-parameter family of symplectic leaves of $\pi$, and $Z$ is everywhere transversal to the symplectic leaves in the family (in particular, $Z|_\Oo$ is everywhere non null). In the first case, $\Oo$ is a precontact leaf (with respect to $\mathfrak L_{(\pi, Z)}$) and its precontact form $\theta : T\Oo \to \R_\Oo$ is $\theta = i_Z \omega$, where $\omega$ is the symplectic structure induced by $\pi$ on $\Oo$. In the second case, $\Oo$ is a lcps leaf. Specifically, $\nabla^\Oo$ is the connection in $\R_\Oo$ whose connection $1$-form $\eta := \nabla^\Oo 1$ is uniquely determined by the conditions 1) $\eta$ vanishes on symplectic leaves of $\pi$, and 2) $\eta (Z) = 1$. Moreover, the lcps form $\underline\omega{}_\Oo : \wedge^2 T \Oo \to \R_\Oo$, is uniquely determined by the conditions 1) $\underline\omega{}_\Oo$ agrees with the symplectic structures induced by $\pi$ on the symplectic leaves of $\pi$, and 2) $i_Z \underline\omega{}_\Oo = 0$. Finally $K_{\mathfrak L}$ is spanned by $Z$.
\item When $L = \R_M$ and $ \mathfrak L  = \widehat{\mathcal{L}}$ for some standard Dirac structure $\mathcal L$ (Example \ref{ex:Dirac}), then $\Ff_{\mathfrak L} $ is the presymplectic foliation of $\mathcal L$ \cite[Section~4.2]{B2013}. This means that all leaves $\Oo$ are lcps. In more details, $\nabla^\Oo$ is the trivial connection in $\R_\Oo$, $\underline\omega{}_\Oo$ is the presymplectic structure on $\Oo$ induced by $ \mathfrak L $, and $K_{\mathfrak L} |_\Oo$ is its null distribution, for every leaf $\Oo$.
\end{itemize}
\end{remark}

The rank of the null der-distribution $E_{\mathfrak L} $ is an upper-semicontinuous function. Hence, if $E_{\mathfrak L} $ is smooth, then it is also regular. On the other hand, the null distribution $K_{\mathfrak L} $ may well be smooth, without being regular, as Example \ref{ex:1} above shows. In that case $E_{\mathfrak L}  = \Ii_{\mathfrak L} = V$ which is smooth, while $K_{\mathfrak L}  = \sigma (V) = T \Ff_{\mathfrak L}  $ which is smooth but not regular. So, even when $E_{\mathfrak L} $ is a vector bundle, $K_{\mathfrak L} $ may fail to be regular. However, when $K_{\mathfrak L} $ is regular, then $E_{\mathfrak L} $ is necessarily regular, according to the following

\begin{proposition}\label{prop:KD}
Let $(L, \mathfrak L)$ be a Dirac-Jacobi bundle over a smooth manifold $M$. If the null distribution $K_{\mathfrak L} $ is regular, so is the null der-distribution $E_{\mathfrak L} $. In particular, both $K_{\mathfrak L} $ and $E_{\mathfrak L} $ are vector bundles over each connected component of $M$.
\end{proposition}

\begin{proof}
The statement follows from Theorem \ref{theor:local}. Recall that we defined $\Ii_{\mathfrak L} := \pr_{D} \mathfrak L$, and let $\omega : \wedge^2 \Ii_{\mathfrak L} \to L$, be the bilinear form induced by $ \mathfrak L $. Notice preliminarily that
\begin{equation}\label{eq:E_omega}
 E_{\mathfrak L}  = \{ \Delta \in \Ii_{\mathfrak L} :  i_\Delta \omega  = 0 \}. \\
\end{equation}
Now, first prove that $\rk E_{\mathfrak L} $ is constant around every point of a precontact characteristic leaf $\Oo$. Thus, let $x_0 \in \Oo$ and let $x$ be a point in a neighborhood $U$ of $x_0$ where $ \mathfrak L $ has the form (\ref{eq:norm_form_precont}) and, additionally, the rank of $K_{\mathfrak L} $ is constant in $U$. An easy computation shows that
\begin{equation}\label{eq:Dcap}
(E_{\mathfrak L} )_x = \{ X^i (\delta_i + E^a_i \delta_a) + X(\mathbb 1 + E^a \delta_a) : \mathbb{F}\cdot \mathbf{X} = 0\}
\end{equation}
where we denote $\mathbf X := \| X^i, X \|^t $. Hence,
\[
\rk_x E_{\mathfrak L}  = \dim \Oo + 1 - \rk_x \mathbb F.
\]
It follows from (\ref{eq:Dcap}) that
\begin{equation}\label{eq:sigma_Dcap}
(K_{\mathfrak L} )_x = \left\{ X^i \left( \frac{\partial}{\partial x^i} + E^a_i \frac{\partial}{\partial y^a} \right) + X E^a \frac{\partial}{\partial y^a} : \mathbb{F}\cdot \mathbf{X} = 0\right\}.
\end{equation}
Now, distinguish two cases:

\emph{Case I: $\mathbf{E} := \|E^a (x)\| \neq 0 $}. In this case, 
\[
\frac{\partial}{\partial x^i} + E^a_i \frac{\partial}{\partial y^a}, \quad \text{and} \quad E^a \frac{\partial}{\partial y^a}
\]
are linearly independent. Hence 
\[
\rk_x K_{\mathfrak L}  = \dim \Oo + 1 - \rk_x \mathbb F.
\]

\emph{Case II: $\mathbf{E} := \|E^a (x) \| = 0 $}. In this case
\[
(K_{\mathfrak L} )_x = \left\{ X^i \left( \frac{\partial}{\partial x^i} + E^a_i \frac{\partial}{\partial y^a} \right) : \mathbb{F}\cdot \mathbf{X} = 0\right\}.
\]
and there are two possibilities:
\begin{itemize}
\item\emph{Case IIa: $\mathbf{E} = 0 $ and $\mathbf{F} := \|F_i (x)\| \neq 0$}. Then the kernel of projection $\mathbf{X} \mapsto \|X^i\|$ does not intersect $\ker \mathbf F$ and 
\[
\rk_x K_{\mathfrak L}  = \dim \Oo + 1 - \rk_x \mathbb F,
\]
again. 
\item\emph{Case IIb: $\mathbf{E} = 0 $ and $\mathbf{F} := \|F_i (x)\| = 0$}. Then  $\ker \mathbf F$ contains the kernel of projection $\mathbf{X} \mapsto \|X^i\|$ and 
\[
\rk_x K_{\mathfrak L}  = \dim \Oo - \rk_x \mathbb F.
\]
\end{itemize}
Since $\mathbb F$ is a skew-symmetric matrix, its rank is even. So \emph{Case IIb} (on one side) and \emph{Case I or IIa} (on the other side) cannot occur simultaneously for two different points $x,x^\prime$ in $U$, if $\rk K_{\mathfrak L} $ is to be constant in $U$. Hence either
\[
\rk_x K_{\mathfrak L}  = \dim \Oo +1 - \rk_x \mathbb F = \rk_x E_{\mathfrak L} ,
\]
or
\[
\rk_x K_{\mathfrak L}  = \dim \Oo - \rk_x \mathbb F = \rk_x E_{\mathfrak L}  -1,
\]
for all $x \in U$. In any case, $\rk E_{\mathfrak L} $ is constant in $U$.

It remains to prove that $\rk E_{\mathfrak L} $ is constant around every point of a lcps chracteristic leaf $\Oo$. This case is simpler. Let $x_0 \in \Oo$ and let $x$ be a point in a neighborhood of $x_0$ where $ \mathfrak L $ has the form (\ref{eq:norm_form_lcps}), and, additionally, $\rk K_{\mathfrak L} $ is constant in $U$. Then 
\[
(E_{\mathfrak L} )_x = \{ X^i (\delta_i + E^a_i \delta_a + E_i \mathbb 1) : \mathbb F \cdot \mathbf X = 0 \},
\]
where $\mathbf X = \| X^i \|$. Hence $\rk_x E_{\mathfrak L}  = \dim \Oo - \rk_x \mathbb F$. On the other hand,
\[
(K_{\mathfrak L} )_x = \left\{ X^i \left(\frac{\partial}{\partial x^i} + E^a_i \frac{\partial}{\partial y^a} \right) : \mathbb F \cdot \mathbf X = 0 \right\}.
\] 
Hence $\rk_x K_{\mathfrak L}  = \dim \Oo - \rk_x \mathbb F = \rk_x E_{\mathfrak L} $. This concludes the proof.

\end{proof}

\begin{remark}\label{rem:KD}
Suppose that the null distribution $K_{\mathfrak L} $ is regular, so that the null der-distribution $E_{\mathfrak L} $ is regular as well. It is easy to see that \emph{Case IIb} in the proof of Proposition \ref{prop:KD} occurs precisely when $(\mathbb 1, 0) \in \Gamma (\mathfrak L )$ (use $\mathbf E = \mathbf F = 0$ in (\ref{eq:norm_form_precont})). Hence  the same proof shows that, when $K_{\mathfrak L} $ is regular, (in each connected component of $M$) there are only two possibilities:
\begin{itemize}
\item $(\mathbb 1_x, 0) \notin \mathfrak L_x$ for any $x \in M$, and $\rk E_{\mathfrak L}  = \rk K_{\mathfrak L} $, or
\item there are only precontact leaves, $(\mathbb 1, 0) \in \Gamma (\mathfrak L )$, and $\rk E_{\mathfrak L}  =\linebreak \rk K_{\mathfrak L}  + 1$.
\end{itemize}
When the second case occurs, then $\Ii_{\mathfrak L} \subset T^\ast M \otimes L = \langle \mathbb 1 \rangle^0$.
\end{remark}

\subsection{Jacobi reduction of Dirac-Jacobi bundles}\label{subsec:adm}
Let $(L, \mathfrak L)$ be a Dirac-Jacobi bundle over a smooth manifold $M$. In the following, we denote by $J_1 L \to M$ the bundle of first order differential operators $L \to \R_M$. It is the dual vector bundle of $J^1 L$, and $J_1 L \otimes L = D L$. Moreover, denote by $\sigma^\ast : T^\ast M \to (D L)^\ast = J^1 L \otimes L^\ast$ the dual map of the symbol $\sigma : D L \to TM$. In other words, $\sigma^\ast$ is obtained from the embedding $T^\ast M \otimes L \INTO J^1 L$ tensoring by $L^\ast$.
\begin{definition}
A (possibly local) section $\lambda \in \Gamma (L)$ is \emph{admissible} for the Dirac-Jacobi structure $ \mathfrak L $ if there exists $\Delta \in \Der L$ such that $(\Delta, j^1 \lambda) \in \Gamma (\mathfrak L) $. A derivation $\Delta$ like that is called \emph{Hamiltonian}. The space of admissible sections is denoted by $\Gamma_{\mathrm{adm}} (L)$. Similarly, a function $f \in C^\infty (M)$ is \emph{admissible} for $ \mathfrak L $ if there exists $F \in \Gamma (J_1 L)$ such that $(F, \sigma^\ast (df)) \in \Gamma (\mathfrak L  \otimes L^\ast) \subset \D L \otimes L^\ast = J_1 L \oplus (D L)^\ast$. A differential operator $F$ like that is called \emph{Hamiltonian}. The space of admissible functions is denoted by $C^\infty_{\mathrm{adm}} (M)$.
\end{definition}

\begin{remark}
The Hamiltonian derivation (resp.~differential operator) associated to an admissible section (resp.~function) is not uniquely determined. Namely, it can be changed by adding any (smooth) section of $E_{\mathfrak L} $ (resp.~$K_{\mathfrak L} $).
\end{remark}

\begin{proposition}\label{prop:Jac_mod}
The pair $(C^\infty_{\mathrm{adm}}(M), \Gamma_{\mathrm{adm}} (L))$ is a \emph{Jacobi module} in the sense that
\begin{enumerate}
\item $C^\infty_{\mathrm{adm}}(M)$ is a commutative, associative, unital algebra (actually a subalgebra in $C^\infty (M)$), and $\Gamma_{\mathrm{adm}} (L)$ is a $C^\infty_{\mathrm{adm}}(M)$-module,
\item $\Gamma_{\mathrm{adm}} (L)$ is a Lie algebra and $C^\infty_{\mathrm{adm}}(M)$ is a $\Gamma_{\mathrm{adm}} (L)$-module,
\item the action $\lambda \mapsto X_\lambda$ of $\Gamma_{\mathrm{adm}} (L)$ on $C^\infty_{\mathrm{adm}}(M)$ satisfies
\begin{equation}\label{eq:adm1}
X_\lambda (fg) = f X_\lambda (g) + g X_\lambda (f)
\end{equation}
for all $\lambda \!\in\! \Gamma_{\mathrm{adm}} (L)$, and $f,g \!\in\! C^\infty_{\mathrm{adm}}(M)$, i.e.~$\Gamma_{\mathrm{adm}} (L)$ acts on $C^\infty_{\mathrm{adm}}(M)$ by derivations,
\item the Lie bracket $\{-,-\}$ on $\Gamma_{\mathrm{adm}} (L)$ satisfies
\begin{equation}\label{eq:adm2}
\{ \lambda , f \mu \} = X_\lambda (f) \mu + f \{ \lambda , \mu \},
\end{equation}
for all $\lambda, \mu \in \Gamma_{\mathrm{adm}} (L)$, and $f \in C^\infty_{\mathrm{adm}}(M)$, i.e.~$\{-,-\}$ is a first order, differential operator with scalar-type symbol in each entry.
\end{enumerate}
\end{proposition}
\begin{proof}
The product of two admissible functions $f,g$ is admissible as well. Indeed, let $F, G : J^1 L \!\to\! \R_M$ be linear maps such that $(F, \sigma^\ast (df)), (G, \sigma^\ast (dg)) \in \Gamma (\mathfrak L  \otimes L^\ast)$. Since $\sigma^\ast d (fg) = f \sigma^\ast (dg) + g \sigma^\ast (df)$, then $(fG+gF, \sigma^\ast d(fg)) \in \Gamma (\mathfrak L  \otimes L^\ast)$ as well. Similarly, the product of an admissible function times an admissible section is an admissible section as well. Next, let $\lambda, \mu$ be admissible sections and let $f$ be an admissible function. Define
\[
\{\lambda, \mu \} := \Delta (\mu),
\]
with $(\Delta, j^1 \lambda) \in \Gamma (\mathfrak L) $. Well-posedness of $\{-,-\}$ follows from Lemma \ref{lem:adm}. Skew-symmetry follows from isotropy, and the Jacobi identity follows from involutivity of $ \mathfrak L $. Similarly, define
\[
X_\lambda (f) := \sigma (\Delta) (f).
\]
The same arguments as above show that correspondence $\lambda \mapsto X_\lambda$ is a well-defined Lie-algebra action. Identities (\ref{eq:adm1}), (\ref{eq:adm2}) are straightforward. Details are left to the reader.
\end{proof}

\begin{lemma}\label{lem:adm}
Let $\lambda$ (resp.~$f$) be an admissible section (resp.~function), and let $x \in M$. Then $\square (\lambda) = 0$ (resp.~$X (f) = 0$) for all $\square \in (E_{\mathfrak L} )_x$ (resp.~$X \in (K_{\mathfrak L} )_x$).
\end{lemma}
\begin{proof}
Let $\lambda$, $x$ and $\square$ be as in the statement, then $(\Delta, j^1 \lambda) \in \Gamma (\mathfrak L) $ for some $\Delta \in \Der L$ and $(\square, 0) \in \mathfrak L $. Hence $\square (\lambda) = \langle \square, j^1_x \lambda \rangle = \bla (\Delta_x, j^1_x \lambda), (\square, 0) \bra = 0$ by isotropy of $ \mathfrak L $. Similarly for admissible functions.
\end{proof}

When $K_{\mathfrak L} $ is a regular distribution, Lemma \ref{lem:adm} can be inverted giving the following

\begin{proposition}\label{prop:adm_null}
Let $K_{\mathfrak L} $ be a regular distribution. A section $\lambda$ of $L$ (resp.~a function $f$ on $M$) is admissible if and only if $\square (\lambda) = 0$ for all $\square \in \Gamma(E_{\mathfrak L} )$ (resp.~$X (f) = 0$ for all $X \in \Gamma(K_{\mathfrak L} )$).
\end{proposition}

\begin{proof}
The ``only if part'' of the statement follows from Lemma \ref{lem:adm}. For the ``if part'', recall that, in view of Proposition \ref{prop:KD}, when $K_{\mathfrak L} $ is regular, $E_{\mathfrak L} $ is regular as well. Hence, in view of the second of (\ref{eq:char}), $\pr_{J^1} (\mathfrak L )$ is a regular distribution in $J^1 L$. Now, $\square (\lambda) = 0$ for all $\square \in \Gamma (E_{\mathfrak L} )$ tells us that $j^1 \lambda$ is a section of the annihilator of $E_{\mathfrak L} $, and, from (\ref{eq:char}) again, there is a section $\alpha$ of $ \mathfrak L $ such that $\pr_{J^1} (\alpha) = j^1 \lambda$, whence $\alpha = (\Delta, j^1 \lambda)$ for some $\Delta \in \Der L$. Similarly for admissible functions.
\end{proof}

Now recall that, when $K_{\mathfrak L} $ is a regular distribution, then, in each connected component of $M$, either $(\mathbb 1, 0) \in \Gamma (\mathfrak L) $ and, in this case, $\rk E_{\mathfrak L}  = \rk K_{\mathfrak L}  +1$, or $(\mathbb 1_x, 0) \notin \mathfrak L_x$ for any $x \in M$, and, in this case, $\rk E_{\mathfrak L}  = \rk K_{\mathfrak L}  $ (Remark \ref{rem:KD}). Hence we have the following

\begin{corollary}\label{cor:adm_nabla}
Let $K_{\mathfrak L} $ be a regular distribution. In each connected component of $M$, if $(\mathbb 1, 0) \in \Gamma (\mathfrak L) $, then there are no admissible sections, otherwise there is a canonical $K_{\mathfrak L} $-connection $\nabla^K$ in $L$, and a section of $L$ is admissible if and only if it is $\nabla^K$-constant. 
\end{corollary}

\begin{proof}
If $(\mathbb 1, 0) \in \Gamma (\mathfrak L) $ then $\pr_{J^1} (\mathfrak L ) \subset T^\ast M \otimes L$ and there cannot be admissible sections. Otherwise $\rk E_{\mathfrak L}  = \rk K_{\mathfrak L} $, hence $E_{\mathfrak L} $ is the image of a flat $K_{\mathfrak L}$-connection $\nabla^K$ in $L$. The last part of the statement is obvious.
\end{proof}

\begin{example}
Let $ \mathfrak L  = \mathfrak L _{\mathbb 1}$. Then $E_{\mathfrak L}  = \langle \mathbb 1 \rangle$, and $K_{\mathfrak L}  = 0$. Hence every function is admissible but there are no admissible sections. This simple example shows that the situation here is slightly different from that in Dirac geometry. Namely, under similar regularity conditions, a Dirac structure does always possess admissible functions, while Dirac-Jacobi structures may fail to possess admissible sections.
\end{example}

\begin{remark}
Let $K_{\mathfrak L}$ be a regular distribution. The case $(\mathbb 1, 0) \in \Gamma (\mathfrak L)$ is exceptional and can be completely characterized as follows. First of all, if $(\mathbb 1, 0) \in \Gamma (\mathfrak L)$, then all characteristic leaves $\mathcal O$ are precontact, moreover $\mathbb 1 \in \Gamma (\mathcal I_{\mathfrak L})$ and Remark \ref{rem:theta=0} shows that $\theta_\Oo$, and hence $\omega_\Oo$, vanish. It follows that $K_{\mathfrak L} = T \mathcal F_{\mathfrak L}$. So $\mathfrak L = V \oplus V^0$ for some Lie subalgebroid $V \subset DL$ such that $\mathbb 1 \in \Gamma (V)$ (see Example \ref{ex:inv} and the third item in Remark \ref{rem:list}).
\end{remark}

Now on, in this section, we assume that $K_{\mathfrak L} $ is regular and $(\mathbb 1, 0) \notin \Gamma (\mathfrak L) $. Moreover, we make the following

\begin{assumption}\label{ass:simple_D}\ 
\begin{enumerate}
\item The distribution $K_{\mathfrak L} $ is \emph{simple}, i.e.~its leaf space $M_{\mathrm{red}}$ is a smooth manifold and the projection $\pi: M \to M_{\mathrm{red}}$ is a submersion.
\item The line bundle $L$ is isomorphic to the pull-back bundle $\pi^\ast L_{\mathrm{red}}$ for some line bundle $L_{\mathrm{red}} \to M_{\mathrm{red}}$.
\end{enumerate}
\end{assumption}

From Assumption \ref{ass:simple_D}.(1) and Proposition \ref{prop:adm_null}, admissible functions identify with functions on $M_{\mathrm{red}}$. Similarly, we would like to use Assumption~\ref{ass:simple_D}.(2) to identify admissible sections with sections of $L_{\mathrm{red}}$. However, for sections the situation is slightly more delicate. From Corollary \ref{cor:adm_nabla}, admissible sections are $\nabla^K$-constant sections of $L$. Denote by $\phi : L \to \pi^\ast L_{\mathrm{red}}$ the isomorphism in the Assumption \ref{ass:simple_D}.(2). Isomorphism $\phi$ induces another $K_{\mathfrak L} $-connection $\nabla{}_0^K$ in $L$ characterized by the property that $\nabla{}_0^K$-constant sections correspond to pull-back sections $\pi^\ast \lambda_{\mathrm{red}}$ via $\phi$, with $\lambda_{\mathrm{red}} \in \Gamma (L_{\mathrm{red}})$. Hence the composition of the pull-back followed by $\phi^{-1}$ identifies sections of $L_{\mathrm{red}}$ with $\nabla{}_0^K$-constant sections. In general, $\nabla^K \neq \nabla{}_0^K$ and the difference $A :=\nabla^K -  \nabla{}_0^K : K_{\mathfrak L}  \to D L$ is a $1$-form on $K_{\mathfrak L} $ with values in the vector bundle of endomorphisms of $L$. In its turn, the bundle of endomorphisms of $L$ is the trivial line bundle generated by $\mathbb 1$, hence $A$ can be regarded as a section of $K_{\mathfrak L} ^\ast$. Since both $\nabla^K$ and $\nabla{}_0^K$ are flat connections, $A$ is a $1$-cocycle in the de Rham complex of the Lie algebroid $K_{\mathfrak L} $:
\[
\xymatrix{0 \ar[r] & C^\infty (M) \ar[r]^-{d_{K_{\mathfrak L} }} & \Gamma (K_{\mathfrak L} ^\ast)  \ar[r]^-{d_{K_{\mathfrak L} }} & \Gamma (\wedge^2 K_{\mathfrak L} ^\ast)   \ar[r]^-{d_{K_{\mathfrak L} }} &  \cdots.}
\]
Assume that the cohomology class of $A$ in $\Gamma (\wedge^\bullet K_{\mathfrak L} ^\ast)$ vanishes. Then connection $\nabla^K$ is trivial in the following sense: \emph{there exists an(other) isomorphism $L \simeq \pi^\ast L_{\mathrm{red}}$ which identify $\nabla^K$-constant sections with pull-back sections}. To see this, let $A=d_{K_{\mathfrak L} } f$ for some function $f$. The required isomorphism is the composition $\psi : L \to \pi^\ast L_{\mathrm{red}}$ of $e^f : L \to L$, $\lambda \mapsto e^f \lambda$ followed by $\phi : L \to \pi^\ast L_{\mathrm{red}}$. Summarizing, the composition of the pull-back followed by $\psi^{-1}$ identifies sections of $L_{\mathrm{red}}$ with admissible sections, and Proposition \ref{prop:Jac_mod} reveals that the line bundle $L_{\mathrm{red}}$, equipped with the bracket $\{-,-\}$, is a Jacobi bundle. We have thus proved the following

\begin{proposition}
Let $(L, \mathfrak L)$ be a Dirac-Jacobi bundle with regular null distribution, and $(\mathbb 1, 0) \notin \Gamma (\mathfrak L) $. Under assumption \ref{ass:simple_D} and the condition $A = d_{K_{\mathfrak L} }f$ for some $f \in C^\infty (M)$, $L_{\mathrm{red}}$ is a Jacobi bundle and there is an isomorphism of line bundles $L \simeq \pi^\ast L_{\mathrm{red}}$ identifying sections of $L_{\mathrm{red}}$ with admissible sections of $L$ and the Jacobi bracket on $\Gamma (L_{\mathrm{red}})$ with the Jacobi bracket on $\Gamma_{\mathrm{adm}} (L)$.
\end{proposition}

In the generic case, i.e.~$K_{\mathfrak L} $ regular but not simple, both Assumption~\ref{ass:simple_D} and the assumption $A = d_{K_{\mathfrak L} } f$ are still valid locally, but they may fail to be valid globally. However, $M$ is globally equipped with a \emph{Jacobi structure tranverse to $K_{\mathfrak L} $} that we now describe.

Consider the short exact sequence
\begin{equation}\label{eq:transv_der}
0 \longrightarrow K_{\mathfrak L}  \overset{\nabla^K}{\longrightarrow} D L \longrightarrow D^\bot   L \longrightarrow 0. 
\end{equation}
where $D^\bot   L := D L / \operatorname{im} \nabla^K$. Sections of $D^\bot L$ should be interpreted as \emph{derivations of $L$ transverse to $K_{\mathfrak L} $}. Similarly, sections of the tensor product $J_1^\bot L := (D^\bot   L)^\ast \otimes L^\ast = J_1 L / (K_{\mathfrak L} \otimes L^\ast )$ should be interpreted as \emph{first order differential operators $\Gamma (L) \to C^\infty (M)$ transverse to $K_{\mathfrak L} $}.

Now, let $N$ be a submanifold of $M$ \emph{complementary to $K_{\mathfrak L} $}, i.e.~such that $T_x M = T_x N \oplus (K_{\mathfrak L} )_x$ for all $x \in N$. Any $\Delta \in \Gamma (D^\bot   L )$ can be \emph{restricted} to $N$ to give a genuine derivation of $L|_N$ as follows. Let $\Delta = \widetilde{\Delta} + \operatorname{im} \nabla^K$, with $\widetilde{\Delta} \in \Der L$. Moreover, let $\lambda$ be a section of $L|_N$ and let $\widetilde \lambda$ be a section of $L$ such that $\widetilde \lambda |_N = \lambda$, at least locally around a point $x \in N$. Put
\[
(\Delta|_N (\lambda))_x := (\widetilde{\Delta} (\widetilde{\lambda}))_x.
\] 
It is easy to see that $\Delta|_N : \Gamma (L|_N) \to \Gamma (L|_N)$ is a well defined derivation. Sections of $J_1^\bot L$ can be restricted to $N$ in a similar way.

\begin{proposition}\label{prop:transv_jac}
Let $(L, \mathfrak L)$ be a Dirac-Jacobi bundle with regular null distribution, and $(\mathbb 1, 0) \notin \Gamma (\mathfrak L) $. The Dirac-Jacobi structure $ \mathfrak L $ induces a canonical section $J^\bot$ of $\wedge^2 J_1^\bot L \otimes L$. Moreover, for every submanifold $N$ of $M$ complementary to $K_{\mathfrak L} $, $J^\bot$ restricts to a Jacobi bracket $\{-,-\}_N$ on $\Gamma (L|_N)$. 
\end{proposition}

\begin{proof}
From the proof of Corollary \ref{cor:adm_nabla}, the image of $\nabla_K$ is $E_{\mathfrak L} $. So, the dual of sequence (\ref{eq:transv_der}), tensorized by $L$, reads
\[
0 \longleftarrow K_{\mathfrak L} ^\ast \otimes L \longleftarrow J^1 L \longleftarrow E_{\mathfrak L} ^0 \longleftarrow 0.
\]
From (\ref{eq:char}), we have $E_{\mathfrak L} ^0 \!=\! \pr_{J^1} \mathfrak L$. Recall that $J_1^\bot L \!=\! J_1 L / (K_{\mathfrak L}  \otimes L^\ast) = (E_{\mathfrak L} ^0)^\ast$. Hence a section of $\wedge^2 J_1^\bot L \otimes L$ is the same as a morphism
\[
\wedge^2 (\pr_{J^1} \mathfrak L) \longrightarrow L.
\]
Thus, let $\varphi, \psi$ be sections of $\pr_{J^1}\mathfrak L$. There are $\Delta, \square \in \Der L$ such that\linebreak $(\Delta, \varphi), (\square, \psi) \in \Gamma(\mathfrak L )$. Put
\begin{equation}\label{eq:J_bot}
J^\bot (\varphi,\psi) := \langle \square, \psi \rangle.
\end{equation}
The morphism $J^\bot : (\pr_{J^1}\mathfrak L)^{\otimes 2} \to L$ is obviously well defined. Moreover, since $ \mathfrak L $ is isotropic, $J^\bot$ is skew-symmetric. This proves the first part of the Proposition.

For the second part, let $N \subset M$ be a submanifold complementary to $K_{\mathfrak L} $. Restrict $J^\bot$ to $N$ in the obvious way
to get a skew-symmetric, first order, bidifferential operator $\{-,-\}_N$ on sections of $L|_N$. Every section of $L|_N$ can be extended, locally around any point of $N$, to a $\nabla^K$-constant, hence admissible, local section of $L$, and
\[
\{ \lambda|_N ,  \mu|_N \}_N = \{ \lambda , \mu \}|_N
\]
for all admissible, local sections $\lambda, \mu$ of $L$ defined around $N$, where $\{-,-\}$ is the Jacobi bracket of admissible sections. The Jacobi identity for $\{-,-\}_N$ now follows from the Jacobi identity for $\{-,-\}$.
 \end{proof}
 
 \begin{remark}
Every submanifold $N \subset M$ complementary to $K_{\mathfrak L} $ can be locally understood as (the image of) a section of the projection $M \to M_{\mathrm{red}}$ onto the leaf space of $K_{\mathfrak L} $. The proof of Proposition \ref{prop:transv_jac} then shows that $\{-,-\}_N$ is the pull-back of the Jacobi bracket on $\Gamma (L_{\mathrm{red}})$.
\end{remark}

 \section{Morphisms of Dirac-Jacobi bundles} \label{sec:morphisms}
Let $(M,C)$ and $(M',C')$ be precontact manifolds. A \emph{morphism of precontact manifolds} $f : (M,C) \to (M',C')$ is a smooth map $f : M \to M'$ such that $df (C) \subset C'$. In particular, $df$ induces a morphism of line bundles $F : TM/C \to TM' / C'$ over $f$. As it will be clear from what follows, it is in practice useful demanding, additionally, that $F$ is a morphism in the category $\mathbf{VB}^{\mathrm{reg}}$ (see Section \ref{sec:line}), i.e.~it is an isomorphism on fibers (this choice excludes, for instance, the embedding of a line, with the zero rank distribution, as an integral manifold in a contact manifold, from morphisms of precontact manifolds). In terms of the precontact forms $\theta : TM \to TM/C$ and $\theta ' : TM' \to TM' /C$ a morphism of precontact manifolds is a regular morphism of line bundles $F : TM/C \to TM'/C'$ over a smooth map $f : M \to M'$ such that $\theta ' \circ df = F \circ \theta$. On the other hand, let $(L, \{-,-\})$ and $(L', \{-,-\}')$ be Jacobi bundles over manifolds $M$ and $M'$ respectively. A \emph{morphism of Jacobi bundles} or, shortly, a \emph{Jacobi morphism}, is a \emph{regular} morphism $F : L \to L'$ such that $\{F^\ast \lambda' , F^\ast \mu' \} = F^\ast \{ \lambda', \mu' \}'$ for all $\lambda', \mu' \in \Gamma (L')$. The two notions of morphism of precontact manifolds, and Jacobi morphism, do not agree on contact manifolds. Hence, similarly as for Dirac manifolds, there are two distinct notions of morphisms of Dirac-Jacobi bundles. What follows in this section parallels \cite[Section 5]{B2013} (see also \cite{BR2003}). However, morphisms of Dirac-Jacobi bundles exhibit a few novel features which make discussing them slightly more complicated than discussing morphisms of standard Dirac manifolds. 

\subsection{Backward Dirac-Jacobi morphisms}\label{subsec:bw}

Similarly as for Dirac manifolds, under suitable regularity conditions, a Dirac-Jacobi structure can be pulled-back along a regular morphism of line bundles. Namely, let $L \to M$ and $L' \to M'$ be line bundles, and let $F : L \to L'$ be a regular vector bundle morphism over a smooth map $\underline F : M \to M'$. Additionally, let $ \mathfrak L '$ be a Dirac-Jacobi structure on $L'$.

\begin{definition}
The \emph{backward image of $ \mathfrak L '$ along $F$} is the, non-necessarily regular, distribution in $\D L$ defined as:
\[
\mathfrak B_F (\mathfrak L ') := \{ (\Delta, F^\ast \psi ') : (F_\ast \Delta , \psi ') \in \mathfrak L ' \} \subset \D L.
\]
\end{definition}

\begin{definition}
A \emph{backward Dirac-Jacobi map} $F : (L, \mathfrak L) \to (L',\mathfrak L')$ between Dirac-Jacobi bundles is a regular morphism of line bundles $F : L \to L'$ such that $ \mathfrak L  = \mathfrak B_F (\mathfrak L ')$.
\end{definition}

\begin{remark}
Let $\Gamma_F$ be the distribution in $\D L \oplus \underline F^\ast (\D L')$ defined by
\[
\Gamma_F := \{ ((\Delta, \psi) , (\Delta', \psi ')) : \Delta' = F_\ast \Delta \text{ and } \psi = F^\ast \psi ' \}.
\]
Clearly, $\Gamma_F$ is the kernel of the smooth vector bundle epimorphism 
\[
\begin{aligned}
\D L \oplus F^\ast (\D L') & \longrightarrow J^1 L \oplus F^\ast (D L') \\
((\Delta,\psi) , (\Delta', \psi ')) & \longmapsto (\psi - F^\ast\psi ' , F_\ast\Delta - \Delta ')
\end{aligned}
\]
Hence it is a smooth vector bundle (of rank $\dim M + \dim M' + 2$). Notice that $\mathfrak B_F (\mathfrak L ')$ is the image of $\Gamma_F \cap \underline F^\ast \mathfrak L'$ under projection  $\D L \oplus \underline F^\ast (\D L') \to \D L$ onto the first summand. It is easy to see that the kernel of surjection $\Gamma_F \cap \underline F^\ast \mathfrak L' \to \mathfrak B_F (\mathfrak L ')$ consists of points in $\D L \oplus F^\ast (\D L')$ of the form $((0,0), (0, \psi '))$, with $(0,\psi ') \in \mathfrak L '$, and $F^\ast \psi ' = 0$. Summarizing, there is a short (point wise) exact sequence
\begin{equation}\label{eq:exact_bw}
0  \longrightarrow \ker j^1 F \cap \underline F^\ast \mathfrak L' \longrightarrow \Gamma_F \cap \underline F^\ast \mathfrak L'  \longrightarrow \mathfrak B_F (\mathfrak L ') \longrightarrow 0,
\end{equation}
where the second arrow is the inclusion $\underline F^\ast (\D L' ) \INTO \D L \oplus \underline F^\ast (\D L')$, and the third arrow is the projection $\D L \oplus \underline F^\ast (\D L') \to \D L$. Finally, $\mathfrak B_F (\mathfrak L ')$ is a maximal isotropic distribution in $\D L$. To see this, let $x \in M$. Upon selecting a generator $\mu$ in $L_x$ and the generator $F(\mu)$ in $L_{\underline F (x)}$, maximal isotropy of $\mathfrak B_F (\mathfrak L ')$ immediately follows from \cite[Proposition 5.1]{B2013}. In particular, $\mathfrak B_F (\mathfrak L ')$ has constant rank.
\end{remark}

\begin{proposition}\label{prop:bw}
If $\rk (\ker j^1 F \cap \underline F^\ast \mathfrak L')$ is constant, then the backward image $\mathfrak B_F (\mathfrak L ')$ is a vector subbundle of $\D L$. In this case, $(L, \mathfrak B_F (\mathfrak L '))$ is a Dirac-Jacobi bundle.
\end{proposition}

\begin{definition}
Let $L \to M$ and $L' \to M'$ be line bundles and let $F : L \to L'$ be a regular morphism of line bundles over a smooth map $\underline F : M \to M'$. A section $\alpha = (\Delta, \psi)$ of $\D L$ and a section $\alpha' = (\Delta', \psi')$ of $\D L'$ are \emph{$F$-related} if $\psi = F^\ast \psi'$ and, additionally, $\Delta$ and $\Delta'$ are $F$-related.
\end{definition}

\begin{proof}[Proof of Proposition \ref{prop:bw}]
The proof is an adaptation of the analogous proof for the Dirac case \cite[Proof of Proposition 5.9]{B2013}, and it is very similar to that one. We report it for completeness. Let $\rk (\ker j^1 F \cap \underline F^\ast \mathfrak L')$ be constant. Since $\rk \mathfrak B_F (\mathfrak L ')$ is also constant, then, from (\ref{eq:exact_bw}), $\rk \Gamma_F \cap \underline F^\ast \mathfrak L' $ is constant as well. Hence  $\Gamma_F \cap \underline F^\ast \mathfrak L' $ is a vector bundle, and so is $\mathfrak B_F (\mathfrak L ')$. It remains to prove that $\mathfrak B_F (\mathfrak L ')$ is involutive. To see this, it is useful to prove the following

\begin{lemma}\label{lem:bw}
Let $x_0 \in M$ be a point such that the rank of $dF$ and the rank of $\ker j^1 F \cap \underline F^\ast \mathfrak L'$ are constant around $x_0$. Then, for every $\alpha_0 \in \mathfrak B_F (\mathfrak L ')_{x_0}$, there exist a local section $\alpha $ of $\mathfrak B_F (\mathfrak L ')$ and a local section $\alpha^\prime $ of $ \mathfrak L '$ such that $\alpha_{x_0} = \alpha_0$ and, additionally, $\alpha$ and $\alpha'$ are $F$-related.
\end{lemma}
\begin{proof}[Proof of Lemma \ref{lem:bw}]
Let $x_0$ be as in the statement and let $r := \rk_{x_0} d \underline F$. Then we can choose coordinates $(x^1, \ldots, x^m)$ in $M$, centered in $x_0$, and coordinates $(y^1, \ldots, y^n)$ in $M'$, centered in $y_0 := \underline F (x_0)$, such that $\underline F$ looks like
\begin{equation}\label{eq:loc}
\underline F (x^1, \ldots, x^m) \mapsto (x^1, \ldots, x^r, 0, \ldots, 0),
\end{equation}
around $x_0$. Moreover, let $\mu '$ be a local generator of $\Gamma (L')$ around $y_0$. Finally let $S$ be the submanifold in $M$ defined by $x^{r+1} = \cdots = x^m = 0$ so that 1) $\underline F$ establishes a (local) diffeomorphism $\underline F : S \to S' = \operatorname{im} \underline F =\{ y^{r+1} = \cdots = y^n = 0 \}$, and 2) $F|_S : L|_S \to L'|_{S'}$ is an isomorphism over $\underline F : S \to S'$. In order to find $\alpha$ and $\alpha '$ as in the statement, first extend $\alpha_0$ to a local section $\alpha_S = (\Delta_S, \psi_S)$ of $\mathfrak B_F (\mathfrak L ')|_S$, and use the isomorphism $F|_S$ to find a section $\alpha'_S = (\Delta'_S, \psi '_S)$ of $ \mathfrak L '|_{S'}$ such that $(\psi_S)_x = F^\ast (\psi '_S)_{\underline F (x)} $ and $F_\ast (\Delta_x) = (\Delta_S')_{\underline F (x)}$ for all $x \in S$. Next, using (\ref{eq:loc}), extend $\Delta_S$ to a section $\Delta$ of $D L$ such that $F_\ast (\Delta_x) = (\Delta_S')_{\underline F (x)}$ for all $x$ in a neighborhood of $x_0$. Then $\alpha := (\Delta, F^\ast \psi'_S)$ is a section of $\mathfrak B_F (\mathfrak L ')$ and $\alpha_{x_0} = \alpha_0$. Finally, extend $\alpha '_S$ to a section of $ \mathfrak L '$ in any way and notice that $\alpha, \alpha'$ are $F$-related.
\end{proof}

Now, we are ready to prove that $\mathfrak B_F (\mathfrak L ')$ is involutive. Let $\Upsilon$ be the Courant-Jacobi tensor of $\mathfrak B_F (\mathfrak L ')$. We want to show that $\Upsilon = 0$. To see this, first compute $\Upsilon$ at a point $x_0$ as in the hypothesis of Lemma \ref{lem:bw}. Thus, let $(\alpha_i)_0 \in \mathfrak B_F (\mathfrak L ')_{x_0}$, and let $\alpha_i$ be sections of $\mathfrak B_F (\mathfrak L ')$, and $\alpha '_i$ sections of $ \mathfrak L '$, such that $(\alpha_i)_{x_0} = (\alpha_i)_0$ and $\alpha_i, \alpha_i '$ are $F$-related, $i =1,2,3$. Then it is easy to see from 
(\ref{eq:F-related}) that
\[
\Upsilon ((\alpha_1)_0, (\alpha_2)_0, (\alpha_3)_0) = \Upsilon (\alpha_1, \alpha_2, \alpha_3)_{x_0} = \Upsilon_{\mathfrak L'} (\alpha'_1, \alpha'_2, \alpha'_3)_{\underline F (x_0)} = 0,
\]
i.e.~$\Upsilon$ vanishes at all points $x_0$ as in the statement of Lemma \ref{lem:bw}. Since such points are dense in $M$, it follows that $\Upsilon$ vanishes everywhere.
\end{proof}

In analogy with the Dirac case, we call \emph{clean intersection condition} the condition that $\ker j^1 F \cap \underline F^\ast \mathfrak L'$ has constant rank (see \cite[Section 5]{B2013}). Clearly, it always holds when $ \mathfrak L ' = \mathfrak L _\omega$ for some $2$-cocycle $\omega$ in $(\Omega^\bullet_{L'}, d_{D})$.

\begin{example}\label{ex:bw}
Let $(L, \mathfrak L)$ be a Dirac-Jacobi bundle over a smooth manifold $M$, and let $\underline F : S \INTO M$ be the inclusion of a submanifold. Equip $S$ with the restricted line bundle $L|_S$ and let $F : L|_S \INTO L$ be the inclusion. Now $\ker j^1 F$ is the annihilator $D (L|_S)^0$ of $D (L|_S)$ in $(J^1 L)|_S$. It is easy to see that, actually, $D (L|_S)^0 = N^\ast S \otimes L|_S \subset T^\ast M |_S \otimes L|_S \subset (J^1 L)|_S$, where $N^\ast S$ is the \emph{conormal bundle} to $S$. Hence the clean intersection condition reads: \emph{$ \mathfrak L  \cap (N^\ast S \otimes L|_S)$ has constant rank}. If this condition is satisfied, then $L|_S$ inherits a Dirac-Jacobi structure
\begin{equation}\label{eq:bw_sub}
\mathfrak B_F (\mathfrak L ) = \{(\Delta , F^\ast \psi) \in \D L|_S : (\Delta, \psi) \in \mathfrak L  \},
\end{equation}
(cf.~\cite[Section 3.1]{C1990}). For instance, let $K_{\mathfrak L} $ be a regular distribution, and let $S \subset M$ be a submanifold complementary to $K_{\mathfrak L} $ (see Subsection \ref{subsec:adm}). Then, $ \mathfrak L  \cap (N^\ast S \otimes L|_S) = 0$ and the clean intersection condition is automatically satisfied. Indeed a form $\eta \in \mathfrak L  \cap (N^\ast S \otimes L|_S)$ annihilates tangent vectors to $S$. Being in $ \mathfrak L $, $\eta$ does also annihilate tangent vectors in $K_{\mathfrak L} $. Since $K_{\mathfrak L} |_S$ and $TS$ span the whole $TM|_S$, it follows that $\eta = 0$. Under the additional assumption that $(\mathbb 1, 0) \notin \Gamma (\mathfrak L) $, the backward image $\mathfrak B_F (\mathfrak L )$ is precisely the Dirac-Jacobi structure corresponding to the bracket $\{-,-\}_S$ induced by $ \mathfrak L $ on $\Gamma (L|_S)$ (Proposition \ref{prop:transv_jac}). Indeed, $\mathfrak B_F (\mathfrak L ) \cap D (L|_S)$ consists of points in $ \mathfrak L |_S$ of the form $(\Delta, 0)$ with $\sigma (\Delta) \in TS$. But, from $(\Delta, 0) \in \mathfrak L $, we also get $\sigma (\Delta) \in K_{\mathfrak L} $, hence $\sigma (\Delta) = 0$, and from $(\mathbb 1,0) \notin \Gamma (\mathfrak L) $, we get $\Delta = 0$. So, from (\ref{eq:DJ_jac}), $\mathfrak B_F (\mathfrak L )$ is the Dirac-Jacobi structure corresponding to a Jacobi bracket $\{-,-\}$. That $\{-,-\} = \{-,-\}_S$ now follows from (\ref{eq:J_bot}).
\end{example}

\subsection{Forward Dirac-Jacobi morphisms}

Under suitable regularity conditions, a Dirac-Jacobi structure can be pushed-forward along a regular morphism of line bundles. Namely, let $L \to M$ and $L' \to M'$ be line bundles, and let $F : L \to L'$ be a regular morphism of line bundles over a smooth map $\underline F : M \to M'$. Additionally, let $ \mathfrak L $ be a Dirac-Jacobi structure on $L$.

\begin{definition}
The \emph{forward image of $ \mathfrak L $ along $F$} is the, non-necessarily regular, distribution in $\underline F^\ast \D L'$ defined as:
\[
\mathfrak F_F (\mathfrak L ) := \{ (F_\ast \Delta, \psi ') : ((\Delta , F^\ast \psi ') \in \mathfrak L ' \} \subset \underline F^\ast (\D L).
\]
\end{definition}

\begin{definition}
A \emph{forward Dirac-Jacobi map} $F : (L, \mathfrak L) \to (L',\mathfrak L')$ between Dirac-Jacobi bundles is a regular morphism of line bundles $F : L \to L'$ such that $\mathfrak F_F (\mathfrak L ) = \underline F^\ast \mathfrak L'$.
\end{definition}

Clearly, an isomorphism $F : L \to L'$ over a diffeomorphism $\underline F : M \to M'$ is a forward Dirac-Jacobi map if and only if it is a backward Dirac-Jacobi map, and this happens iff
\[
(F_\ast \Delta, (F^{-1})^\ast \psi) \in \mathfrak L ' \ \text{for all } (\Delta, \psi) \in \mathfrak L . 
\]
So the two notions of Dirac-Jacobi maps agree for isomorphisms. An isomorphism which is (either a backward or a forward) Dirac-Jacobi map is a \emph{Dirac-Jacobi isomorphism}.

\begin{remark}
The forward image $\mathfrak F_F (\mathfrak L )$ is the image of $\Gamma_F \cap \mathfrak L$ under the projection  $\D L \oplus \underline F^\ast (\D L') \to F^\ast (\D L')$ onto the second summand. It is easy to see that the kernel of the surjection $\Gamma_F \cap \mathfrak L \to \mathfrak F_F (\mathfrak L )$ consists of points in $\D L \oplus \underline F^\ast (\D L')$ of the form $((\Delta,0), (0, 0))$, with $(\Delta,0) \in \mathfrak L $, and $F_\ast \Delta = 0$. Summarizing, there is a short (point wise) exact sequence
\begin{equation}
0  \longrightarrow \ker d_{D} F \cap \mathfrak L \longrightarrow \Gamma_F \cap  \mathfrak L  \longrightarrow \mathfrak F_F (\mathfrak L ) \longrightarrow 0,
\end{equation}
where the second arrow is the inclusion $ \D L \INTO \D L \oplus \underline F^\ast (\D L')$ and the third arrow is the projection $\D L \oplus \underline F^\ast (\D L') \to \underline F^\ast (\D L')$. Finally, $\mathfrak F_F (\mathfrak L )$ is a maximal isotropic distribution in $\D L$. This follows from \cite[Proposition 5.1]{B2013} similarly as for backward images. In particular, $\mathfrak F_F (\mathfrak L )$ has constant rank.
\end{remark}

In order to check whether or not the forward image does actually determine a Dirac structure on $M'$ one should check two things: first, whether or not $\mathfrak F_F (\mathfrak L )$ is a vector subbundle of $\underline{F}^\ast (\D L)$, and second, whether or not $\mathfrak F_F (\mathfrak L )$ descends to a vector subbundle of $\D L$. The first issue is addressed in the following

\begin{proposition}
If $\rk (\ker d_{D} F \cap  \mathfrak L)$ is constant, then the forward image $\mathfrak F_F (\mathfrak L )$ is a vector subbundle of $\underline F ^\ast (\D L')$.
\end{proposition}

\begin{proof}
It immediately follows from (\ref{eq:exact_bw}) by a similar argument as that in the proof of Proposition \ref{prop:bw}.
\end{proof}

We call \emph{clean intersection condition}, the condition that $\ker d_{D} F \cap \mathfrak L$ has constant rank. Clearly, it always holds when $ \mathfrak L  = \mathfrak L _J$ for some Jacobi bundle $(L,J)$ over $M$.

\begin{remark}\label{rem:fw}
The clean intersection condition is equivalent to the condition that $\ker d \underline F \cap  K_{\mathfrak L}$ has constant rank. Indeed, $\ker d_{D} F$ is mapped point-wise isomorphically onto $\ker d \underline F$ by the symbol map $\sigma : D L \to TM$. Hence $\ker d_{D} F \cap \mathfrak L = \ker d_{D} F \cap E_{\mathfrak L} $ is mapped point-wise isomorphically onto $\ker d \underline F \cap K_{\mathfrak L} $. In particular, $\ker d_{D} F \cap \mathfrak L$ has constant rank if and only if so does $\ker d \underline F \cap  K_{\mathfrak L} $.
\end{remark}

Next issue is addressed as follows. The fiber $\mathfrak F_F (\mathfrak L )_x$ of $\mathfrak F_F (\mathfrak L )$ over a point $x \in M$ is a subspace in the fiber of $\underline F ^\ast (\D L')$ over $x$. The latter is the fiber $\D_{\underline F (x)} L'$ of $\D L'$ over $\underline F (x)$. Hence, as for standard Dirac structures, it is natural to give the following

\begin{definition}
Dirac-Jacobi structure $ \mathfrak L $ is \emph{$F$-invariant} if $\mathfrak F_F (\mathfrak L )_x$ is independent of the choice of $x$ in a fiber of $\underline F$.
\end{definition}

Now, suppose for a moment that $\rk d \underline F$ is constant. Then the image of $\underline F$ is locally a submanifold. If $ \mathfrak L $ satisfies the clean intersection condition with respect to $F$, and, additionally, it is $F$-invariant, then $\mathfrak F_F (\mathfrak L )$ descends to a vector subbundle of $\D L'$ (over every smooth piece of $\underline F (M)$) whose fiber at $\underline F (x)$ is $\mathfrak F_F (\mathfrak L )_x$, $x \in M$. In particular, in order to get a vector subbundle of $\D L'$ over the whole $M'$, we need $\underline F$ to be a surjective submersion. Even more, we have the following

\begin{proposition}
Let $\underline F : M \to M'$ be a surjective submersion. If the rank of $\ker d_{D} F \cap  \mathfrak L$ is constant, and $ \mathfrak L $ is $F$-invariant, then the forward image $\mathfrak F_F (\mathfrak L )$ descends to a Dirac-Jacobi structure on $L'$.
\end{proposition}

\begin{proof}
The clean intersection condition and the $F$-invariance guarantee that $\mathfrak F_F (\mathfrak L )$ is a vector bundle descending to a maximal isotropic vector subbundle of $\D L'$. Denote it by $ \mathfrak L '$. It remains to prove that $ \mathfrak L '$ is involutive. This can be done similarly as in Proposition \ref{prop:bw} and \cite[Proposition 5.9]{B2013}. Details are left to the reader.
\end{proof}

\begin{example}\label{ex:fw}
Let $(L, \mathfrak L)$ be a Dirac-Jacobi bundle over a smooth manifold $M$, and let $L_N \to N$ be a line bundle. Moreover, let $F : L \to L_N$ be a regular morphism of line bundles over a surjective submersion $\underline F : M \to N$ with connected fibers, such that $\ker d \underline F \subset K_{\mathfrak L} $. In this case, the clean intersection condition holds, without further assumptions, in view of Remark~\ref{rem:fw}. Moreover, if, additionally,  $\ker d_{D} F \subset \mathfrak L$, then $ \mathfrak L $ is $F$-invariant. To see this, first notice that the Lie derivative $\Ll_\square$ along a derivation $\square \in \Gamma (\ker d_{D} F) \subset \Der L$ preserves sections of $ \mathfrak L $. Namely, let $(\Delta, \psi) \in \Gamma (\D L)$, and put $\Ll_\square (\Delta, \psi) := ([\square , \Delta], \Ll_\square \psi)$. Then $\Ll_\square (\Delta, \psi) \in \Gamma (\mathfrak L) $ for all $(\Delta, \psi) \in \Gamma (\mathfrak L) $. Indeed 
\[
\Ll_\square (\Delta, \psi) = ([\square , \Delta], \Ll_\square \psi) = \blq (\square, 0), (\Delta,\psi) \brq.
\]
In its turn, since $(\square, 0) \in \Gamma (\ker d_{D} F) \subset \Gamma (\mathfrak L)$, then $\blq (\square, 0), (\Delta,\psi) \brq$ is in $\Gamma (\mathfrak L)$ by involutivity of $ \mathfrak L $. In particular, $ \mathfrak L $ is preserved by the flow of $\square$, for all $\square \in \Gamma (\ker d_{D} F)$. Hence, since fibers of $\underline F$ are connected, for any $x, x'$ in the same $\underline F$-fiber, there is a Dirac-Jacobi isomorphism $G : L \to L$, over a diffeomorphism $\underline G: M \to M$, such that $\underline G$ maps $x$ to $x'$ and, additionally, $G$ is $F$-vertical, i.e.~$F \circ G = F$. It is now easy to see that
\[
\begin{aligned}
 \mathfrak F_F (\mathfrak L )_{x'} & = \{ (F_\ast \Delta, \psi ') : ( \Delta, F^\ast \psi ') \in \mathfrak L _{x'} \} \\
                                     & = \{ ( (F \circ G)_\ast \Delta ' , \psi ' ) : (\Delta ' , (F \circ G)^\ast \psi ') \in \mathfrak L _x \} \\
                                     & = \mathfrak F_{F \circ G} (\mathfrak L )_x = \mathfrak F_{F} (\mathfrak L )_x,
 \end{aligned}
\]
so that $ \mathfrak L $ is $F$-invariant. Hence  $\mathfrak F_F (\mathfrak L )$ descends to a Dirac-Jacobi structure on $L_N$. Denote it by $ \mathfrak L _N$. By construction $E_{\mathfrak L_N} = \mathfrak L _N \cap D L_N = (d_{D} F) (E_{\mathfrak L} )$. For instance, let $\ker d_{D} F = E_{\mathfrak L} $, so that $\ker d \underline F = K_{\mathfrak L} $. This means that distribution $K_{\mathfrak L} $ is simple, its leaf space is $N$, and $\underline F : M \to N$ is the natural projection. Additionally, $(\mathbb 1, 0) \notin \Gamma (\mathfrak L) $ and the canonical $(\ker d \underline F)$-connection in $L = \underline F^\ast L_N$ is the one induced by $ \mathfrak L $ (see Subection \ref{subsec:adm}). Hence $ \mathfrak L _N$ is the Dirac-Jacobi structure corresponding to the Jacobi bundle induced on $N$ by $(L, \mathfrak L)$. Clearly $F: (L, \mathfrak L) \to (\mathfrak L _N, L_N)$ is both a forward Dirac-Jacobi map and, from Example \ref{ex:bw}, a backward Dirac-Jacobi map.
\end{example}

\begin{remark}
The two conditions $\ker d \underline F \subset K_{\mathfrak L} $, and $\ker d_{D} F \subset \mathfrak L$ appearing in Example \ref{ex:fw} are not equivalent. While from $\ker d_{D} F \subset \mathfrak L$ immediately follows $\ker d \underline F \subset K_{\mathfrak L} $, the converse is not true. For instance, let $M = \R^2$, $L = \R_M$ and let $ \mathfrak L  = \mathfrak L _\nabla$, where $\nabla$ is any flat connection in $L$. So $E_{\mathfrak L} $ is the image of $\nabla$ and $K_{\mathfrak L}  = TM$. In particular, $\ker d \underline F \subset K_{\mathfrak L} $ for any $F$. Now let $N$ consist of one point $\ast$, $L_N = \R_{\{\ast\}} \to N = \{ \ast \}$, and let $F : M \times \R \to \R_{\{ \ast \}} = \R$ be the projection onto the second factor. It is easy to see that  $\ker d_{D} F$ is the image of the canonical trivial connection $\nabla_0$ in $L$. Hence, unless $\nabla = \nabla_0$, the kernel of $d_{D} F$ is not contained into $ \mathfrak L $.
\end{remark}

\section{Coisotropic embeddings of Dirac-Jacobi bundles}\label{sec:coisotrop}

Let $(L, \mathfrak L)$ be a Dirac-Jacobi bundle over a smooth manifold $M$, and let $S \subset M$ be a submanifold. The inclusion $\iota : L|_S \INTO L$ of the restricted line bundle is a regular morphism. Hence $ \mathfrak L $ induces a Dirac-Jacobi structure  $\mathfrak B_\iota (\mathfrak L )$ on $L|_S$ (up to a clean intersection condition, see Subsection \ref{subsec:bw}). When $ \mathfrak L $ is the Dirac-Jacobi structure corresponding to a precontact form $\theta : TM \to L$, then $\mathfrak B_\iota (\mathfrak L )$ is the Dirac-Jacobi structure corresponding to the pull-back precontact form $\iota^\ast \theta : TS \to L|_S$. On the other hand, when $ \mathfrak L $ is the Dirac-Jacobi structure corresponding to a Jacobi bracket $\{-,-\}$ on $L$, in general $\mathfrak B_\iota (\mathfrak L )$ does not correspond to a Jacobi bracket on $L|_S$. In other words, Jacobi brackets ``do not pass to submanifolds'' and we need to work in the general setting of Dirac-Jacobi bundles when dealing with submanifolds in a Jacobi manifold.

Coisotropic submanifolds are of a special interest. In Poisson geometry, they model first class constraints of Hamiltonian systems. Moreover the graph of a Poisson morphism is a coisotropic submanifold is a suitable \emph{product Poisson manifold}. Similar considerations hold in Jacobi geometry. For instance the graph of a Jacobi map is a coisotropic submanifold in a suitable \emph{product Jacobi manifold} \cite{ILMM1997}.

Now, a presymplectic manifold can be coisotropically embedded in a symplectic manifold if and only if the null distribution is regular, and the coisotropic embedding is essentially unique. This classical result, due to Gotay (see \cite{G1982} for details), plays an important role in symplectic geometry. For instance, Gotay's Theorem can be used to show that deformations of a coisotropic submanifold $S$ in a symplectic manifold are controlled by an $L_\infty$-algebra depending only on the intrinsic presymplectic geometry of $S$ \cite{OP2005}. There is a contact version of Gotay's Theorem stating that a precontact manifold can be coisotropically embedded in a contact manifold if and only if the null distribution is regular, and the coisotropic embedding is essentially unique. One can use the latter result to study deformations of a coisotropic submanifold in a contact manifold (see \cite{LOTV2014} for more details). Similarly, any Dirac manifold can be regarded as a coisotropic submanifold (with the pull-back Dirac structure) in a Poisson manifold, if and only if the null distribution is regular \cite[Section 8]{CZ2009}. This can be used, for instance, to reduce the quantization problem of the Poisson algebra of admissible functions on a Dirac manifold to the quantization problem of the Poisson algebra of basic functions on a coisotropic submanifold \cite{CF2007}. It is natural to ask: \emph{can one unify the above mentioned coisotropic embedding theorems in contact and Dirac geometry, proving an analogous result for Dirac-Jacobi bundles?} In this section we show that the answer is affirmative (see Theorem \ref{theor:coisotrop}).

Let $(L, \{-,-\})$ be a Jacobi bundle. As already remarked, the Jacobi bracket $\{-,-\}$ can be regarded as a morphism $J : \wedge^2 J^1 L \to L$ and defines a morphism $J^\sharp : J^1 L \to J_1 L \otimes L = D L$, $\psi \mapsto J(\psi, -)$. Recall that a submanifold $S \subset M$ is called \emph{coisotropic} (with respect to $(L,J)$) if $X_\lambda := (\sigma \circ J^\sharp)(\lambda)$ is tangent to $S$ for all sections $\lambda$ of $L$ such that $\lambda|_S = 0$. More information about coisotropic submanifolds in Jacobi geometry may be found, e.g., in \cite{LOTV2014}.

Let $(L_S, \mathfrak L)$ be a Dirac-Jacobi bundle over a manifold $S$.

\begin{definition}
A \emph{coisotropic embedding} of $(L_S, \mathfrak L)$ in a Jacobi bundle $(L, J)$ over a manifold $M$ is an embedding $\iota : L_S \INTO L$ over an embedding $\underline \iota : S \INTO M$, such that
\begin{enumerate}
\item the image of $\underline \iota$ is a coisotropic submanifold of $M$, and
\item  $ \mathfrak L  = \mathfrak B_\iota (\mathfrak L _J)$, i.e. $\iota : (L_S, \mathfrak L) \to (L, \mathfrak L_J)$ is a backward Dirac-Jacobi map.
\end{enumerate}
\end{definition} 

\begin{theorem}\label{theor:coisotrop}
Dirac-Jacobi bundle $(L_S, \mathfrak L)$ can be embedded coisotropically (i.e.~there is a coisotropic embedding) into a Jacobi bundle if and only if the null der-distribution $E_{\mathfrak L}  = \mathfrak L  \cap D L_S$ is regular.
\end{theorem}

\begin{proof}
The proof parallels the proof of the analogous proposition for Dirac manifolds \cite[Theorem 8.1]{CZ2009}. Let $(L_S, \mathfrak L)$ be a Dirac-Jacobi bundle over $S$, and let $\iota : L_S \to L$ be a coisostropic embedding into a Jacobi-bundle $(L, J)$ (over an embedding $\underline \iota : S \to M$). In the following, we use $\underline \iota$ to regard $S$ as a submanifold in $M$ and identify $L_S$ with $L|_S$. The rank of $E_{\mathfrak L}  = \mathfrak L  \cap D L_S$ is an upper semi-continuous function on $S$. On the other hand, it is also lower semi-continuous. Indeed, 
\[
\begin{aligned}
\mathfrak L = \mathfrak B_\iota (\mathfrak L _J) & = \{ (\Delta , \iota^\ast \psi) : (\Delta, \psi) \in \mathfrak L _J \} \\
 & = \{(J^\sharp (\psi) , \iota^\ast \psi) : \psi \in (J^1 L)|_S \text{ and } J^\sharp (\psi) \in D L_S \}.
 \end{aligned}
\]
Since $S$ is coisotropic, it follows that
\[
E_{\mathfrak L}  = \{ J^\sharp (\psi) : \psi \in \ker j^1 \iota \}.
\]
But $\ker j^1 \iota = N^\ast S \otimes L_S$ is a vector bundle (see Example \ref{ex:bw}). So, $E_{\mathfrak L}  = J^\sharp (N^\ast S \otimes L_S)$ is a smooth distribution. Hence it is regular.

Conversely, let $E_{\mathfrak L} $ be a regular distribution in $D L_S$. Restricting to connected components of $S$, if necessary, we can assume that $E_{\mathfrak L} $ is a vector bundle. We want to show that there is a Jacobi bundle on a neighborhood of the zero section $\mathbf 0$ of the vector bundle $\pi: E_{\mathfrak L} ^\dag := E_{\mathfrak L} ^\ast \otimes L_S \to S$ such that $\mathbf 0 : S \to E_{\mathfrak L} ^\dag$ is a coisotropic embedding.

Pick a complement $G$ of $E_{\mathfrak L} $ in $D L_S$, i.e.~$D L_S = G \oplus E_{\mathfrak L} $. Equip $E_{\mathfrak L} ^\dag$ with the pull-back line bundle $L := \pi^\ast L_S$, and define a maximal isotropic subbundle $ \mathfrak L _G \subset \D L$, depending on $G$, as follows. Abusing the notation, we denote again by $\pi : L = \pi^\ast L_S \to L_S$ the projection. Thus, take the backward image $\mathfrak B_\pi (\mathfrak L ) \subset \D L$. The clean intersection condition is automatically satisfied.  Hence $\mathfrak B_\pi (\mathfrak L )$ is a Dirac-Jacobi structure on $L$.

There is a canonical $1$-cochain $\Theta_G$ in $(\Omega^\bullet_L, d_{D})$. The value of $\Theta_G$ at $\varepsilon \in E_{\mathfrak L} ^\dag$ is the composition:
\begin{equation}\label{eq:Theta_G}
\xymatrix{D_{\varepsilon} L \ar[r]^-{d_{D} \pi} & D_x L_S \ar[r] &(E_{\mathfrak L} )_x \ar[r]^-{\varepsilon} & (L_S)_x}, 
\end{equation}
where $x = \pi (\varepsilon)$ and the second arrow is the projection with kernel $G_x$.
Take the differential $\omega_G := -d_{D} \Theta_G$. There is an alternative description of $\omega_G$. Namely, the splitting $D L_S = G \oplus E_{\mathfrak L} $ induces a splitting $J^1 L_S = (G^\ast \otimes L_S) \oplus E_{\mathfrak L} ^\dag$, whence an embedding $E_{\mathfrak L} ^\dag \INTO J^1 L_S$ of vector bundles. Recall that $J^1 L_S$ is equipped with a canonical contact form $\theta$ taking values in the pull-back line bundle $J^1 L_S \times_S L_S$ (see, e.g., \cite[Example 5.5]{LOTV2014}). Consider the $2$-cocycle $\omega_{J^1 L_S} := - d_{D} (\theta \circ \sigma)$. It is easy to see that $\omega_G$ coincides with the pull-back of $\omega_{J^1 L_S}$ along the embedding $L \INTO J^1 L_S \times_S L_S$ over the embedding $E_{\mathfrak L} ^\dag \INTO J^1 L_S$.

Next, use $\omega_G$ to ``gauge transform'' $\mathfrak B_\pi (\mathfrak L )$ and get a Dirac-Jacobi structure (see Example \ref{ex:gauge}):
\[
\mathfrak L_G := \tau_{\omega_G}\mathfrak B_\pi (\mathfrak L ) = \{ (\Delta, \psi + i_\Delta \omega_G) : (\Delta, \psi) \in \mathfrak B_\pi (\mathfrak L ) \}.
\]
We claim that $ \mathfrak L _G$ is the Dirac-Jacobi structure corresponding to a Jacobi bracket on $\Gamma(L)$, at least around the (image of the) zero section $\mathbf 0$ of $\pi$. To prove the claim, it suffices to show that the characteristic leaves of $ \mathfrak L _G$ are either genuinely contact or genuinely lcs around $\mathbf 0$. Thus, describe the characteristic foliation of $ \mathfrak L _G$. It is easy to see that $\Ii_{\mathfrak L_G} = \pr_{D} \mathfrak L_G = \pi_\ast^{-1} \Ii_{\mathfrak L}$. It immediately follows that all characteristic leaves of $ \mathfrak L _G$ are of the form $\widehat{\Oo} := \pi^{-1} (\Oo)$ where $\Oo$ is a characteristic leaf of $ \mathfrak L $, and, for every characteristic leaf $\Oo$ of~$\mathfrak L$,
\[
\Ii_{\mathfrak L_G}|_{\widehat \Oo} = \pi_{\ast}^{-1} (\Ii_{\mathfrak L}|_\Oo) = 
\begin{cases}
D (L|_{\widehat \Oo}) & \text{if $\Oo$ is precontact} \\
\pi_\ast^{-1} (\operatorname{im} \nabla^\Oo) & \text{if $\Oo$ is lcps}
\end{cases}
\]
(here, when $\Oo$ is lcps, $\nabla^\Oo$ is the connection in $L_S|_{\Oo}$). In particular, $\widehat \Oo$ is precontact (resp.~lcps) if and only if $\Oo$ is such. Compute the $2$-cocycle 
$\omega_{\widehat \Oo}$. First, let $\Oo$ be precontact. A direct check shows that, in this case,
\begin{equation}\label{eq:omega_O}
\omega_{\widehat \Oo} = \pi^\ast \omega_\Oo + \iota^\ast \omega_G,
\end{equation}
where $\iota : L|_{\widehat \Oo} \INTO L$ is the inclusion. Now, we show that $\omega_{\widehat \Oo}$ is non-degenerate around the image of the zero section $\mathbf 0 : \widehat \Oo \to \Oo$, which, abusing the notation, we denote again by $\mathbf 0$. It suffices to show that the point-wise restriction $\omega_{\widehat \Oo}|_{\mathbf 0}$ is non-degenerate. Recall that the exact sequence
\[
0 \longrightarrow (\ker d \pi)|_{\mathbf 0} \longrightarrow T \widehat \Oo|_{\mathbf 0} \overset{d \pi}{\longrightarrow} T\Oo \longrightarrow 0
\]
splits via the inclusion $d \mathbf{0} : T\Oo \to T\widehat \Oo|_{\mathbf 0}$. Moreover, since $\widehat \Oo = E_{\mathfrak L} ^\dag|_\Oo \to \Oo$ is a vector bundle, then $(\ker d \pi)|_{\mathbf 0} \simeq \widehat \Oo$. Summarizing, there is a direct sum decomposition $T\widehat \Oo|_{\mathbf 0} = T\Oo \oplus \widehat \Oo$. Denote by $p_{\widehat \Oo} : T\widehat \Oo|_{\mathbf 0} \to \widehat \Oo$ the projection with kernel $T \Oo$.

Similarly, the exact sequence
\[
0 \longrightarrow (\ker d_{D} \pi)|_{\mathbf 0} \longrightarrow D (L|_{\widehat \Oo})|_{\mathbf 0} \overset{d_{D} \pi}{\longrightarrow} D (L_S|_\Oo) \longrightarrow 0
\]
splits via the inclusion $d_{D} \mathbf{0} : D (L_S|_\Oo) \to D (L|_{\widehat \Oo}) |_{\mathbf 0}$. Moreover, the composition $\sigma \circ p_{\widehat \Oo} : (\ker d_{D} \pi)|_{\mathbf 0} \to \widehat \Oo$ is an isomorphism. Hence there is a direct sum decomposition $D (L|_{\widehat \Oo})|_{\mathbf 0} = D ( L_S|_\Oo) \oplus \widehat \Oo$. Accordingly, a section $\varepsilon$ of $\widehat \Oo$ identifies with the unique differential operator $\mathbb D_{\varepsilon} : L|_{\widehat \Oo} \to L|_{\widehat \Oo}$ such that $\mathbb D_{\varepsilon} \lambda = 0$ for all fiber-wise constant sections $\lambda \in \Gamma (L|_{\widehat \Oo})$, and $\mathbb D_{\varepsilon} e = \langle \varepsilon, e \rangle$ for all fiber-wise linear sections $e \in \Gamma (L|_{\widehat \Oo})$ (a fiber-wise linear section of $L$ is a section of the form $\sum f \otimes \lambda$ where the $f$'s are fiber-wise linear function on $\widehat \Oo$ and the $\lambda$'s are fiber-wise constant sections of $L|_{\widehat \Oo}$, i.e.~pull-back sections $\pi^\ast \lambda_0$, with $\lambda_0 \in \Gamma (L_S|_\Oo)$ - in particular, fiber-wise linear sections of $L|_{\widehat \Oo}$ identify with sections of $E_{\mathfrak L} |_\Oo$). Now, notice that $E_{\mathfrak L} |_\Oo \subset D (L_S|_\Oo)$ and $D (L_S|_\Oo) = G_\Oo \oplus E_{\mathfrak L} |_\Oo$ where $G_\Oo = G \cap D (L_S|_\Oo)$ ($G_\Oo$ differs from $G|_\Oo$ in general, because sections of $G$ are derivations $\Gamma (L_S) \to \Gamma (L_S)$ but, in general, not all of them are tangent to $\Oo$). Summarizing, there is a direct sum decomposition $D (L|_{\widehat \Oo})|_{\mathbf 0} = G_\Oo \oplus E_{\mathfrak L} |_\Oo \oplus \widehat \Oo$. 

From (\ref{eq:omega_O}), $\omega_{\widehat \Oo}|_{\mathbf 0} = (\pi^\ast \omega_\Oo)|_{\mathbf 0} + (\iota^\ast \omega_G)|_{\mathbf 0}$. Compute the first summand $(\pi^\ast \omega_\Oo)|_{\mathbf 0}$. Since $\widehat \Oo$ is the kernel of $ \pi_\ast : D (L|_{\widehat \Oo})|_{\mathbf 0} \to D (L_S|_\Oo)$, and, from (\ref{eq:E_omega}), $E_{\mathfrak L} |_\Oo = \ker \omega_{\Oo}$, then $(\pi^\ast \omega_\Oo)|_{\mathbf 0}$ is given by matrix
\[
\begin{tabular}{cccc}
                & $G_\Oo$                                                          & $E_{\mathfrak L} |_\Oo$ & $\widehat \Oo$  \\
$G_\Oo$         & \multicolumn{1}{||c}{$\omega_\Oo |_{G_\Oo}$} & $0$  &\multicolumn{1}{c||}{$0$}\\
$E_{\mathfrak L} |_\Oo$         & \multicolumn{1}{||c}{$0$}                          & $0$  &\multicolumn{1}{c||}{$0$} \\
$\widehat \Oo$ & \multicolumn{1}{||c}{$0$}                         & $0$  &\multicolumn{1}{c||}{$0 $} 
\end{tabular}\ ,
\]
where $\omega_\Oo |_{G_\Oo}$ is non-degenerate. Now, compute the second summand $(\iota^\ast \omega_G)|_{\mathbf 0}$. From (\ref{eq:Theta_G}), $\Theta_G$ vanishes on $\mathbf 0$. Hence, for all $\Delta, \square \in D (L|_{\widehat \Oo})$,
\[
(\iota^\ast \omega_G)|_{\mathbf 0}(\Delta|_{\mathbf 0}, \square|_{\mathbf 0}) = \Delta |_{\mathbf 0} (\Theta_G (\square)) - \square |_{\mathbf 0} (\Theta_G (\Delta)).
\]
Let $\Delta|_{\mathbf 0} \in \Gamma (G)$. Clearly, $\Delta$ can be chosen in such a way that $(p_{E} \circ d_{D} \pi) \Delta = 0$ so that $\Theta_G (\Delta) = 0$ (here $p_E : D (L_S|_\Oo) \to E_{\mathfrak L} |_\Oo$ is the projection with kernel $G_\Oo$). Since $\Delta|_{\mathbf 0}$ is tangent to $\mathbf 0$, we also have 
\[
\Delta |_{\mathbf 0} (\Theta_G (\square))= \Delta |_{\mathbf 0} (\Theta_G (\square)|_{\mathbf 0}) = 0.
\]
 Hence $(\iota^\ast \omega_G)|_{\mathbf 0}(\Delta|_{\mathbf 0}, \square|_{\mathbf 0}) = 0$ whenever $\Delta |_{\mathbf 0} \in \Gamma (G)$. Similarly, 
 \[
 (\iota^\ast \omega_G)|_{\mathbf 0}(\Delta|_{\mathbf 0}, \square|_{\mathbf 0}) = 0
 \]
  when both $\Delta|_{\mathbf 0}, \square|_{\mathbf 0} \in \Gamma (E_{\mathfrak L} |_\Oo)$ or $\Delta|_{\mathbf 0}, \square|_{\mathbf 0} \in \Gamma (\widehat \Oo)$. Finally, let $\Delta|_{\mathbf 0} = \varepsilon \in \Gamma (\widehat \Oo)$ and $\square|_{\mathbf 0} = e  \in \Gamma (E_{\mathfrak L} |_\Oo)$. In particular, $\square$ is tangent to $\mathbf 0$ so that 
  \[
  \square |_{\mathbf 0} (\Theta_G (\Delta)) = \square |_{\mathbf 0} (\Theta_G (\Delta)|_{\mathbf 0}) = 0 .
  \] Choose $\square$ such that $(p_{E} \circ d_{D} \pi) \Delta = e$ and $\Delta = \mathbb D_{\varepsilon}$. With these choices, it is easy to see that $\Delta |_{\mathbf 0} (\Theta_G (\square)) = \langle \varepsilon, e \rangle$. We conclude that $(\iota^\ast \omega_G)|_{\mathbf 0}$ is given by the matrix
\[
\begin{tabular}{cccc}
                & $G_\Oo$                                                          & $E_{\mathfrak L} |_\Oo$ & $\widehat \Oo$  \\
$G_\Oo$         & \multicolumn{1}{||c}{$0$}                          & $0$  &\multicolumn{1}{c||}{$0$}\\
$E_{\mathfrak L} |_\Oo$         & \multicolumn{1}{||c}{$0$}                          & $0$  &\multicolumn{1}{c||}{$- \langle -,- \rangle$} \\
$\widehat \Oo$ & \multicolumn{1}{||c}{$0$}                         & $\langle -, -\rangle$  &\multicolumn{1}{c||}{$0 $} 
\end{tabular}\ ,
\]
where $\langle -, -\rangle : \widehat \Oo \otimes E_{\mathfrak L} |_\Oo \to L_S|_\Oo$ is the duality pairing (twisted by $L_S|_\Oo$). Hence $\omega_{\widehat \Oo}|_{\mathbf 0}$ is given by the non-degenerate matrix
\[
\begin{tabular}{cccc}
               & $G_\Oo$                                                          & $E_{\mathfrak L} |_\Oo$ & $\widehat \Oo$  \\
$G_\Oo$         & \multicolumn{1}{||c}{$\omega_\Oo |_{G_\Oo}$}                          & $0$  &\multicolumn{1}{c||}{$0$}\\
$E_{\mathfrak L} |_\Oo$         & \multicolumn{1}{||c}{$0$}                          & $0$  &\multicolumn{1}{c||}{$- \langle -,- \rangle$} \\
$\widehat \Oo$ & \multicolumn{1}{||c}{$0$}                         & $\langle -, -\rangle$  &\multicolumn{1}{c||}{$0 $} 
\end{tabular}\ .
\]
So, locally around $\mathbf 0$, the characteristic leaf $\widehat \Oo$ is actually a contact leaf.

Now, let $\Oo$ be lcps. A direct computation shows that $\omega_{\widehat \Oo} = \sigma^\ast \underline{\omega}{}_{\widehat \Oo}$ with
\[
\underline{\omega}{}_{\widehat \Oo} = \pi^\ast \underline \omega{}_\Oo + (\nabla^{\widehat \Oo})^\ast  \iota^\ast\omega_G,
\]
where the second summand is defined by $(\nabla^{\widehat \Oo})^\ast \iota^\ast \omega_G (X, Y) := \omega_G (\nabla^{\widehat \Oo}_X, \nabla^{\widehat \Oo}_Y)$, for all $X,Y \in T \widehat \Oo$. Similarly as above, we want to show that $\underline{\omega}{}_{\widehat \Oo}$ is non-degenerate around the image $\mathbf 0$ of the zero section of $\widehat \Oo \to \Oo$ and, to do this, it suffices to prove that $\underline{\omega}{}_{\widehat \Oo}|_{\mathbf 0}$ is non-degenerate. First of all, notice that the symbol $\sigma : D L_S \to TS$ establishes an isomorphism between $E_{\mathfrak L} |_\Oo$ and $\underline E{}_{\mathfrak L}|_\Oo := \ker \underline \omega{}_\Oo$ whose inverse isomorphism is given by the connection $\nabla^\Oo$. In particular, $\underline E{}_{\mathfrak L}|_\Oo$ is a vector subbundle of $T\Oo$, and there is a direct sum decomposition $T\Oo = \underline G{}_\Oo \oplus \underline E{}_{\mathfrak L}|_\Oo $, where 
\[
\underline G{}_\Oo := \sigma (G \cap D (L|_{\widehat \Oo})) = \sigma (G) \cap T \Oo.
\]
 Hence there is a direct sum decomposition $T \widehat \Oo|_{\mathbf 0} = \underline G_\Oo \oplus \underline E{}_{\mathfrak L}|_\Oo \oplus \widehat \Oo$. Similarly as in the contact case, one proves that $\underline \omega{}_{\widehat \Oo}|_{\mathbf 0}$ is given by the non-degenerate matrix
\[
\begin{tabular}{cccc}
               & $\underline G{}_\Oo$                                                          & $\underline E{}_{\mathfrak L}|_\Oo$ & $\widehat \Oo$  \\
$\underline G{}_\Oo$         & \multicolumn{1}{||c}{$\underline \omega{}_\Oo |_{\underline G{}_\Oo}$}                          & $0$  &\multicolumn{1}{c||}{$0$}\\
$\underline E{}_{\mathfrak L}|_\Oo$         & \multicolumn{1}{||c}{$0$}                          & $0$  &\multicolumn{1}{c||}{$- \langle -,- \rangle$} \\
$\widehat \Oo$ & \multicolumn{1}{||c}{$0$}                         & $\langle -, -\rangle$  &\multicolumn{1}{c||}{$0 $} 
\end{tabular}\ .
\]
where $\langle -, -\rangle : \widehat \Oo \otimes \underline E{}_{\mathfrak L}|_\Oo \to L_S|_\Oo$ is the duality pairing. Details are left to the reader. So, locally around $\mathbf 0$, the characteristic leaf $\widehat \Oo$ is a lcs leaf.

It remains to prove that the zero section of $E_{\mathfrak L} ^\dag$ is a coisotropic embedding of $S$. This immediately follows from \cite[Corollary 3.3.(3)]{LOTV2014} and the fact that $\mathbf 0 \cap \widehat \Oo = \Oo$ for every characteristic leaf $\Oo$ of $(L_S, \mathfrak L)$.
\end{proof}

The construction in the proof of Theorem \ref{theor:coisotrop} depends, a priori, on the choice of a complementary vector bundle $G$. However, two different choices of $G$ determine isomorphic Jacobi bundles, at least around the zero section of $E_{\mathfrak L} ^\dag$, as shown by the following 

\begin{proposition}
Let $(L_S, \mathfrak L)$ be a Dirac-Jacobi bundle over a manifold $S$, such that $E_{\mathfrak L} $ is a vector bundle and let $G_0,G_1 \subset D L_S$ be two complementary vector subbundles, i.e.~$E\oplus G_0 = E \oplus G_1 = D L_S$. Finally, let $(L:= \pi^\ast L_S, J_{0})$ (resp.~$(L:= \pi^\ast L_S, J_{1})$) be the Jacobi structure determined by $G_0$ (resp., $G_1$) on a neighborhood of the (image of the) zero section $\mathbf 0$ of $\pi : E_{\mathfrak L} ^\dag \to S$ as in the proof of Theorem \ref{theor:coisotrop}. Then there is a Jacobi isomorphism $\Phi : (L, J_{0}) \to (L, J_{1})$ locally defined around $\mathbf 0$, such that $\Phi \circ \mathbf 0 = \mathbf 0$.
\end{proposition}

\begin{proof}
The proof parallels the proof of \cite[Proposition 8.2]{CZ2009}. We use the same notations as in the proof of Theorem \ref{theor:coisotrop}. Let $G_0, G_1$ be as in the statement, and let $A : G_0 \to E_{\mathfrak L} $ be the vector bundle morphism whose graph is $G_1 \subset G_0 \oplus E_{\mathfrak L}  = D L_S$. Let $G_t$ be the graph of the vector bundle morphism $tA : G_0 \to E_{\mathfrak L} $, $t \in [0,1]$. The $G_t$'s interpolate between $G_0$ and $G_1$. Clearly, $D L_S = G_t \oplus E$ for all $t$. Hence every $G_t$ determines a Dirac-Jacobi structure $ \mathfrak L _t$ on $L$. Moreover, $ \mathfrak L _t$ corresponds to a Jacobi bundle $(L, J_t)$, at least around $\mathbf 0$. We claim that the time-$1$ flow $\Phi$ of the (time dependent) derivation 
\[
\Delta_t := J_t^\sharp \left(\frac{d}{dt} \Theta_{G_t} \right) \in \Der L
\]
fixes $\mathbf 0$ and maps $J_0$ to $J_1$. Since $\Theta_{G_t}$ vanishes on $\mathbf 0$ for all $t$'s, so does $\Delta_t$, hence $\Phi \circ\mathbf 0 = \mathbf 0$ (in particular, $\Phi$ is well defined around $\mathbf 0$). Finally, notice that the $J_t$'s share the same characteristic foliation: 
\[
\{ \widehat \Oo: \Oo \text{ is a characteristic leaf of } D_S \}.
\]
In particular, $\Delta_t$ is tangent to $\widehat \Oo$ for all $\Oo$. Denote by $\omega_{\widehat \Oo}(t)$ the $2$-cochain on $\widehat \Oo$ induced by $ \mathfrak L _t$. If $\Oo$ is precontact then $\omega_{\widehat \Oo}(t)$ is non degenerate around $\mathbf 0$ and
\[
\omega_{\widehat \Oo} (t) = \omega_{\widehat \Oo}(0) + \iota^\ast (\omega_{G_t} - \omega_{G_0}).
\]
It follows from the \emph{contact version of the Moser lemma} that $\Phi$ maps $\omega_{\widehat \Oo} (0)$ to $\omega_{\widehat \Oo} (1)$. Similarly in the lcps case. Details are left to the reader. So $\Phi$ maps the contact/lcs foliation of $J_0$ to the contact/lcs foliation of $J_1$, hence it maps $J_0$ to $J_1$.
\end{proof}

Any Jacobi bundle as the one constructed in the proof of Theorem \ref{theor:coisotrop} will be called a \emph{Jacobi thickening} of $(L_S, \mathfrak L)$. The above proposition shows that the Jacobi thickening is unique up to isomorphisms. Hence every Dirac-Jacobi bundle $(L_S, \mathfrak L)$ such that $E_{\mathfrak L} $ is a regular distribution, admits an essentially canonical coisotropic embedding. 

\begin{remark}
Let $(L_S, \mathfrak L)$ be a Dirac-Jacobi bundle such that $E_{\mathfrak L} $ is a regular distribution, and let $(L, J)$ be a Jacobi thickening. The proof of Theorem \ref{theor:coisotrop} shows that the contact/lcs characteristic leaves of $(L, J)$ are Jacobi thickenings of the precontact/lcps characteristic leaves of $(L_S, \mathfrak L)$. In the case when $S$ has just one characteristic leaf, i.e.~$(L_S, \mathfrak L)$ corresponds to either a precontact or a lcps manifold, Jacobi thickening reproduces either the contact thickening of \cite{LOTV2014} or the lcs thickening of \cite{LO2012}. 
\end{remark}

\begin{example}
Let $L_S$ be a line bundle over a smooth manifold $S$, and let $V \subset D L_S$ be an involutive vector subbundle. The vector bundle $V$ inherits from $D L_S$ a Lie algebroid structure. Moreover, the inclusion $V \INTO D L_S$ is a representation of the Lie algebroid $V$ in the line bundle $L_S$. So $V$ is a Jacobi algebroid in the sense of \cite{LOTV2014}. In particular, there is a Jacobi bundle $(L := \pi^\ast L_S, J)$, with fiberwise linear Jacobi bracket $J$, over the total space of the vector bundle $\pi : V^\ast \otimes L_S \to S$ (see \cite{LOTV2014} for details). Now, let $ \mathfrak L  = V \oplus V^0$ be the Dirac-Jacobi structure on $L_S$ corresponding to $V$. The distribution $E_{\mathfrak L}  = V$ is automatically regular, hence $(L_S,\mathfrak L)$ admits a Jacobi thickening. Namely, let $G \subset D L_S$ be a complementary vector subbundle, i.e.~$V \oplus G = D L_S$. Then $G$ determines a Jacobi bundle $(L , J_G)$ on a neighborhood of the zero section $\mathbf 0$ of $\pi: V^\ast \otimes L \to S$, and $\mathbf 0$ is a coisotropic embedding. We claim that $J_G$ is in fact independent of $G$ and coincides with $J$. In other words, we claim that, for every $\alpha, \beta \in \Gamma (V)$ and $\varepsilon \in V^\ast \otimes L_S$,
\begin{equation}\label{eq:Vdag}
J_G (j_{\varepsilon}^1 \alpha, j_{\varepsilon}^1 \beta) = \langle \varepsilon , [\alpha, \beta ] \rangle,
\end{equation}
where, in the left hand side, we interpret $\alpha, \beta$ as fiber-wise linear sections of $L$ (see the proof of Theorem \ref{theor:coisotrop} for a definition of fiber-wise linear sections).

In order to prove (\ref{eq:Vdag}), recall that $\omega_\Oo =0$ for every characteristic leaf $\Oo$ of $V \oplus V^0$. Now denote by $\ell$ the restriction $L_S|_\Oo$. Assume that $\Oo$ is precontact. Then $V|_\Oo = D \ell$ and $V^\ast \otimes L |_\Oo = J^1 \ell$. As in the proof of Theorem~\ref{theor:coisotrop}, denote by $\iota : \widehat \Oo \to V^\ast \otimes L$ the inclusion. By definition $\iota^\ast \Theta_G$ agrees with the contact form on $J^1 \ell$, hence $\omega_{\widehat \Oo} = \pi^\ast \omega_\Oo + \iota^\ast \omega_G = \iota^\ast \omega_G$ agrees with the canonical $2$-cocycle on $J^1 \ell$ and Equation (\ref{eq:Vdag}) holds for all $\varepsilon \in \widehat \Oo$. On the other hand, if $\Oo$ is lcps, then $V|_\Oo = \mathrm{im} \nabla_S$ where $\nabla_S : T \Oo \to D \ell$ is a flat connection. Hence $V^\ast \otimes L |_\Oo$ identifies with $T^\ast \Oo \otimes \ell$ and $L|_{\widehat \Oo}$ identifies with $(T^\ast \Oo \otimes \ell) \times_{\Oo} \ell$. Similarly as in the precontact case, connection $\nabla^{\widehat \Oo}$ in $L|_{\widehat \Oo}$ induced by $ \mathfrak L $ agrees with the pull-back connection $\pi^\ast \nabla_S$, and $\underline \omega{}_{\widehat \Oo}$ agrees with the canonical lcs form on $T^\ast \Oo \otimes \ell$. Hence Equation (\ref{eq:Vdag}) holds for $\varepsilon$ belonging to lcs leaves as well.
\end{example}

\begin{example}
Let $(L_S, \mathfrak L)$ be a Dirac-Jacobi bundle over a smooth manifold $S$. From Proposition \ref{prop:KD}, if the null distribution $K_{\mathfrak L} $ is regular, so is $E_{\mathfrak L} $, hence $(L_S, \mathfrak L)$ can be coisotropically embedded in a Jacobi bundle. So regularity of $K_{\mathfrak L} $ is sufficient for coisotropic embedding. However, it is not necessary as the following example shows. Let $M = \R^2$ with standard coordinates $x,y$ and $L = \R_M$ so that $\D L = (TM \oplus \R_M) \oplus (T^\ast M \oplus \R_M)$. Moreover, let $J \in \Gamma (\wedge^2 (J^1 L)^\ast \otimes L )= \Gamma (\wedge^2 (TM \oplus \R_M))$ be the Jacobi bracket on $C^\infty (M)$ given by  
\[
J = x \frac{\partial}{\partial y} \wedge \frac{\partial}{\partial x} + \frac{\partial}{\partial y} \wedge 1.
\]
It is easy to see that $S = \{ y = 0 \}$ is a coisotropic submanifold. Moreover, $ \mathfrak L _J$ can be pulled-back to $(L_S := L|_S = \R_S)$ giving a Dirac-Jacobi structure $ \mathfrak L  \subset \D L_S =  (TS \oplus \R_S) \oplus (T^\ast S \oplus \R_S)$ generated by 
\[
1 + x \frac{\partial}{\partial x}, \ dx - x \cdot 1.
\]
In particular, $E_{\mathfrak L} $ is spanned by $1 + x \frac{\partial}{\partial x}$ and it is, as expected, a smooth rank one vector bundle. On the other hand, $K_{\mathfrak L} $ is spanned by $x \frac{\partial}{\partial x}$. So its rank is one everywhere except in $x = 0$ where it is zero.
\end{example}

\section{Integration of Dirac-Jacobi structures on line bundles}\label{sec:integr}
A Lie algebroid is \emph{integrable} if it is isomorphic to the Lie algebroid $A(\Gg)$ of a Lie groupoid $\Gg$. An \emph{integration} of a Lie algebroid $A$ is a Lie groupoid $\Gg$ together with an isomorphism $A \simeq A (\Gg)$. A geometric structure $\mathcal X$ on a Lie groupoid $\Gg$, which is compatible with the groupoid structure in a suitable sense, usually determines a geometric structure on $A(\Gg)$ which is compatible with the algebroid structure in a suitable sense, and is to be considered as the ``infinitesimal counterpart of $\mathcal X$''. Conversely, let $A$ be an integrable Lie algebroid, and let $\Gg$ be an integration on $A$. Suppose that $A$ is equipped with a geometric structure $\underline{\mathcal{X}}$, compatible with the algebroid structure. It is natural to wonder whether $\underline{\mathcal{X}}$ is the infinitesimal counterpart of a suitable structure $\mathcal X$ on $\Gg$. If this is the case, we say that $(A, \underline{\mathcal X})$ integrates to $(\Gg, \mathcal X)$. There are many examples of positive answers to the above question. For instance, Poisson manifolds integrate, if at all integrable, to symplectic groupoids. More generally, Dirac manifolds integrate, if at all integrable, to presymplectic groupoids \cite{BCZ2004}. In \cite{IW2006} Iglesias and Wade show that $\mathcal{E}^1 (M)$-Dirac manifolds integrate to cooriented precontact groupoids. In this section we discuss the case of a generic Dirac-Jacobi bundle using the language of Spencer operators introduced by Crainic, Salazar and Struchiner in \cite{CSS2012}.

As already mentioned in the introduction, we assume that the reader is familiar with (fundamentals of) the theory of Lie groupoids, Lie algebroids and their representations (see, e.g., \cite{CF2010} and references therein). Here we will only recall some aspects of the theory. The unfamiliar reader will find a quick introduction to all other aspects relevant for this section in \cite{S2013} (see Section 1.2 therein, see also Chapter 4). 

Let $\Gg \rightrightarrows M$ be a Lie groupoid with source $s$, target $t$ and unit $u$. We identify $M$ with its image under $u$. Denote by $\Gg_2 = \{ (g_1, g_2) \in \Gg \times \Gg: s(g_1) = t (g_2)\}$ the manifold of composable arrows and let $m : \Gg_2 \to \Gg$, $(g_1, g_2) \mapsto g_1g_2$ be the multiplication. We denote by $\pr_1, \pr_2 : \Gg_2 \to \Gg$ the projections onto the first and the second factor respectively.

Recall that the Lie algebroid $A(\Gg)$ of $\Gg$ consists of tangent vectors to the source fibers at points of $M$. Every section $\alpha$ of $A(\Gg)$ corresponds to a unique right invariant, $s$-vertical vector field $\alpha^r$ on $\Gg$ such that $\alpha = \alpha^r |_M$. The Lie bracket in $\Gamma(A(\Gg))$ is induced by the commutator of right invariant vector fields, and the anchor of $A(\Gg) \to TM$ is given by $\rho (\alpha) = t_\ast (\alpha)= t_\ast (\alpha^r|_M)$, for all $\alpha \in \Gamma (A(\Gg))$. We denote by $\phi_\alpha^z$ the time-$z$ flow of $\alpha^r$. 

Now, let $E \to M$ be a vector bundle carrying a representation of $\Gg$, i.e.~a smooth map $\Gg \times_s E \to M$, written $(g, e) \mapsto g \cdot e$, satisfying the usual properties of an action. The infinitesimal counterpart of a $\Gg$-action is a representation of $A(\Gg)$, i.e.~a flat $A(\Gg)$-connection $\nabla : A(\Gg) \to D E$ given by 
\begin{equation}\label{eq:nabla_0}
\nabla_\alpha e = \left. \frac{d}{dz} \right|_{z= 0} g(z)^{-1} \cdot e_{t(g(z))}
\end{equation}
where $z \mapsto g (z)$ is any smooth curve in the $s$-fiber through $x$, such that $g(0) = x$ and $\dot g (0) = \alpha$. Here $x \in M$, and $\alpha \in A(\Gg)_x$.

\begin{proposition}\label{prop_nabla_G}
There is a canonical flat $\ker ds$-connection $\nabla^{\Gg} :\linebreak \ker ds \to D t^\ast E$ in the pull-back bundle $t^\ast E$ such that
\begin{equation}\label{eq:nabla}
\nabla_\alpha = t_\ast (\nabla^\Gg_{\alpha^r}|_M)
\end{equation}
for all $\alpha \in \Gamma (A(\Gg))$.
\end{proposition}

\begin{proof}
Let $g : x \to y \in \Gg$, $X \in \ker d_g s$, and let $z \mapsto g(z)$ be a smooth curve in $s^{-1}(x)$ such that $g(0) = g$, and $\dot g (0) = X$. Define $\nabla^{\Gg}_X \in D_g t^\ast E$ as follows. For $s \in \Gamma (t^\ast E)$ put
\begin{equation}\label{eq:nabla_G}
\nabla^{\Gg}_X s := \left. \frac{d}{dz} \right|_{z=0} g \cdot g(z)^{-1} \cdot s(g (z)).
\end{equation}
It is straightforward to check that $\nabla^{\Gg}$ is a well-defined $\ker ds$-connection in $t^\ast E$. Moreover $\nabla^{\Gg}_{[X,Y]} = [\nabla^{\Gg}_X, \nabla^{\Gg}_Y]$ for all right invariant vector fields $X,Y$. Since right invariant vector fields generate $\Gamma (\ker ds)$, it follows that $\nabla^{\Gg}$ is a flat connection. Finally, Equation (\ref{eq:nabla}) immediately follows from (\ref{eq:nabla_G}) and (\ref{eq:nabla_0}).
\end{proof}

A $t^\ast E$-valued differential form $\omega$ on $\Gg$ is \emph{multiplicative} (see \cite{CSS2012}) if 
\[
(m^\ast \theta)_{(g_1, g_2)} = \pr_1^\ast \theta_{g_1} + g_1 \cdot (\pr_2^\ast \theta_{g_2}),
\]
for all $(g_1, g_2) \in \Gg_2$. Various geometric structures on Lie groupoids are encoded by multiplicative forms (with generically non-trivial coefficients). Standard examples are provided by presymplectic groupoids, contact groupoids and multiplicative foliations of groupoids. Crainic, Salazar and Struchiner realized that the infinitesimal counterparts of vector bundle valued multiplicative forms are what they call \emph{Spencer operators} \cite{CSS2012}.

Let $A \to M$ be a Lie algebroid, with anchor $\rho : A \to TM$, and Lie bracket $[-,-] : \Gamma (A) \times \Gamma (A) \to \Gamma (A)$. Moreover, let $E \to M$ be a vector bundle carrying a representation $\nabla$ of $A$. Crainic, Salazar, and Struchiner define $E$-valued $k$-Spencer operators on $A$, for all positive integers $k$. I'm only interested in the case $k = 1$. 

\begin{definition}
An $E$-valued $1$-Spencer operator on $A$ is a pair $(\Dd, l)$, where 
\[
\Dd : \Gamma (A) \longrightarrow \Omega^1 (M, E)
\]
is a first order differential operator, and 
\[
l : A \longrightarrow E
\]
is a vector bundle morphism such that
\begin{equation}\label{eq:Spenc1}
\Dd(f \alpha)  = f \Dd(\alpha) + df \otimes l (\alpha),
\end{equation}
and, moreover,
\begin{align}
\Dd([\alpha, \beta]) & = \Ll_{\nabla_\alpha} \Dd(\beta) - \Ll_{\nabla_\beta} \Dd(\alpha) \label{eq:Spenc2} \\
l ([\alpha, \beta]) & = \nabla_\alpha l (\beta) - i_{\rho (\beta)} \Dd(\alpha) \label{eq:Spenc3}
\end{align}
for all $\alpha, \beta \in \Gamma (A)$, and $f \in C^\infty (M)$.
\end{definition}

Formula (\ref{eq:Spenc2}) contains the \emph{Lie derivative of an $E$-valued form $\omega$ on $M$ along a derivation $\Delta \in \Der E$} (in the case of Formula (\ref{eq:Spenc2}), $\Delta = \nabla_\alpha$ for some $\alpha \in \Gamma (A)$) which can be defined by
\[
\begin{aligned}
&\quad\ (\Ll_\Delta \omega)(X_1, \ldots, X_k) \\
& = \Delta (\omega (X_1, \ldots, X_k)) - \sum_{i=1}^k \omega (X_1, \ldots, [\sigma (\Delta), X_i], \ldots, X_k),
\end{aligned}
\]
for $\omega \in \Omega^k (M, E)$, $X_1, \ldots, X_k \in \mathfrak X (M)$. In particular, if the flow of $\Delta$ is $\{ \Phi_z \}$, and the flow of the symbol $\sigma(\Delta)$ is $\{ \underline\Phi{}_z \}$, then 
\begin{equation}\label{eq:Lie_Delta}
\Ll_\Delta \omega = \left. \frac{d}{dz} \right|_{z=0} \Phi_z^\ast \omega,
\end{equation}
where 
\[
(\Phi_z^\ast \omega)_x (X_1, \ldots, X_k) = \Phi_z|_{E_x}^{-1} \left( \omega_{\underline \Phi_z (x)} ((\underline \Phi{}_z)_\ast X_1, \ldots, (\underline \Phi{}_z)_\ast X_k) \right),
\]
for all $x \in M$, and $X_1, \ldots, X_k \in T_x M$.

\begin{example}\label{ex:class_Spenc}
Let $E \to M$ be a vector bundle. Recall that $E$-valued one forms on $M$, $\Omega^1 (M, E)$, embed into sections of the first jet bundle $J^1 E$, via the $C^\infty (M)$-linear map
\[
\gamma : \Omega^1 (M, E) \to \Gamma (J^1 E), \quad df \otimes \varepsilon \mapsto j^1 f\varepsilon - f j^1 \varepsilon,
\]
$f \in C^\infty (M)$, $\varepsilon \in \Gamma (E)$, and $\gamma$ fits in an exact sequence
\begin{equation}\label{eq:SESx}
0 \longrightarrow \Omega^1 (M, E) \overset{\gamma}{\longrightarrow} \Gamma(J^1 E) \overset{\pr_E}{\longrightarrow} \Gamma (E) \longrightarrow 0
\end{equation}
Additionally, (\ref{eq:SESx}) splits via the first order differential operator
\[
\Dd^{\mathrm{class}} : \Gamma (J^1 E) \to \Omega^1 (M,E)
\]
 well-defined by 
\[
\Dd^{\mathrm{class}} \left( \sum f j^1 \lambda \right) = \sum df \otimes \lambda 
\]
(see also \cite[Example 2.8]{CSS2012}). More precisely, $\Dd^{\mathrm{class}} \circ \gamma = - \operatorname{id}$. Notice that 
\begin{equation}\label{eq:class_Spenc}
\Dd^{\mathrm{class}}(f \psi) = f \Dd^{\mathrm{class}}(\psi) + df \otimes \pr_E (\psi),
\end{equation}
for all $f \in C^\infty (M)$ and $\psi \in \Gamma (J^1 E)$. Hence $(\Dd^{\mathrm{class}}, \pr_L)$ is an $E$-valued $1$-Spencer operator on the Lie algebroid $J^1 E$ with trivial bracket and anchor. Following Crainic, Salazar and Struchiner, we call $\Dd^{\mathrm{class}}$ the \emph{classical Spencer operator}.
\end{example}

\begin{theorem}[Crainic, Salazar, and Struchiner {\cite{CSS2012}}]\label{theor:CSS}
Let $E \to M$ be a vector bundle carrying a representation of a Lie groupoid $\mathcal G \rightrightarrows M$, with source $s$, target $t$, and unit $u$, and let $A$ be the Lie algebroid of $\mathcal G$. Then any multiplicative form $\theta \in \Omega^1 (\mathcal G, t^\ast E)$ induces an $E$-valued $1$-Spencer operator $(\Dd, l)$ on $A$, given by 
\begin{equation} \label{eq:D}
\begin{aligned}
\Dd(\alpha) & =  u^\ast (\Ll_{\nabla^\Gg_{\alpha^r}} \theta)\\
l (\alpha) & = u^\ast (i_{\alpha^r} \theta)
\end{aligned}.
\end{equation}
If, additionally, $\mathcal G$ is $s$-simply connected, then the above construction establishes a one-to-one correspondence between multiplicative $E$-valued $1$-forms on $\mathcal G$ and $E$-valued $1$-Spencer operators on $A$.
\end{theorem} 

\begin{remark}
It immediately follows from (\ref{eq:nabla_G}) and (\ref{eq:Lie_Delta}) that Formulas (2.8) in \cite{CSS2012} are indeed equivalent to the more compact (\ref{eq:D}).
\end{remark}

My next aim is to show that a Dirac-Jacobi bundle is the same as a Lie algebroid equipped with a $1$-Spencer operator of a certain kind. It will then follow from Theorem \ref{theor:CSS} that Dirac-Jacobi bundles integrate, if at all integrable, to (non-necessarily coorientable) precontact groupoids (see below). 

It is useful to revise slightly the definition of $1$-Spencer operator, at least in the case when it takes values in a line bundle. Thus, let $A \to M$ be a Lie algebroid.

\begin{proposition}
The following two sets of data are equivalent (see also \cite[Remark 2.9]{CSS2012}):
\begin{enumerate}
\item A line bundle $L \to M$ carrying a representation of $A$, and an $L$-valued $1$-Spencer operator $(\Dd, l)$ on $A$,
\item a vector bundle morphism $(\nabla, \mathscr D) : A \to DL \oplus J^1 L = \D L$ such that
\begin{equation}\label{eq:anchor_DD}
\sigma \circ \nabla = \rho
\end{equation}
and
\begin{equation} \label{eq:bracket_DD}
(\nabla, \mathscr D) ([\alpha, \beta]) = \blq (\nabla, \mathscr D) (\alpha), (\nabla, \mathscr D) (\beta) \brq.
\end{equation}
\end{enumerate}
Moreover, $\operatorname{im} (\nabla, \mathscr D)$ is automatically isotropic for every $(\nabla, \mathscr D)$ as above.
\end{proposition}
\begin{proof}
First let $L \to M$ be a line bundle carrying a representation $\nabla$ of $A$, and let $(\Dd,l)$ be an $L$-valued $1$-Spencer operator on $A$. It follows from (\ref{eq:Spenc1}) that the map $\mathscr D : \Gamma (A) \to \Gamma (J^1 L)$, defined by $ \mathscr D (\alpha) := \gamma (\Dd (\alpha)) - j^1l (\alpha)$, is $C^\infty (M)$-linear. Hence it comes from a vector bundle morphism, also denoted by $\mathscr D : A \to J^1 L$. Consider $(\nabla , \mathscr D) : A \to D L \oplus J^1 L = \D L$. Equation (\ref{eq:anchor_DD}) is automatically satisfied. Now, using (\ref{eq:Spenc2}), (\ref{eq:Spenc3}) and the fact that $d_{D} \circ j^1 = d_{D} \circ d_{D} = 0$, we get
\[
\begin{aligned}
\mathscr D ([\alpha, \beta]) & = \gamma (\Dd ([\alpha, \beta]) - j^1 l ([\alpha, \beta]) \\
                              & = \gamma (\Ll_{\nabla_\alpha} \Dd (\beta) - \Ll_{\nabla_\beta} \Dd (\alpha) ) - j^1 (\nabla_\alpha l (\beta) - i_{\rho (\beta)} \Dd (\alpha)) \\
                              & = \Ll_{\nabla_\alpha} (\gamma (\Dd (\beta)) - j^1 l (\beta) ) - \Ll_{\nabla_\beta} \gamma (\Dd (\alpha)) + d_{D} i_{\rho (\beta)} \Dd (\alpha) \\
                              & = \Ll_{\nabla_\alpha} \mathscr D (\beta) - \Ll_{\nabla_\beta} \gamma (\Dd (\alpha)) + d_{D} i_{\nabla_\beta} \gamma (\Dd (\alpha)) \\
                              & = \Ll_{\nabla_\alpha} \mathscr D (\beta) - i_{\nabla_\beta} d_{D} \gamma (\Dd (\alpha)) \\
                              & = \Ll_{\nabla_\alpha} \mathscr D (\beta) - i_{\nabla_\beta} d_{D} \mathscr D (\alpha),
\end{aligned}
\]
for all $\alpha, \beta \in \Gamma (A)$.
Hence
\[
\begin{aligned}
(\nabla_{[\alpha, \beta]}, \mathscr D ([\alpha, \beta]))
                              & = ([\nabla_\alpha, \nabla_\beta], \Ll_{\nabla_\alpha} \mathscr D (\beta) - i_{\nabla_\beta} d_{D} \mathscr D (\alpha)) \\
                              & = \blq (\nabla_\alpha, \mathscr D (\alpha)), (\nabla_\beta , \mathscr D (\beta))\brq \\
                              & = \blq (\nabla, \mathscr D) (\alpha), (\nabla, \mathscr D) (\beta) \brq.
\end{aligned}
\]
Conversely, let $(\nabla, \mathscr D) : A \to \D L$ be a vector bundle morphism such that (\ref{eq:anchor_DD}) and (\ref{eq:bracket_DD}) hold. Put, $\Dd := \Dd^{\mathrm{class}} \circ \mathscr D$, and $l := \pr_L \circ \mathscr D$, where $\pr_L :\linebreak J^1 L \to L$ is the natural projection and $\Dd^{\mathrm{class}} : \Gamma (J^1 L) \to \Omega^1 (M, L)$ is the classical Spencer operator defined in Example \ref{ex:class_Spenc}. This means that
\begin{equation}\label{eq:newx}
\mathscr D (\alpha) = j^1 l (\alpha) - \gamma (\Dd (\alpha)),
\end{equation}
for all $\alpha \in \Gamma (A)$. Then (\ref{eq:Spenc1}) follows from (\ref{eq:class_Spenc}), $\sigma \circ \nabla = \rho$ follows from (\ref{eq:anchor_DD}), and $\nabla_{[\alpha, \beta]} = [\nabla_\alpha, \nabla_\beta]$, (\ref{eq:Spenc2}) and (\ref{eq:Spenc3}) follow from (\ref{eq:bracket_DD}). Moreover, the construction of $(\nabla, \Dd, l)$ from $(\nabla, \mathscr D)$ clearly inverts the construction of $(\nabla, \mathscr D)$ from $(\nabla, \Dd, l)$.

Finally, compute
\[
\begin{aligned}
\bla (\nabla, \mathscr D) (\alpha), (\nabla, \mathscr D) (\beta) \bra & = \langle \nabla_\alpha, j^1 l (\beta) -\gamma (\mathcal D (\beta))   \rangle +  \langle \nabla_\beta, j^1 l (\alpha) -\gamma (\mathcal D (\alpha))   \rangle \\
                                                 & =  \nabla_\alpha l (\beta) - i_{\rho (\alpha)} \mathcal D (\beta)+ \nabla_\beta l (\alpha) - i_{\rho (\beta)} \mathcal D (\alpha)  \\
                                                 & = l ([\alpha, \beta]) + l ([\beta, \alpha]) \\
                                                 & = 0,
\end{aligned}
\]
where we used (\ref{eq:newx}) and (\ref{eq:Spenc3}).

\end{proof}

\begin{corollary}\label{cor}
Let $A \to M$ be a Lie algebroid with $\rk A = \dim M + 1$, let $L \to M$ be a line bundle carrying a representation of $A$, and let $(\mathcal D, l)$ be an $L$-valued $1$-Spencer operator on $A$ such that the associated vector bundle morphism $(\nabla, \mathscr D) : A \to \D L$ is an embedding, then $(L, \operatorname{im} (\nabla, \mathscr D))$ is a Dirac-Jacobi bundle. Conversely, let $(L, \mathfrak L)$ be a Dirac-Jacobi bundle over $M$, and let $i : \mathfrak L \INTO \D L$ be the inclusion, then $(\Dd := \Dd^{\mathrm{class}} \circ \pr_{J^1} \circ i , l := \pr_L \circ \pr_{J^1} \circ i)$ is an $L$-valued $1$-Spencer operator on the Lie algebroid $ \mathfrak L  \to M$.
\end{corollary}

In other words, a Dirac-Jacobi bundle is essentially the same as a rank $\dim M + 1$ Lie algebroid $ \mathfrak L  \to M$ equipped with a $1$-Spencer operator with values in a line bundle $L$ such that the associated vector bundle morphism $(\nabla, \mathscr D) : \mathfrak L \to \D L$ is injective.

\begin{proposition}\label{prop:above}
Let $\mathcal G \rightrightarrows M$ be a Lie groupoid, with source $s$, target $t$, and unit $u$, $A$ the Lie algebroid of $\mathcal G$, and let $L \to M$ be a line bundle carrying a representation of $\mathcal G$. Moreover, let $\theta \in \Omega^1 (\mathcal G, t^\ast L)$ be a multiplicative form, $(\Dd , l)$ the induced $L$-valued $1$-Spencer operator on $A$, and let $(\nabla, \mathscr D) : A \to \D L$ be the associated vector bundle morphism. Then $(\nabla, \mathscr D)$ is injective if and only if 
\begin{equation}\label{eq:non_deg}
(\nabla^\Gg)^{-1}\left(\left. \ker d_{D} t \cap K_\omega\right|_M \right) = \mathbf{0},
\end{equation}
where $K_\omega = \ker \omega$, and $\omega$ is the $2$-cocycle in $(\Omega^\bullet_{t^\ast L}, d_{D})$ corresponding to $\theta$ via Proposition \ref{prop:prec}.
\end{proposition}

\begin{proof}
It suffices to show that 
\[
\ker \nabla \cap \ker \mathscr D = (\nabla^\Gg)^{-1}\left(\left. \ker d_{D} t \cap K_\omega \right|_M \right).
\]
We concentrate on the right hand side, which is more explicitly given by
\[
(\nabla^\Gg)^{-1} \left(\left.\ker d_{D} t \cap K_\omega \right|_M  \right) = \{ \alpha \in \ker ds |_M : t_\ast (\nabla^\Gg_\alpha |_M) = \omega |_M (\nabla^\Gg_\alpha, -) = 0 \}.
\]
First of all, from (\ref{eq:nabla}), $\ker \nabla = (\nabla^\Gg)^{-1} (\ker d_{D} t |_M)$. We claim that $\ker \mathscr D = (\nabla^\Gg)^{-1} (K_\omega |_M) $. To see this, let $\alpha \in \Gamma (A)$, and $x \in M$. Denote $\widehat \alpha{}^r := \nabla^\Gg_{\alpha^r}$. So, from (\ref{eq:newx}) and (\ref{eq:D}),
\begin{align}
\mathscr D (\alpha_x) & = j^1_x l(\alpha) - \gamma (\Dd (\alpha)_x)  \nonumber \\
                                   & = j^1_x u^\ast (i_{\alpha^r} \theta) - \gamma (u^\ast (\Ll_{\widehat \alpha{}^r} \theta)_x) \nonumber \\
                                   & = u^\ast \left( j^1_x i_{\alpha^r} \theta - \gamma (\Ll_{\widehat \alpha{}^r} \theta)_x  \right). \label{eq:mathscr_D}
\end{align}
Now, put $\Theta = \gamma (\theta) = \sigma^\ast \theta = \theta \circ \sigma \in \Omega^1_{t^\ast L} = \Gamma (J^1 t^\ast L)$, and notice that $\gamma( \Ll_{\widehat \alpha{}^r} \theta) = \Ll_{\widehat \alpha{}^r} \Theta$. Hence, from (\ref{eq:mathscr_D}), $\alpha_x \in \ker \mathscr D$ iff
\[
\begin{aligned}
0 & = \langle \gamma (\Ll_{\widehat \alpha{}^r} \theta)_x  - j^1_x i_{\alpha^r} \theta, \Delta \rangle \\
& = i_\Delta \left( \Ll_{\widehat \alpha{}^r} \Theta  -  d_{D} i_{\widehat \alpha{}^r} \Theta \right) \\
& = i_\Delta \left( i_{\widehat \alpha{}^r} d_{D} \Theta \right), \\
& = i_\Delta u^\ast \left( i_{\widehat \alpha{}^r} d_{D} \Theta \right),
\end{aligned}
\]
for all $\Delta \in D_x L$. We conclude that $\alpha_x \in \ker \mathscr D$ if and only if $u^\ast \left( i_{\widehat \alpha{}^r} d_{D} \Theta \right)_x = 0$. Now, recall that $\theta$ is multiplicative. It follows that $\Theta$ is multiplicative as well, i.e.
\begin{equation}\label{eq:Theta_mult}
(m^\ast \Theta)_{(g_1, g_2)} = \pr_1^\ast \Theta_{g_1} + g_1 \cdot (\pr_2^\ast \Theta_{g_2}),
\end{equation}
for all $(g_1, g_2) \in \Gg_2$, where $m^\ast$ is the pull-back of elements in $\Omega_{t^\ast L}^\bullet$ along the regular line bundle morphism $m^\ast t^\ast L \to t^\ast L$ (similarly for $\pr_1^\ast, \pr_2^\ast$). Since $d_{D}$ commutes with pull-backs, $d_{D} \Theta$ is also multiplicative. So, From Lemma \ref{lem:below} below, $u^\ast \left( i_{\widehat \alpha{}^r} d_{D} \Theta \right)_x = \left( i_{\widehat \alpha{}^r} d_{D} \Theta \right)_x$. Hence
$\alpha_x \in \ker \mathscr D$ if and only if $i_{\widehat \alpha{}^r} \omega_x = 0$, i.e. $(\nabla^\Gg_{\alpha^r})_x \in K_\omega$. This concludes the proof.
\end{proof}

\begin{lemma}\label{lem:below}
Let $\Gg \rightrightarrows M$ and $L \to M$ be as in the statement of Proposition \ref{prop:above}. Moreover, let $\Theta \in \Omega_{t^\ast L}^\bullet$ be multiplicative, i.e.~$\Theta$ satisfies (\ref{eq:Theta_mult}). Then $i_{\widehat \alpha{}^r} \Theta = t^\ast u^\ast i_{\widehat \alpha{}^r} \Theta $, for all $\alpha \in \Gamma (A)$.
\end{lemma}
\begin{proof}
First of all, it is easy to check that, from right invariance of $\alpha^r$, right invariance of $\widehat \alpha{}^r = \nabla^\Gg_{\alpha^r}$ follows, i.e., for any $g : x \to y \in \Gg$, right translation $R_g : s^{-1}(y) \to s^{-1}(x)$ maps $\widehat \alpha{}^r|_{s^{-1}(y)}$ to $\widehat \alpha{}^r|_{s^{-1}(x)}$. In particular,
\[
\widehat\alpha{}^r_g = (R_g)_\ast \widehat\alpha{}^r_{t(g)},
\]
for all $g \in \Gg$. Now, let $\Theta \in \Omega^{k+1}_{t^\ast L}$. Take $\Delta_1, \ldots, \Delta_k \in D_g t^\ast L = D _{t(g)} L$. We can express $\widehat \alpha{}^r_g$ and $\Delta_i$ as (cf.~\cite[Proof of Lemma 4.2]{CSS2012})
\[
\widehat\alpha{}^r_g = (R_g)_\ast (\widehat\alpha{}^r_{t(g)}) = m_\ast (\widehat\alpha{}^r_{t(g)}, 0_g) \quad \text{and} \quad \Delta_i = m_\ast (t_\ast (\Delta_i), \Delta_i),
\]
$i= 1, \ldots, k$. Hence, from multiplicativity,
\[
\begin{aligned}
&\quad\ (i_{\widehat \alpha{}^r} \Theta)_g(\Delta_1, \ldots, \Delta_k) \\
                           & = \Theta_g (\widehat \alpha{}^r_g, \Delta_1, \ldots, \Delta_k) \\
                           & = \Theta_g (m_\ast (\widehat\alpha{}^r_{t(g)}, 0_g), m_\ast (t_\ast (\Delta_1), \Delta_1) , \ldots, m_\ast (t_\ast (\Delta_k), \Delta_k))  \\
                           & = (m^\ast \Theta)_{(t(g),g)} ((\widehat\alpha{}^r_{t(g)}, 0_g), (t_\ast (\Delta_1), \Delta_1), \ldots, (t_\ast (\Delta_k), \Delta_k)) \\
                           & = (\pr_1^\ast \Theta_{t(g)} + g \cdot (\pr_2^\ast \Theta_{g})) ((\widehat\alpha{}^r_{t(g)}, 0_g), (t_\ast (\Delta_1), \Delta_1), \ldots, (t_\ast (\Delta_k), \Delta_k)) \\
                           & = \Theta_{t(g)} (\widehat\alpha{}^r_{t(g)}, t_\ast (\Delta_1), \ldots, t_\ast (\Delta_k) ) \\
                           & = (u^\ast (i_{\widehat \alpha{}^r} \Theta))_{t(g)} (t_\ast (\Delta_1), \ldots, t_\ast (\Delta_k) ) \\
                           & = (t^\ast u^\ast (i_{\widehat \alpha{}^r} \Theta))_g (\Delta_1, \ldots, \Delta_k).
\end{aligned}
\]
\end{proof}

\begin{definition} \quad
\begin{enumerate}
\item A \emph{precontact groupoid} is a triple $(\Gg, L, \theta)$ where $\Gg \rightrightarrows M$ is a Lie\linebreak groupoid such that $\dim \Gg = 2 \dim M + 1$, $L \to M$ is a line bundle carrying a representation of $\Gg$, and $\theta \in \Omega^1 ( \Gg, t^\ast L)$ is a multiplicative $1$-form satisfying condition (\ref{eq:non_deg}).
\item A \emph{Dirac-Jacobi Lie algebroid} is a Lie algebroid $A \to M$ equipped with a Lie algebroid isomorphism $A \simeq \mathfrak L$ onto a Dirac-Jacobi structure $\mathfrak L$ on a line bundle.
\end{enumerate}
\end{definition}

Collecting the above results we get the following 

\begin{theorem}\label{theor:int}
Let $\mathcal G \rightrightarrows M$ be a Lie groupoid such that $\dim \Gg =\linebreak 2 \dim M + 1$, and let $A$ be the Lie algebroid of $\mathcal G$. If $\mathcal G$ is $s$-simply connected, construction in Theorem \ref{theor:CSS} establishes a one-to-one correspondence between precontact groupoid structures on $\Gg$ and Dirac-Jacobi Lie algebroid structures on $A$. 
\end{theorem}

\appendix

\section{Dirac-ization trick}\label{app:Dirac_trick}

Jacobi bundles are equivalent to \emph{homogeneous Poisson manifolds} of a special kind, namely principal $\R^\times$-bundles $\widetilde M \to M$ equipped with a Poisson structure $\Pi$ on the total space $\widetilde M$ such that $(\Pi, \mathcal E)$ is a homogeneous Poisson structure, i.e.~$\Ll_{\mathcal E} \Pi = - \Pi = h_{-1}^\ast \Pi$. Here, $h_{-1} : \widetilde M \to \widetilde M$ is the multiplication by $-1 \in \R^\times$, and $\Ee$ is the \emph{Euler vector field} on $\widetilde M$, i.e.~the fundamental vector field corresponding to the canonical generator $1$ in the Lie algebra $\R$ of the structure group $\R^\times$. Often, properties of Jacobi bundles can be proved using their interpretation as homogeneous Poisson structures. This is the so called \emph{Poissonization trick} exploited, for instance, by Crainic and Zhu in \cite{CZ2007} to study integrability of Jacobi brackets. At the same time, precontact manifolds can be understood as \emph{homogeneous presymplectic manifolds} (of a special kind), i.e.~principal $\R^\times$-bundles $\widetilde M \to M$ equipped with a presymplectic structure $\widetilde \omega$ on $\widetilde M$ such that $\Ll_\Ee \widetilde \omega = \widetilde \omega = - h_{-1}^\ast \omega$ (see, e.g., \cite{V2015b}). Actually, this last construction inspired the interpretation of precontact distributions presented in Section \ref{sec:cont_lcs_geom}.

Similarly, Dirac-Jacobi bundles can be understood as \emph{homogeneous Dirac manifolds}. This was already proved by Iglesias-Ponte and Marrero for $\mathcal E^1 (M)$-Dirac structures, i.e.~Dirac-Jacobi structures on trivial line bundles \cite[Section 5]{IM2002}. In this appendix, we consider the general case. Namely, we define homogeneous Dirac manifolds and prove their equivalence with Dirac-Jacobi line bundles. Let us start with the remark that a line bundle $L \to M$ can be understood as a principal $\R^\times$-bundle. Namely, take the \emph{slit} dual bundle $\widetilde M := L^\ast \smallsetminus \mathbf 0$, where $\mathbf 0$ is the (image of the) zero section of $L$. There is an obvious (fiber-wise) action $h : \R^\times \times \widetilde M \to \widetilde M$, $(t, \upsilon) \mapsto h_t \upsilon$, of $\R^\times$ on $\widetilde M$ which makes it a principal $\R^\times$-bundle over $M$. On the other hand, every principal $\R^\times$-bundle is of this form. Now, denote by $\pi : \widetilde M \to M$ the projection, and by $\Ee \in \mathfrak X (\widetilde M)$ the Euler-vector field on $\widetilde M$. Sections of $L$ identify with (degree one) \emph{homogeneous functions} on $\widetilde M$, i.e.~those functions $f$ such that $\Ee (f) = f = - h_{-1}^\ast f$. We denote by $\widetilde \lambda$ the homogeneous function corresponding to section $\lambda \in \Gamma (L)$.

\begin{proposition}\label{prop:homog} \quad
\begin{enumerate}
\item There is an embedding of $C^\infty (M)$-modules $\iota: \Der L \INTO \mathfrak X (\widetilde M)$, $\Delta \mapsto \widetilde \Delta$, with $\widetilde \Delta$ uniquely determined by $\widetilde \Delta (\widetilde \lambda) = \widetilde{\Delta (\lambda)}$. Moreover $\widetilde{[\Delta, \square]} = [\widetilde \Delta, \widetilde \square]$ for all $\Delta, \square \in \Der L$.
\item The image of $\iota$ consists of \emph{degree zero homogeneous vector fields}, i.e., vector fields $X$ on $\widetilde M$ such that $[\Ee, X] = 0$, and $h_{-1}^\ast X = X$.
\item Geometrically, there are canonical isomorphisms $\pi^\ast D L \!\simeq\! \pi^\ast J_1 L \!\simeq\! T \widetilde M$ of vector bundles over $\widetilde M$, and the embedding $\Gamma (D L) \!\INTO\! \Gamma (\pi^\ast D L) \!\simeq\! \Gamma (T \widetilde M)$ agrees with $\iota$.
\end{enumerate}
\end{proposition}

\begin{proof}
For point (1), it is enough to notice that a vector field on $\widetilde M$ is completely determined by its action on homogeneous functions. The same argument shows that every vector field in the image of $\iota$ is homogeneous of degree zero. To see that every degree zero homogeneous vector field $X$ is in the image of $\iota$, and complete the proof of point (2), notice that $X$ preserves homogeneous functions. For point (3), define an isomorphism $\pi^\ast J_1 L \simeq T \widetilde M$ by mapping $(\upsilon, F) \in \pi^\ast J_1 L$, with $\upsilon \in L_x^\ast \smallsetminus 0$, and $F \in (J_1)_x L$, $x \in M$, to the unique tangent vector $X \in T_\upsilon \widetilde M$ mapping a homogeneous function $\widetilde \lambda$ to $X (\widetilde \lambda) := F (\lambda)$, where $\lambda \in \Gamma (L)$. The last equality can be read from the right to the left to define the inverse isomorphism. Finally, define an isomorphism $\pi^\ast D L \simeq \pi^\ast J_1 L$ by mapping $(\upsilon, \Delta) \in \pi^\ast D L$, with $\upsilon$ as above, and $\Delta \in D_x L$, to $(\upsilon, \upsilon \circ \Delta) \in (J_1)_x L$. The last claim in the statement immediately follows from the definition of isomorphism $\pi^\ast D L \simeq T \widetilde M$.
\end{proof}

\begin{remark}
According to the proof of Proposition \ref{prop:homog}, a tangent vector $X \in T_\upsilon \widetilde M$ identifies with the point $(\upsilon, \Delta) \in \pi^\ast D L$, where $\Delta (\lambda) := X (\widetilde \lambda) \upsilon^\ast$, and $\upsilon^\ast \in L_x$ is the unique point such that $\langle \upsilon , \upsilon^\ast \rangle = 1$.
\end{remark}

\begin{proposition}\label{prop:homog*} \quad
\begin{enumerate}
\item There is an embedding of $C^\infty (M)$-modules $\iota^\vee: \Gamma (J^1 L) \INTO \Omega^1 (\widetilde M)$, $\psi \mapsto \widetilde \psi$, with $\widetilde \psi$ uniquely determined by $\widetilde \psi (\widetilde \Delta) = \widetilde{\langle \Delta , \psi \rangle }$. Moreover, $\widetilde{j^1 \lambda} = d \widetilde \lambda$, for all $\lambda \in \Gamma (L)$.
\item The image of $\iota^\vee$ consists of \emph{degree one homogeneous $1$-forms}, i.e., $1$-forms $\sigma$ on $\widetilde M$ such that $\Ll_\Ee \sigma = \sigma = - h_{-1}^\ast \sigma$.
\item Geometrically, there are canonical isomorphisms $\pi^\ast (D L)^\ast \simeq \pi^\ast J^1 L \simeq T^\ast \widetilde M$ of vector bundles over $\widetilde M$, and the embedding $\Gamma (J^1 L) \INTO \Gamma (\pi^\ast J^1 L) \simeq \Gamma (T^\ast \widetilde M)$ agrees with $\iota^\vee$.
\end{enumerate}
\end{proposition}

\begin{proof}
For point (1), it is enough to notice that a differential $1$-form on $\widetilde M$ is completely determined by its contraction with degree zero homogeneous vector fields. In particular, it follows from $\widetilde{j^1 \lambda} (\widetilde \Delta) = \widetilde{\langle \Delta , j^1 \lambda \rangle} = \widetilde{\Delta (\lambda)} = \widetilde {\Delta} (\widetilde \lambda) = (d\widetilde \lambda)(\widetilde \Delta)$ that $\widetilde{j^1 \lambda} = d \widetilde \lambda$. The same argument shows that every differential $1$-form in the image of $\iota^\vee$ is homogeneous of degree one. To see that every degree one homogeneous differential $1$-form $\sigma$ is in the image of $\iota^\vee$, and complete the proof of point (2), notice that $\sigma$ maps degree zero homogeneous vector fields to degree one homogeneous functions. Point (3) follows from Proposition \ref{prop:homog}.
\end{proof}

Put $\mathbb T \widetilde M := T\widetilde M \oplus T^\ast \widetilde M$.

\begin{theorem}\label{theor:homog}\quad
\begin{enumerate}
\item There is an embedding of $C^\infty (M)$-modules $\iota: \Gamma (\D L) \INTO \Gamma (\mathbb T \widetilde M)$,\linebreak $(\Delta, \psi) \mapsto (\widetilde \Delta, \widetilde \psi)$. Moreover $\iota$ intertwines the bi-linear pairing $\bla -, - \bra : \Gamma (\D L) \times \Gamma (\D L) \to \Gamma (L)$ and the symmetric bi-linear pairing $\bla -, - \bra_{\widetilde M} : \Gamma (\mathbb T \widetilde M) \times \Gamma (\mathbb T \widetilde M) \to C^\infty (\widetilde M)$ in the sense that
\[
\bla \iota (\alpha), \iota (\beta) \bra_{\widetilde M} = \widetilde{\bla \alpha, \beta \bra}, \quad \alpha, \beta \in \Gamma (\D L).
\]
Embedding $\iota$ does also intertwine bracket $\blq -, - \brq$ on $\Gamma (\D L)$ and Dorfman bracket on $\Gamma (\mathbb T \widetilde M)$.
\item The image of $\iota$ consists of pairs $(X, \sigma)$, where $X$ is a degree zero homogeneous vector field and $\sigma$ is a degree one homogeneous differential $1$-form.
\item Geometrically, there is a canonical isomorphism $\pi^\ast \D L \simeq  \mathbb T \widetilde M$ of vector bundles over $\widetilde M$. Embedding $\Gamma (\D L) \INTO \Gamma (\pi^\ast \D L) \simeq \Gamma (\mathbb T\widetilde M)$ agrees with~$\iota$.
\end{enumerate}
\end{theorem}

\begin{proof}
The statement immediately follows from Propositions \ref{prop:homog} and \ref{prop:homog*}. The only part deserving some more comments is that about the Dorfman bracket in point (1). The latter follows from the identity
\begin{equation}\label{eq:Lie}
\Ll_{\widetilde \Delta} \widetilde \psi = \widetilde{\Ll_{\Delta} \psi},
\end{equation}
for all $\Delta \in \Der L$, and $\psi \in \Gamma (J^1 L)$. In order to prove Equation (\ref{eq:Lie}), it suffices to notice that the embeddings $\Gamma (L) \INTO C^\infty (\widetilde M)$, and $\Gamma (J^1 L) \INTO \Omega^1 (\widetilde M)$, uniquely extend to an injective cochain map $(\Omega^\bullet_L, d_{D}) \INTO (\Omega^\bullet (\widetilde M), d)$, $\omega \mapsto \widetilde \omega$ such that
\[
\widetilde \omega (\widetilde \Delta_1, \ldots, \widetilde \Delta_k) = \widetilde{\omega (\Delta_1, \ldots, \Delta_k)},
\]
for all $\Delta_1, \ldots, \Delta_k \in \Der L$, and $\omega \in \Omega_L^k$.
\end{proof}

Now, let $L \to M$ and $\widetilde M$ be as above. Denote again by $\pi : \widetilde M \to M$ the projection.

\begin{proposition}
Let $(L, \mathfrak L)$ be a Dirac-Jacobi bundle over $M$. The subbundle $\pi^\ast \mathfrak L \subset \pi^\ast \D L \simeq \mathbb T \widetilde M$ is a Dirac structure on $\widetilde M$.
\end{proposition}

\begin{proof}
It immediately follows from Theorem \ref{theor:homog} and dimension counting.
\end{proof}

\begin{definition}
The Dirac structure $\pi^\ast \mathfrak L$ is called the \emph{Dirac-ization} of the Dirac-Jacobi structure $ \mathfrak L $. A Dirac structure on $\widetilde M$ is \emph{homogeneous} if it is the Dirac-ization of some Dirac-Jacobi structure $ \mathfrak L $ on $L$.
\end{definition}

\begin{remark}
When $(L, \mathfrak L)$ is the Dirac-Jacobi bundle corresponding to a Jacobi bundle $(L, \{-,-\})$ (resp.~a precontact distribution $C$) on $M$, then its Dirac-ization $\widetilde{\mathfrak L} = \pi^\ast \mathfrak L$ corresponds to the Poissonization of $(L, \{-,-\})$ (resp.~the presymplectization of $C$).
\end{remark}

Let $(L, \mathfrak L)$ be a Dirac-Jacobi bundle over $M$, and let $\widetilde{\mathfrak L}$ be its Dirac-ization. We conclude this section discussing the relationship between the characteristic foliation $\Ff_{\mathfrak L} $ of $\mathfrak L$ and the characteristic (presymplectic) foliation $\Ff_{\widetilde{\mathfrak L}}$ of $\widetilde{\mathfrak L}$. First of all, notice that the diagram
\[
\xymatrix@C=0pt@R=5pt{ & & T \widetilde M \ar[ddrrrrrr]|!{[rrrdd];[rrrrr]}\hole ^-{d\tau} \ar[rrrrr] \ar[ddddrr]|!{[lldd];[rrrdd]}\hole  & & & & & D L \ar[rdd]^-{\sigma} & & \\
 & & & & & & & & & \\
 \mathbb T \widetilde M \ar[uurr]^-{\pr_{T}} \ar[rrrrr] \ar[rrdddd] & & & & & \D L \ar[uurr] \ar[rrdddd] & & & TM \ar[rdd]^-{\tau}
  & \\
 & & & & & & & & & \\
 & & & & \widetilde M \ar[rrrrr]|!{[uur];[ddrrr]}\hole 
 & & & & & M \\
 & & & & & & & & & \\
 & & T^\ast \widetilde M \ar[rrrrr] \ar[uurr] & & & & & J^1 L \ar[uurr] & &
}
\]
commutes. It immediately follows that $\pr_T \widetilde{\mathfrak L} = \pi^\ast \pr_{D} \mathfrak L$. In particular, $T \Ff_{\mathfrak L}  = (d \tau )( T \Ff_{\widetilde{\mathfrak L}})$, hence leaves of $\Ff_{\widetilde{\mathfrak L}}$ project onto leaves of $\Ff_{\mathfrak L} $ under $\pi$. More precisely, $\pi$ establishes a surjection between leaves of $\Ff_{\widetilde{\mathfrak L}}$ and leaves of $\Ff_{\mathfrak L} $. If $\Oo$ is a leaf of $\Ff_{\mathfrak L} $, and $\widetilde{\Oo}$ is a leaf of $\Ff_{\widetilde{\mathfrak L}}$ projecting onto $\Oo$, then the tangent bundle $T \widetilde{\Oo}$ is $\pi^\ast (\Ii_{\mathfrak L}|_\Oo) \subset \pi^\ast D L \simeq T \widetilde M$. Now, recall that, if $\Oo$ is lcps, then $\Ii_{\mathfrak L}|_\Oo$ is actually the image of a flat connection $\nabla^\Oo$ in $L|_\Oo$. Consider the dual connection, also denoted by $\nabla^\Oo$, in $L|_\Oo^\ast$ and its restriction to $\widetilde M |_\Oo = \pi^{-1} (\Oo) = L|_\Oo^\ast \smallsetminus \mathbf 0$. The above discussion shows that $\widetilde{\Oo}$ is (locally) the image of a $\nabla^\Oo$-constant section of $\pi : \widetilde M |_\Oo \to \Oo$. On the other hand, if $\Oo$ is precontact, then $\Ii_{\mathfrak L}|_\Oo = D (L|_\Oo)$ and $\widetilde{\Oo}$ coincides with $\widetilde M |_\Oo$. In particular, 
\begin{equation}\label{eq:dim_Oo_tilde}
\dim \widetilde{\Oo} =  
\begin{cases}
\dim \Oo & \text{if $\Oo$ is lcps} \\
\dim \Oo +1 & \text{if $\Oo$ is precontact}.
\end{cases}
\end{equation}

\begin{remark}\label{rem:opposite}
Identity (\ref{eq:dim_Oo_tilde}) provides an alternative way to prove Corollary \ref{cor:parity}. Namely, it is known that the parity of the dimensions of characteristic leaves of a Dirac manifold is locally constant (see, e.g., \cite[Corollary 3.3]{DW2008}). Together with (\ref{eq:dim_Oo_tilde}) this immediately implies Corollary \ref{cor:parity}. Additionally, it follows from (\ref{eq:dim_Oo_tilde}) that, locally, the dimension of a lcps leaf and the dimension of a precontact leaf have different parities.
\end{remark}

We also have $\widetilde{\mathfrak L} \cap T \widetilde M = \pi^\ast (\mathfrak L  \cap D L)$. This shows that the null distribution $K_{\widetilde{\mathfrak L}}$ of $\widetilde{\mathfrak L}$ projects onto $K_{\mathfrak L} $ point-wise isomorphically. This can be also seen as follows. Let $\Oo$ be a leaf of $\Ff_{\mathfrak L} $ and let $\widetilde{\Oo}$ be a leaf of $\Ff_{\widetilde{\mathfrak L}}$ over $\Oo$. Moreover, let $\omega_{\widetilde \Oo} \in \Omega^2 (\widetilde \Oo)$ be the presymplectic structure on $\widetilde{\Oo}$ induced by $\widetilde{\mathfrak L}$. Clearly, $\omega_{\widetilde \Oo}$ is completely determined by its values on vector fields of the form $\widetilde \Delta$, with $\Delta$ a section of $\Ii_{\mathfrak L}|_\Oo$. Hence, if $\Oo$ is lcps, $\omega_{\widetilde{\Oo}}$ is completely determined by
\[
\omega_{\widetilde{\Oo}} (\widetilde{\nabla^\Oo_X}, \widetilde{\nabla^\Oo_Y}) = \langle \widetilde{\nabla^\Oo_Y}, \widetilde \psi \rangle = \widetilde{\langle \nabla^\Oo_Y , \psi \rangle } = \widetilde{\omega_{\Oo} (X, Y)},
\] 
where $X,Y \in \mathfrak X (\Oo)$, and $\psi \in \Gamma ((J^1 L)|_\Oo)$ is such that $(\nabla^\Oo_X, \psi)$ is a section of $ \mathfrak L |_\Oo$. If $\Oo$ is precontact, $\omega_{\widetilde{\Oo}}$ is completely determined by
\[
\omega_{\widetilde{\Oo}} (\widetilde \Delta, \widetilde \square) = \langle \widetilde{\square}, \widetilde \psi \rangle = \widetilde{\langle \square , \psi \rangle } = \widetilde{\omega_{\Oo} (\Delta, \square)},
\]
where $\Delta, \square \in \Der L|_{\Oo}$, and $\psi \in \Gamma ((J^1 L)|_\Oo)$ is such that $(\Delta, \psi)$ is a section of $ \mathfrak L |_\Oo$. In any case, 
\[
\omega_{\widetilde{\Oo}} = \widetilde{\omega_{\Oo}} |_{\widetilde{\Oo}}.
\]


\subsection*{Acknowledgements}
I thank the anonymous referee for suggesting me several revisions that improved a lot the presentation. I also thank her/him for suggesting me to use the local structure Theorem \ref{theor:local} to investigate transverse structures to characteristic leaves. 

\bigskip

\end{document}